\documentclass[10.5pt, a4paper]{article}
\usepackage{fullpage}
\usepackage{listings}
\usepackage{latexsym, amsfonts, amssymb, amsmath, amsthm, mathrsfs, mathtools, setspace, graphics, graphicx, bbm, float, bigints}
\usepackage{enumerate}
\usepackage{caption}
\usepackage{indentfirst}
\usepackage{framed}
\usepackage{yhmath}
\usepackage[nodayofweek,level]{datetime}
\usepackage{xcolor}
\usepackage{appendix}

\RequirePackage{tikz}
\usetikzlibrary{calc,trees,positioning,arrows,chains,shapes.geometric,%
    decorations.pathreplacing,decorations.pathmorphing,shapes,%
    matrix,shapes.symbols}
\allowdisplaybreaks    

\usepackage{geometry}
\title{Functional convex order for the scaled McKean-Vlasov processes}

\usepackage{hyperref}
\newcommand{\footremember}[2]{
   \footnote{#2}
    \newcounter{#1}
    \setcounter{#1}{\value{footnote}}
}

\author{%
 Yating Liu\footremember{a}{\small CEREMADE, CNRS, UMR 7534, Université Paris-Dauphine,
PSL University, 75016 Paris, France, \texttt{liu@ceremade.dauphine.fr}.}%
  \and Gilles Pag{\`e}s\footremember{b}{\small Sorbonne Université, CNRS, Laboratoire de Probabilités, Statistiques et Modélisations (LPSM),
75252 Paris, France, \texttt{gilles.pages@sorbonne-universite.fr}.}%
  }

\begin{document}
\maketitle

\numberwithin{equation}{section}
\newtheorem{thm}{Theorem}
\newtheorem{lem}{Lemma}[section]
\newtheorem{prop}{Proposition}[section]
\newtheorem{cor}{Corollary}[section]
\newtheorem{defn}{Definition}[section]
\theoremstyle{remark}
\newtheorem{rem}{Remark}[section]
\newtheorem{exple}{Example}[section]
\newtheorem{manualtheoreminner}{Assumption}
\newenvironment{manualtheorem}[1]{%
  \renewcommand\themanualtheoreminner{#1}%
  \manualtheoreminner
}{\endmanualtheoreminner}

%\endlocaldefs

\newcommand{\comp}[1]{{#1}^{\mathsf{c}}} %pour le complémentaire avec un joli "c"

\newcommand{\widesim}[2][1.5]{
  \mathrel{\overset{#2}{\scalebox{#1}[1]{$\sim$}}}
}

\newcommand{\RRD}{\mathbb{R}^{d}}
\newcommand{\vertiii}[1]{{\left\vert\kern-0.25ex\left\vert\kern-0.25ex\left\vert #1 
    \right\vert\kern-0.25ex\right\vert\kern-0.25ex\right\vert}}
\newcommand{\vertii}[1]{\left\Vert #1\right\Vert}
\newcommand{\Proj}{\mathrm{Proj}}

\newcommand{\tabincell}[2]{\begin{tabular}{@{}#1@{}}#2\end{tabular}}
\newcommand*\circled[1]{\tikz[baseline=(char.base)]{
            \node[shape=circle,draw,inner sep=2pt] (char) {#1};}}

\newcommand\independent{\protect\mathpalette{\protect\independenT}{\perp}}
\def\independenT#1#2{\mathrel{\rlap{$#1#2$}\mkern2mu{#1#2}}}

\newcommand{\PPC}{\mathcal{P}_{p}\big(\mathcal{C}([0, T], \mathbb{R}^{d})\big)}
\newcommand{\CPP}{\mathcal{C}\big([0, T], \mathcal{P}_{p}(\mathbb{R}^{d})\big)}
\newcommand{\CRD}{\mathcal{C}([0, T], \mathbb{R}^{d})}
\newcommand{\PPRD}{\mathcal{P}_{p}(\mathbb{R}^{d})}

\newcommand{\PP}{\mathbb{P}}
\newcommand{\RD}{\mathbb{R}^{d}}
\newcommand{\RR}{\mathbb{R}}
\newcommand{\PRD}{\mathcal{P}(\mathbb{R}^{d})}
\newcommand{\MDQ}{\mathbb{M}_{d \times q}}
\newcommand{\EE}{\mathbb{E}\,}
\newcommand{\conright}{\preceq_{\,cv}}
\newcommand{\conleft}{\succeq{\,cv}}
\newcommand{\coefax}{\alpha(t)}
\newcommand{\coefay}{\alpha(t)}

\newcommand{\coefbx}{\beta(t, \mathbb{E}X_t)}
\newcommand{\coefby}{\beta(t, \mathbb{E}Y_t)}

\newcommand{\driftx}{b(t, X_t, \mu_t)}
\newcommand{\drifty}{b(t, \,Y_t\,, \nu_t\,)}

\newcommand{\ve}{\varepsilon}

\begin{abstract}
We establish the functional convex order results for two scaled McKean-Vlasov processes $X=(X_{t})_{t\in[0, T]}$ and $Y=(Y_{t})_{t\in[0, T]}$ defined on a filtered probability space $(\Omega, \mathcal{F}, (\mathcal{F}_{t})_{t\geq0}, \mathbb{P})$ by 

\[\begin{cases}
dX_{t}= b(t, X_{t}, \mu_{t})dt+\sigma(t, X_{t}, \mu_{t})dB_{t}, \;\;X_{0}\in L^{p}(\mathbb{P}),\\
dY_{t}\,= b(t, \,Y_{t}\,,\, \nu_{t})dt+\theta(t, \,Y_{t}\,,\, \nu_{t})dB_{t}, \;\;Y_{0}\in L^{p}(\mathbb{P}),\\
%\text{where }\forall\, t\in[0, T],\, \mu_{t}=\mathbb{P}\circ X_{t}^{-1},  \,\nu_{t}=\mathbb{P}\circ Y_{t}^{-1}, \\
%X_{0}, Y_{0}\in L^{p}(\mathbb{P}) \,\text{ with }\, p\geq 2.
\end{cases}\]
 where $p\geq2$, for every $ t\in[0, T]$, $\mu_t$, $\nu_t$ denote the probability distribution of $X_t$, $Y_t$  respectively and the drift coefficient $b(t, x, \mu)$ is affine in $x$ (scaled). 
%\[\begin{cases}
%dX_{t}=\big(\alpha(t) X_{t}+\beta(t)\EE X_{t}+\gamma(t)\big)\,dt+\sigma(t, X_{t}, \mu_{t})dB_{t}, \\
%dY_{t}\,=\big(\alpha(t) \,Y_{t}\,+\beta(t)\EE \,Y_{t}\,+\gamma(t)\big)\,dt+\theta(t, \,Y_{t}\,,\, \nu_{t})dB_{t}, \\
%\text{where }\forall\, t\in[0, T],\, \mu_{t}=\mathbb{P}\circ X_{t}^{-1},  \,\nu_{t}=\mathbb{P}\circ Y_{t}^{-1}, \\
%X_{0}, Y_{0}\in L^{p}(\mathbb{P}) \,\text{ with }\, p\geq 2.
%\end{cases}\]
%with $p\geq 2$.
If we make the convexity and monotony assumption (only) on $\sigma$ and if $\sigma\preceq\theta$ with respect to the partial matrix order, the convex order for the initial random variable $X_0 \conright Y_0$ can be propagated to the whole path of process $X$ and $Y$. That is, if we consider a convex functional $F$ defined on the path space with polynomial growth, we have $\mathbb{E}F(X)\leq\mathbb{E}F(Y)$; for a convex functional $G$ defined on the product space involving the path space and its marginal distribution space, we have $\mathbb{E}\,G\big(X, (\mu_t)_{t\in[0, T]}\big)\leq \mathbb{E}\,G\big(Y, (\nu_t)_{t\in[0, T]}\big)$ under appropriate conditions. The symmetric setting is also valid, that is, if $\theta \preceq \sigma$ and $Y_0 \leq X_0$ with respect to the convex order, then $\mathbb{E}\,F(Y) \leq \mathbb{E}\,F(X)$ and $\mathbb{E}\,G\big(Y, (\nu_t)_{t\in[0, T]}\big)\leq \mathbb{E}\,G(X, (\mu_t)_{t\in[0, T]})$. The proof is based on several forward and backward dynamic programming principles and the convergence of the Euler scheme of the McKean-Vlasov equation.

\end{abstract}


\section{Introduction}

%\[\frac{\delta \Phi}{\delta m}( \delta_{\lambda x +(1-\lambda)y} , \lambda x +(1-\lambda)y ) \le \lambda   \frac{\delta \Phi}{\delta m}(  \delta_{\lambda x +(1-\lambda)y}, x ) + (1-\lambda)\frac{\delta \Phi}{\delta m}(  \delta_{\lambda x +(1-\lambda)y}, y) \]

Let $U, V: (\Omega, \mathcal{F}, \mathbb{P})\rightarrow \big(\mathbb{R}^{d}, \mathcal{B}(\mathbb{R}^{d})\big)$ be two {\color{black}integrable} random variables. We say that $U$ is dominated by $V$ for the convex order -  denoted by $U\preceq_{\,cv}V$ - if for any convex function $\varphi: \mathbb{R}^{d}\rightarrow\mathbb{R}$, such that $\mathbb{E} \,\varphi(U)$ and $\mathbb{E} \,\varphi(U)$ are well defined in $(-\infty, +\infty]$,
\begin{equation}\label{defconv}
\mathbb{E} \,\varphi(U)\leq \mathbb{E} \,\varphi (V).\end{equation}
Note  that if $U$ is integrable, then $\mathbb{E} \,\varphi(U)$ is always well-defined in $(-\infty, +\infty]$ by considering $\varphi^{\pm}(U)$ with $\varphi^{\pm}(x)\coloneqq\max(\pm\varphi(x), 0)$ since $\varphi^{-}$ is upper bounded by an affine function. For $p\in[1,+\infty)$, let $\mathcal{P}_{p}(\mathbb{R}^{d})$ denote the set of probability distributions on $\mathbb{R}^{d}$ with $p$-th finite moment. 
%Let $\mathcal{P}(\RD)$ denote the set of all probability distributions on $\big(\RD, \mathcal{B}(\RD)\big)$. 

Hence, the above definition of the convex order has the obvious equivalent version for two probability distributions $\mu, \nu\in\mathcal{P}_{1}(\mathbb{R}^{d})$: we say that the distribution $\mu$ is dominated by $\nu$ for the convex order -  denoted by $\mu\preceq_{\,cv}\nu$ - if, for every convex function $\varphi:\mathbb{R}^{d}\rightarrow\mathbb{R}$, $\int_{\mathbb{R}^{d}} \varphi(\xi)\mu(d\xi)\leq \int_{\mathbb{R}^{d}}\varphi(\xi)\nu(d\xi).$ 
%as soon as these two integrals make sense in $\overline{\mathbb{R}}$.

Also note that, as $U$ and $V$ have a finite first moment, then %$U\preceq_{\,cv} V $ implies that 
\begin{equation}\label{equalE}
U\preceq_{\,cv} V \quad\Longrightarrow \quad\EE\, U= \EE\, V
\end{equation} by simply considering the two linear functions $\varphi(x) = \pm x$. In fact,  
the connection 
%the convex order relation
between the  distributions of $U$ and $V$, say $\mu$ and $\nu$,  is much stronger than this necessary condition or the elementary domination inequality $\text{var}(U)\leq \text{var}(V)$ when $U,V\in L^{2}(\mathbb{P})$. Indeed, a special case of Kellerer's  theorem (\cite{KEL, HiPr}) shows that $\mu\preceq_{cv}\nu$ if and only if there exists a probability space $(\tilde \Omega,  \tilde{\cal A}, \tilde\PP)$  and a couple $(\tilde U, \tilde V)$ such that $U\sim \mu$, $\tilde V\sim \nu$ and $\tilde \EE (\tilde V\,|\, \tilde U) = \tilde U$.  Similarly Strassen's theorem (\cite{Strassen}) establishes the equivalence with the existence of a martingale Markovian kernel~(\footnote{For every $x\!\in \RR^d$, $K(x,dy)$ is a probability measure on $(\RR^d, {\cal B}or(\RR^d))$ and the function $x\mapsto K(x,A)$ is Borel for every fixed Borel set $A$ of $\RR^d$.}) $K(x,dy)$ such that $\nu(dy) = \int_{\RR^d} K(x,dy)\mu(dx)$ and $\int_{\RR^d} yK(x,dy) = x$  for every $x\!\in \RR^d$. 

The functional convex order for two Brownian martingale diffusion processes having a form $dX_{t}=b(t, X_{t})dt+\sigma(t, X_{t})dB_{t}$ has been studied in~\cite{pages2016convex},~\cite{alfonsi2018sampling} and~\cite{jourdain2019convex} (among other references). Such functional convex order results have applications  in  quantitative finance to establish robust bounds for various option prices including those written on path-dependent payoffs. %([Yating] verifier svp). 
In this paper, we extend such functional convex order results to  the McKean-Vlasov equation, which was originally introduced in~\cite{mckean1967propagation} as a stochastic model naturally associated to a class of non-linear PDEs. Nowadays, it refers to the whole family of stochastic differential equations whose coefficients not only depend on the position of the process $X_{t}$  at time $t$  but also on its probability distribution $\mathbb{P}_{X_{t}}=\mathbb{P}\,\circ\, X_{t}^{-1}$. Thanks to this specific structure, the McKean-Vlasov equations have become widely used to model phenomenons  %un common framework not only
 in Statistical Physics (see e.g.~\cite{martzel2001mean}), in mathematical biology (see e.g.~\cite{MR2974499} and~\cite{Bossyetal}), but also in social sciences and in quantitative finance often motivated by the development of the Mean-Field Games (see e.g. \cite{MR3072222}, \cite{MR3395471}, \cite{cardaliaguet2018mean} and~\cite{MR3752669}). Moreover, results in this paper can be used to establish the convex bounds and the convex partitions, which may be extended to applications within the framework of Mean-Field Games in a future work (see further Section~\ref{appli}). For example, the modeling of the gas storage (see e.g. \cite{RapGasnier}), or the stochastic control of McKean–Vlasov type when the control appears in the volatility coefficient (see e.g. \cite{MR2822408}, \cite{MR3072755}).

We consider now a filtered probability space $(\Omega, \mathcal{F}, (\mathcal{F}_{t})_{t\geq0}, \mathbb{P})$ satisfying the usual condition and an $(\mathcal{F}_{t})$-standard Brownian motion $(B_{t})_{t\geq0}$ defined on this space and valued in $\mathbb{R}^{q}$. Let %$\mathcal{P}(\RD)$ denote the set of all probability distributions on $\RD$ and let 
$\mathbb{M}_{d\times q}(\mathbb{R})$ denote the set of matrices with $d$ rows and $q$ columns equipped with the operator norm $\vertiii{\cdot}$ defined by $\vertiii{A}\coloneqq \sup_{\left|z\right|_{q}\leq1}\left|Az\right|$, where $\left|\cdot\right|$ denotes the canonical Euclidean norm on $\RD$ generated by the canonical inner product $\langle\cdot|\cdot\rangle$. %The aim of this paper is to establish functional convex order results for two $d$-dimensional scaled McKean-Vlasov processes 
Let 
$X=(X_{t})_{t\in[0, T]}$ and $Y=(Y_{t})_{t\in[0, T]}$ be two $d$-dimensional McKean-Vlasov processes, respective solutions to 
\begin{align}\label{defconvx}
&dX_{t}=\driftx\,dt+\sigma(t, X_{t}, \mu_{t})dB_{t}, \;\;\;X_{0}\in L^{p}(\mathbb{P}),\\
\label{defconvy}
&dY_{t}\,=\drifty\,dt+\theta(t, \,Y_{t}, \,\nu_{t}\,)\,dB_{t}, \;\;\;\,Y_{0}\in L^{p}(\mathbb{P}),
\end{align}
where $p\geq2$, $b: [0, T]\times\mathbb{R}^{d}\times \mathcal{P}_{p}(\mathbb{R}^{d})\rightarrow \RD$,  %{\color{black}$\alpha: [0, T]\rightarrow \mathbb{M}_{d\times d}, \, \beta: [0, T]\times \RD\rightarrow \RD$}, 
%$\, \gamma: [0, T]\rightarrow\RD$, 
$\sigma, \theta\,:\,[0, T]\times\mathbb{R}^{d}\times \mathcal{P}_{p}(\mathbb{R}^{d})\rightarrow\mathbb{M}_{d\times q}$ 
%$\sigma, \theta$ are two functions defined on $[0, T]\times\mathbb{R}^{d}\times \mathcal{P}_{p}(\mathbb{R}^{d})$ valued in $\mathbb{M}_{d\times q}$ 
and, for every $t\in[0, T]$, $\mu_{t}$ and $\nu_{t}$ respectively denote the probability distribution of $X_{t}$ and $Y_{t}$. 

  In this paper, we will only consider the \textit{scaled} McKean-Vlasov processes, which means the drift function $b: (t, x, \mu)\in [0, T]\times\mathbb{R}^{d}\times \mathcal{P}_{p}(\mathbb{R}^{d})\mapsto b(t,x,\mu)\!\in \RD$ {\em  is affine in $x$} (see further Assumption~\ref{AssumptionII}-(1)).

We define a \textit{partial order} between two matrices in $\mathbb{M}_{d\times q}$ as follows: 
\begin{equation}\label{matrixorder}
\forall\, \, A, B\in\mathbb{M}_{d\times q},\quad \text{$A\preceq B$\quad if \quad$BB^{\top}-AA^{\top}$ is a positive semi-definite matrix},
\end{equation}
where $A^{\top}$ stands for the transpose of the matrix $A$. 
Moreover, we introduce the $L^{p}$-\textit{Wasserstein distance} $\mathcal{W}_{p}$ on $\mathcal{P}_{p}(\mathbb{R}^{d})$ defined for any $\mu, \nu\in\mathcal{P}_{p}(\mathbb{R}^{d})$ by 
\begin{align}\label{defwas2}
\mathcal{W}_{p}(\mu,\nu)&=\Big{(}\inf_{\pi\in\Pi(\mu,\nu)}\int_{\mathbb{R}^{d}\times \mathbb{R}^{d}}d(x,y)^{p}\pi(dx,dy)\Big{)}^{\frac{1}{p}}\nonumber\\
&=\inf\Big{\{}\Big{[}\mathbb{E}\,\left|X-Y\right|^{p}\Big{]}^{\frac{1}{p}},\,  X,Y:(\Omega,\mathcal{A},\mathbb{P})\rightarrow( \mathbb{R}^{d},Bor( \mathbb{R}^{d}))  \,\text{with} \,\mathbb{P}_{X}=\mu, \mathbb{P}_{Y}=\nu\,\Big{\}},
\end{align}
where in the first line of~(\ref{defwas2}), $\Pi(\mu,\nu)$ denotes the set of all probability measures on $( \mathbb{R}^{d}\times  \mathbb{R}^{d}, Bor( \mathbb{R}^{d})^{\otimes2})$ with marginals $\mu$ and $\nu$. 

\smallskip
Throughout this paper, we make the following two assumptions on the coefficients $b$, $\sigma$ and the starting values $X_0$ and $Y_0$. Both depend on an integrability exponent $p\!\in [2, +\infty)$. 

%\noindent\textbf{Assumption (I)}: 
\begin{manualtheorem}{I}\label{AssumptionI}
%There exists $p\in[2, +\infty)$ such that 
Assume $\vertii{X_{0}}_{p}\vee \vertii{Y_{0}}_{p}<+\infty$. 
%The functions $\alpha, \beta, \gamma$ are $\rho\,$-H\"older, $0<\rho\leq 1$.  
The functions $b$, $\sigma$ and $\theta$ are $\rho\,$-H\"older continuous in $t$ and Lipschitz continuous in $x$ and in $\mu$ in the following sense: for every $s, t\in[0, T]$ with $s\le t$, there exist a positive constant $\tilde{L}$ such that
\begin{align}\label{assumpholder}
&\forall\, \, x\in\mathbb{R}^{d}, \,\forall\,\mu\in\mathcal{P}_{p}(\mathbb{R}^{d}),\,%\textcolor{black}{\hbox{norme $\vertiii{.}$ non definie a ce stade : operateur/Frobenius?}} 
\nonumber\\
& \hspace{0.5cm}\left|b(t, x, \mu)-b(s, x, \mu)\right|\vee\vertiii{\sigma(t, x, \mu)-\sigma(s, x, \mu)}\vee\vertiii{\theta(t, x, \mu)-\theta(s, x, \mu)}\nonumber\\
& \hspace{4cm}\leq \tilde{L}\big(1+\left|x\right|+\mathcal{W}_{p}(\mu, \delta_{0})\big)(t-s)^{\rho}, &
\end{align}
where $\delta_{0}$ denotes the Dirac mass at $0$; for every $t\in[0, T]$, there exists $L>0$ such that
\begin{align}\label{assumplip}
&\forall\, \, x,y \in\mathbb{R}^{d},\,\forall\, \, \mu, \nu\in\mathcal{P}_{p}(\mathbb{R}^{d}),\nonumber\\
&\hspace{0.5cm}\left|b(t, x, \mu) - b(t, y, \nu)\right|\vee\vertiii{\sigma(t, x, \mu) - \sigma(t, y, \nu)}\vee\vertiii{\theta(t, x, \mu) - \theta(t, y, \nu)}\nonumber\\
&\hspace{4cm}\leq L\Big(\left|x-y\right|+\mathcal{W}_{p}(\mu, \nu)\Big).
\end{align}
\end{manualtheorem}

%\noindent\textbf{Assumption (II)}: 

\begin{manualtheorem}{II}\label{AssumptionII}

\noindent$(1)$ The function $b$ is affine in $x$ and constant in $\mu$ w.r.t the convex order in the sense that for every $\mu, \,\nu\in\mathcal{P}_{p}(\RD)$ with $\mu\conright\nu$, we have 
\begin{equation}\label{oldb}
\forall\, (t,x)\in[0, T]\times \RD, \quad b(t, x, \mu)=b(t, x, \nu).
\end{equation}

\vspace{-0.2cm}

\noindent$(2)$ For every fixed $t\in \mathbb{R}_{+}$ and $\mu\in\mathcal{P}_{p}(\mathbb{R}^{d})$, the function $x\mapsto\sigma(t, x, \mu)$ is convex in the sense that 
\begin{equation}\label{assumptionuconv}
\forall\, x, y\in\mathbb{R}^{d}, \forall\lambda\in[0, 1], \hspace{0.5cm} \sigma(t, \lambda x+(1-\lambda)y, \mu)\preceq \lambda\sigma(t, x, \mu)+(1-\lambda)\sigma(t, y, \mu). 
\end{equation}

\noindent$(3)$ For every fixed $(t,x)\in\RR_{+}\times\RD$, the function $\mu\mapsto\sigma(t, x, \mu)$ is non-decreasing with respect to the convex order, that is, 
\begin{equation}\label{assumptionorder}
\forall\, \mu, \nu\in \mathcal{P}_{p}(\RD),\quad\mu\preceq_{\,cv}\nu  \quad\Longrightarrow \quad \sigma(t, x, \mu)\preceq \sigma(t, x, \nu).
\end{equation}

\noindent$(4)$  For every $(t, x, \mu)\in \mathbb{R}_{+}\times \mathbb{R}^{d}\times\mathcal{P}_{p}(\mathbb{R}^{d})$, we have
\begin{equation}\label{hysigthe}
\sigma(t, x, \mu)\preceq \theta(t, x, \mu).
\end{equation}

\noindent$(5)$ $X_{0}\preceq_{\,cv}Y_{0}$.
%{\color{black}\noindent$(5)$ The function $b$ is affine in $x$ and for any $\mu, \nu\in \mathcal{P}_{1}(\RD),\,\mu\preceq_{\,cv}\nu$, }
%\begin{equation}
%\color{black}b(t, x, \mu)=b(t, x, \nu).
%\end{equation}
\end{manualtheorem}

\begin{rem} Note that if Assumption~II  is satisfied with some $p_0\ge 1$ (especially when $p_0=1$) then it is satisfied for any $p\ge p_0$ by the restrictions of $b$ and $\sigma$ to $[0,T]\times \RD\times {\cal P}_p(\RD)$ since ${\cal P}_p(\RD)\subset {\cal P}_{p_0}(\RD)$. Idem for Assumption~I, except of course for the integrability of $X_0$ and $Y_0$, since ${\cal W}_{p_0}\le {\cal W}_p$. More generally we will often use without specific mention that the restriction to ${\cal P}_p(\RD)$ of a ${\cal W}_1$-continuous (resp. Lipschitz) functional $\Phi : {\cal P}_1(\RD) \to \RR$  is ${\cal W}_p$-continuous (resp. Lipschitz).
\end{rem}
%\subsection{Main result}
Let $E$ denote a separable Banach space equipped with the norm $\left|\,\cdot\,\right|_{E}$. A function $f: (E, \left|\,\cdot\,\right|_{E})\rightarrow \mathbb{R}$ has an \textit{$r$-polynomial growth} for some $r\ge 0$ if there exists a constant $C\in\mathbb{R}_{+}^{*}$ such that for every 
$ x\in E, \; \left|f(x)\right|\leq C(1+\left|x\right|_{E}^{r})$. %}
%Let $\mathcal{C}()$
Moreover, let 
\begin{flalign}\label{pathdisspacedef}
&\mathcal{C}\big( [0, T], \mathcal{P}_{p}(\mathbb{R}^{d})\big)\!:=\!\big\{(\mu_{t})_{t\in[0, T]}\text{ such that the mapping }t\mapsto\mu_{t}&\nonumber\\
&\hspace{4cm} \text{ is  continuous  from } [0,T] \text{ to } \big(\mathcal{P}_{p}(\mathbb{R}^{d}),\! \mathcal{W}_{p}\big) \big\}&%\hskip-0.45cm
\end{flalign}
equipped with the distance 
\begin{equation}\label{distancedc}
  d_{\mathcal{C}}\big((\mu_{t})_{t\in[0, T]}, (\nu_{t})_{t\in[0, T]}\big)\coloneqq\sup_{t\in[0, T]}\mathcal{W}_{p}(\mu_{t}, \nu_{t})
\end{equation}
be the space in which the marginal distribution of $X=(X_{t})_{t\in [0, T]}$ and $Y=(Y_{t})_{t\in [0, T]}$ have values. The continuity of $t\mapsto \mu_t = \PP_{X_t}$ will be proved later in Lemma~\ref{injectionmeasure}. 

The main theorem of this paper is the following.

\begin{thm}\label{mainthmconv} Let $p\!\in [2, +\infty)$. Assume~\ref{AssumptionI} and~\ref{AssumptionII} are in force.
%Assume that the equations~(\ref{defconvx}) and~(\ref{defconvy}) satisfy Assumptions~\ref{AssumptionI} and~\ref{AssumptionII}. 
Let $X\coloneqq (X_{t})_{t\in[0, T]}$, $Y\coloneqq (Y_{t})_{t\in[0, T]}$  denote the solutions of the McKean-Vlasov equations~(\ref{defconvx}) and~(\ref{defconvy}) respectively. For every $t\in[0, T]$, let $\mu_{t}$, $\nu_{t}$  denote the probability distributions of $X_{t}$ and $Y_{t}$ respectively.   Then, we have

%{\color{red}\noindent $(a)$ {\em(Marginal convex order)} For every $t\in[0, T]$, we have $\mu_t\conright\nu_t$.}

\noindent $(a)$ {\em Functional convex order.} For any convex function $F: \big(\mathcal{C}([0, T], \mathbb{R}^{d}), \vertii{\cdot}_{\sup}\big) \rightarrow \mathbb{R}$ with %$(r, \vertii{\cdot}_{\sup})$-
$p$-polynomial growth, 
%$ 0< r  < p $, %in the sense that 
%\begin{equation}\label{rpolygrowth}
%\forall\, \, \alpha\in\mathcal{C}([0, T], \mathbb{R}^{d}), \text{ there exists } C>0, \;\text{ s.t. } \;\left|F(\alpha)\right|\leq C\big(1+\vertii{\alpha}_{\sup}^{r}\big), 
%\end{equation}
one has 
\begin{equation}\label{convf}
\EE F(X)\leq \EE F(Y). 
\end{equation}

\noindent $(b)$ {\em Extended functional convex order.} For any function \[G: \big(\alpha, (\eta_{t})_{t\in[0, T]}\big)\in\mathcal{C}\big([0, T], \RD\big)\times \mathcal{C}\big([0, T], \mathcal{P}_{1}(\RD)\big)\mapsto G\big(\alpha, (\eta_{t})_{t\in[0, T]}\big)\in \RR\]
%$G: \big(\alpha, (\eta_{t})_{t\in[0, T]}\big)\in\mathcal{C}\big([0, T], \RD\big)\times \mathcal{C}\big([0, T], \PPRD\big)\mapsto G\big(\alpha, (\eta_{t})_{t\in[0, T]}\big)\in\RR$
satisfying the following conditions:

\begin{enumerate}[$(i)$]
\item $G$ is convex in $\alpha$, 
\item $G$ has a $p$-polynomial growth
%, ${\color{ForestGreen} 0< r  < p} $, 
in the sense that 
\begin{align}
\nonumber&\exists \,C\in\mathbb{R}_{+}\text{ such that }\forall\, \, \big(\alpha, (\eta_{t})_{t\in[0, T]}\big)\in\mathcal{C}\big([0, T], \RD\big)\times \mathcal{C}\big([0, T], \PPRD\big),  \\
\label{rpolygrowth}&G\big(\alpha, (\eta_{t})_{t\in[0, T]}\big)\leq C\big[1+\vertii{\alpha}_{\sup}^{p}+\sup_{t\in[0, T]}\mathcal{W}_{p}^{\,p}(\eta_{t}, \delta_{0})\big],
\end{align}

\item $G$ is continuous in $(\eta_{t})_{t\in[0, T]}$ with respect to the distance $d_{\mathcal{C}}$ defined in~(\ref{distancedc}) and non-decreasing in $(\eta_{t})_{t\in[0, T]}$ with respect to the convex order in the sense that 
\begin{flalign}
&\forall\, \alpha\in \mathcal{C}\big([0, T], \RD\big),\: \forall\, \, (\eta_{t})_{t\in[0, T]}, (\tilde{\eta}_{t})_{t\in[0, T]}\in \mathcal{C}\big([0, T], \PPRD\big) \text{ s.t. } \forall\, \, t \!\in[0, T],\,\eta_{t}\conright\tilde{\eta}_{t},&\nonumber\\
& \hspace{4cm} \;G\big(\alpha,  (\eta_{t})_{t\in[0, T]}\big)\leq G\big(\alpha,  (\tilde{\eta}_{t})_{t\in[0, T]}\big),&\nonumber
 \end{flalign}
\end{enumerate}

\noindent one has 
\begin{equation}\label{convgpro}
\EE G\big(X, (\mu_{t})_{t\in[0, T]}\big)\leq \EE G\big(Y, (\nu_{t})_{t\in[0, T]}\big). 
\end{equation}
\end{thm}

%{\color{red}======[New][BEGIN] [dual case]=======}

The proof is postponed to Section~\ref{convprocess} (and Section~\ref{eulerconv} for preliminary discrete  time results). The symmetric case of Theorem~\ref{mainthmconv} remains true, that is, if we replace Assumption~\ref{AssumptionII} by Assumption~\ref{AssumptionII}\textcolor{blue}{'} where conditions $(4)$ and $(5)$ are replaced respectively by $(4')$ and $(5')$ as follows:
%\noindent\textbf{Assumption (II, dual)}: 
\begin{enumerate}
\item[$(4')$] \textit{For every $(t, x, \mu)\in \mathbb{R}_{+}\times \mathbb{R}^{d}\times\mathcal{P}_{p}(\mathbb{R}^{d})$, we have} 
%\begin{equation}\label{hysigthe}
$\theta(t, x, \mu)\preceq \sigma(t, x, \mu).$
%\end{equation}
\item[$(5')$]
%\noindent\textit{} 
$Y_{0}\preceq_{\,cv}X_{0}$,
\end{enumerate}
\noindent then we have the following result,  whose proof is very similar to that of Theorem~\ref{mainthmconv}.% which describes the dual case of Theorem~\ref{mainthmconv}. 

\begin{thm}[Symmetric setting]\label{mainthmconvdual} Let $p\!\in[2, +\infty)$. Under Assumption~\ref{AssumptionI} and~\ref{AssumptionII}\textcolor{blue}{'}, for every functions $F:\mathcal{C}([0, T], \mathbb{R}^{d})\rightarrow \mathbb{R}$ and $G: \mathcal{C}\big([0, T], \RD\big)\times \mathcal{C}\big([0, T], \PPRD\big)\mapsto \RR$ respectively satisfying the conditions in Theorem~\ref{mainthmconv} - $(a)$ and $(b)$, then
 \[
 \EE F(Y)\leq \EE F(X) \quad\text{and}\quad \EE G\big(Y, (\nu_{t})_{t\in[0, T]}\big)\leq \EE G\big(X, (\mu_{t})_{t\in[0, T]}\big).
 \]
\end{thm}

Theorem~\ref{mainthmconv} directly implies the following results.%whose proof is postponed in Section~\ref{convprocess}. 
\begin{cor}\label{cormain}
Let $X\coloneqq (X_{t})_{t\in[0, T]}$, $Y\coloneqq (Y_{t})_{t\in[0, T]}$  denote the solutions of the McKean-Vlasov equations~(\ref{defconvx}) and~(\ref{defconvy}) respectively. For every $t\in[0, T]$, let $\mu_{t}$, $\nu_{t}$  denote the probability distributions of $X_{t}$ and $Y_{t}$ respectively. Under Assumption~\ref{AssumptionI} and~\ref{AssumptionII}, we have :

\noindent $(a)$ {\em Marginal convex order.} For every $t\in[0, T]$, $\mu_t\conright\nu_t$.

\noindent $(b)$ {\em Convexity with respect to the initial value.}  Let $X^{x}=(X_{t}^{x})_{t\in[0, T]}$ denote the McKean-Vlasov process defined by~\eqref{defconvx} starting with the initial value $X_0=x$. Then for every functionals $F$ and $G$ respectively satisfying conditions from Theorem~\ref{mainthmconv}-$(a)$ and $(b)$, the functions 
$$
x\mapsto \mathbb{E}\,F(X^{x})\quad\mbox{ and }\quad x\mapsto \mathbb{E}\,G\big(X^{x}, (\mu_t)_{t\in[0,T]}\big) \mbox{ are convex.}
$$ 
\end{cor}
The proof of Corollary~\ref{cormain} is postponed to  Section~\ref{eulerconv}. It also has an obvious  version under Assumption~\ref{AssumptionII}\textcolor{blue}{'}.

In fact, as far as marginal convex order is concerned, it is also possible to dissociate convexity in $x$ and monotonicity in $\mu$ that is replace Assumption~\ref{AssumptionII}-(3) by the following assumption:

\noindent{\em$(3')$ For every fixed $(t,x)\in\RR_{+}\times\RD$, the function $\theta(t, x, \cdot)$ is non-decreasing in $\mu$ with respect to the convex order in the sense that
\begin{equation}
\forall\, \, \mu, \nu\in\mathcal{P}_{p}(\RD), \;\mu\conright\nu, \quad\quad\quad \theta(t, x, \mu)\preceq\theta(t, x, \nu).\nonumber
\end{equation}}

\vspace{-0.5cm}
\noindent Then, we have the same result in Corollary~\ref{cormain}-(a). This is the purpose of Proposition~\ref{marginal2} in Section~\ref{eulerconv}.
 
\noindent %The proof of Corollary~\ref{cormain}-(b) is obvious as for every $x, y\!\in \RD$, we have $\delta_{\lambda x + (1-\lambda)y}\conright \lambda \delta_{x}+ (1-\lambda)\delta_{y}$ and one %can conclude  by taking $\theta=\sigma$, $X_0\sim \delta_{\lambda x + (1-\lambda)y}$ and $Y_0\sim \lambda \delta_{x}+ (1-\lambda)\delta_{y}$.

\smallskip
This paper is organized as follows. Section~\ref{comments} contains comments on the Assumption~\ref{AssumptionI} and~\ref{AssumptionII} including necessary and sufficient conditions on the monotonicity with respect to the convex order in terms of the linear functional derivative. Next, in Section~\ref{appli},  we show two applications of Theorem~\ref{mainthmconv} and~\ref{mainthmconvdual} in the framework of the stochastic differential equation and the stochastic optimal control. The proof of the main theorem is constructed in Section~\ref{eulerconv}-\ref{convprocess}. Our strategy of proof is to first establish the propagation of convex order for the marginal distribution of the Euler scheme of the McKean-Vlasov equation (see Section~\ref{eulerconv}) and then rely on it to establish in a backward way the functional convex order for the whole trajectory (Section~\ref{convprocess}).  To be more precise, in Section~\ref{eulerconv}, we show the convex order result for $\big(\bar{X}^{M}_{t_{m}}\big)_{m=0, \dots,M}$ and $\big(\bar{Y}^{M}_{t_{m}})_{m=0, \dots, M}$ defined by the Euler schemes (see further (\ref{defeulerx}) and~(\ref{defeulery})).  We first prove that the Euler scheme propagates the marginal convex order, namely, for every $m=0, \ldots, M$, $\bar{X}^{M}_{t_{m}}\conright\bar{Y}^{M}_{t_{m}}$. Then we prove the functional convex order 
\begin{equation}\label{conveuler}
\EE F(\bar{X}^{M}_{t_{0}}, \ldots, \bar{X}^{M}_{t_{M}})\leq \EE F(\bar{Y}^{M}_{t_{0}}, \ldots, \bar{Y}^{M}_{t_{M}})
\end{equation} 
for any convex function $F: (\RD)^{M+1}\rightarrow\RR$ with $p$-polynomial growth, by using a backward  dynamic programming principle. Next, in Section~\ref{convprocess}, we prove Theorem~\ref{mainthmconv}, the functional convex order result for the stochastic processes and their probability distributions based on~(\ref{conveuler}) by applying the convergence of the Euler schemes of the McKean-Vlasov equation.  At the end, in Appendix A, we propose a detailed proof of the convergence rate of the Euler scheme for the McKean-Vlasov equation in the general setting
\[dX_{t}=b(t, X_{t}, \mu_{t})dt+\sigma(t, X_{t}, \mu_{t})dB_{t},\] 
where $b$, $\sigma$ are Lipschitz in $(x, \mu)$ and $\rho\,$-H\"older in $t$.

\noindent\textit{Generalization in dimension 1. } In one dimension, it is possible to consider more general drift $b$ (convex in $x$ and non-decreasing in $\mu$ for  convex ordering)  if we restrict to  monotone (non-decreasing) convex order. This idea originated from Hajek's theorem in (\cite{MR771469}) established for Brownian diffusions by other methods. However our approach based on the Euler scheme cannot be adapted straightforwardly:  a truncated version of the scheme is necessary to complete the proofs which adds significant some technicalities. This extension for the McKean-Vlasov equations is developed  in a devoted paper~\cite{liu2021monotone}.

\section{Comments on the assumptions}\label{comments}
In this section, we give some comments on the assumptions made in this paper. In Section~\ref{ci}, we prove that Assumption~\ref{AssumptionI} implies the convergence of the Euler scheme for the McKean-Vlasov equations and in Section~\ref{cii}, we give some necessary and sufficient conditions for Assumption~\ref{AssumptionII}-(1) and (3).

\subsection{Comments on Assumption~\ref{AssumptionI}}\label{ci}
Let \[ \mathcal{C}([0, T], \mathbb{R}^{d})\coloneqq
 \{f: [0, T]\rightarrow \RD \text{ continuous function}\}\] equipped with the uniform norm $\vertii{f}_{\sup}\coloneqq \sup_{t\in[0, T]}|f(t)|$. 
Assumption~\ref{AssumptionI} guarantees the existence and strong uniqueness of the respective solutions
of (\ref{defconvx}) and (\ref{defconvy}) in $L_{\CRD}^{p}(\Omega, \mathcal{F}, \mathbb{P})$ (see \cite[Section 5.1]{liu:tel-02396797}, \cite[Theorem 3.3]{lacker2018mean}) and the convergence of the following Euler scheme.
%{\color{black}Assumption~\ref{AssumptionI} is a classical condition for the existence and strong uniqueness of the McKean-Vlasov equation (see \cite{}, \cite{}) and for the strong convergence of the following Euler scheme (see further in  Appendix~\ref{CtoA} for the proof).}
Let $M\in\mathbb{N}^{*}$ and let $ h= \frac{T}{M}$. For $m=0, \ldots, M$, we define $t_{m}^{M}\coloneqq h\cdot m=\frac{T}{M}\cdot m$. When there is no ambiguity, we write $t_{m}$ instead of $t_{m}^{M}$. Let  $Z_{m}\coloneqq\frac{1}{\sqrt{h}}(B_{t_{m+1}}-B_{t_{m}}),  \,m=1, \ldots, M,$ be i.i.d random variables having probability distribution $\mathcal{N}(0, \mathbf{I}_{q})$, independent of $X_{0}$ and $Y_{0}$. The Euler schemes of  equations~(\ref{defconvx}) and~(\ref{defconvy}) are defined by 
\begin{align}\label{defeulerx}
\bar{X}_{t_{m+1}}^{M}\!\!=\bar{X}_{t_{m}}^{M}+h\cdot b(t_{m}, \bar{X}_{t_{m}}^{M}, \bar{\mu}_{t_{m}}^{M})+\sqrt{ h\,}\cdot\sigma(t_{m}, \bar{X}_{t_{m}}^{M}, \bar{\mu}_{t_{m}}^{M})Z_{m+1}, \quad \bar{X}^M_{0}=X_{0}\\%&\,\, \bar{X}^M_{0}=X_{0},\!\!\\
\label{defeulery}
\bar{Y}_{t_{m+1}}^{M}\!=\,\bar{Y}_{t_{m}}^{M}+ h\cdot b(t_{m}, \bar{Y}_{t_{m}}^{M}, \,\bar{\nu}_{t_{m}}^{M})+\sqrt{ h\,}\cdot\,\theta(t_{m}, \bar{Y}_{t_{m}}^{M}, \,\bar{\nu}_{t_{m}\,}^{M})Z_{m+1}, \quad \; \bar{Y}^M_{0}=Y_{0}
\end{align}
%(avant revision)\begin{align}\label{defeulerx}
%(avant revision)\bar{X}_{t_{m+1}}^{M}\!\!=\bar{X}_{t_{m}}^{M}+h\cdot\big(\alpha(t_{m}) \bar{X}_{t_{m}}^{M}+\beta(t_{m})\,\EE\bar{X}_{t_{m}}^{M}+\gamma(t_{m})\big)+\sqrt{ h\,}\cdot\sigma(t_{m}^{M}, \bar{X}_{t_{m}}^{M}, \bar{\mu}_{t_{m}}^{M})Z_{m+1}, \\%&\,\, \bar{X}^M_{0}=X_{0},\!\!\\
%(avant revision)\label{defeulery}
%(avant revision)\bar{Y}_{t_{m+1}}^{M}\!=\,\bar{Y}_{t_{m}}^{M}+ h\cdot \big(\alpha(t_{m}) \,\bar{Y}_{t_{m}}^{M}+\beta(t_{m})\,\EE\,\bar{Y}_{t_{m}}^{M}+\gamma(t_{m})\big)+\sqrt{ h\,}\cdot\,\theta(t_{m}^{M}, \bar{Y}_{t_{m}}^{M}, \,\bar{\nu}_{t_{m}}^{M})Z_{m+1}, 
%(avant revision)\end{align}
where %$\bar{X}^M_{0}=X_{0}, \bar{Y}^M_{0}=Y_{0}$ and 
for every $m=0, \ldots, M$, $\bar{\mu}_{t_{m}}^{M}$ and $\bar{\nu}_{t_{m}}^{M}$ respectively denote the probability distribution of $\bar{X}_{t_{m}}^{M}$ and $\bar{Y}_{t_{m}}^{M}$.
Moreover, we classically define the {\em genuine} (or continuous time) Euler scheme $\bar{X}=(\bar{X}_{t}^{M})_{t\in[0, T]}$, $\bar{Y}=(\bar{Y}_{t}^{M})_{t\in[0, T]}$ as follows: for every $t\in[t_{m}, t_{m+1})$, 
\begin{align}\label{defeulercontinuousx}
\bar{X}_{t}^{M}\coloneqq\bar{X}_{t_{m}}^{M}+{\color{black}b(t_{m}, \bar{X}_{t_{m}}^{M}, \bar{\mu}_{t_{m}}^{M})}(t-t_{m})+\sigma(t_{m}, \bar{X}_{t_{m}}^{M}, \bar{\mu}_{t_{m}}^{M})(B_{t}-B_{t_{m}}), \\%& \quad\bar{X}_{0}=X_{0},\\
\label{defeulercontinuousy}
\bar{Y}_{t}^{M}\,\coloneqq\,\bar{Y}_{t_{m}}^{M}+  {\color{black}b(t_{m}, \bar{Y}_{t_{m}}^{M}, \,\bar{\nu}_{t_{m}}^{M})}(t-t_{m})+\theta(t_{m}, \bar{Y}_{t_{m}}^{M}, \,\bar{\nu}_{t_{m}}^{M})(B_{t}-B_{t_{m}}).% &\,\quad\bar{Y}_{0}=Y_{0},
\end{align}
When there is no ambiguity, we write $\bar{X}_{m}$ and $\bar{X}_{t}$ instead of $\bar{X}^{M}_{t_{m}}$ and $\bar{X}^{M}_{t}$ to simplify the notation. 
%\begin{align}\label{defeulercontinuousx}
%\bar{X}_{t}^{M}\coloneqq\bar{X}_{t_{m}}^{M}+(\alpha \bar{X}_{t_{m}}^{M}+\beta)(t-t_{m})+\sigma(t_{m}^{M}, \bar{X}_{t_{m}}^{M}, \bar{\mu}_{t_{m}}^{M})(B_{t}-B_{t_{m}}), \\%& \quad\bar{X}_{0}=X_{0},\\
%\label{defeulercontinuousy}
%\bar{Y}_{t}^{M}\,\coloneqq\,\bar{Y}_{t_{m}}^{M}+  (\alpha \,\bar{Y}_{t_{m}}^{M}+\beta)(t-t_{m})+\theta(t_{m}^{M}, \bar{Y}_{t_{m}}^{M}, \,\bar{\nu}_{t_{m}}^{M})(B_{t}-B_{t_{m}}).% &\,\quad\bar{Y}_{0}=Y_{0},
%\end{align}

%\vspace{2cm}
%{\color{black}The $p$ in Assumption I}
%\vspace{2cm}
The value $p\in[2, +\infty)$ in Assumption~\ref{AssumptionI} such that $\vertii{X_{0}}_{p}\vee \vertii{Y_{0}}_{p}<+\infty$ and in the Lipschitz condition~\eqref{assumplip} is crucial for 
%The following proposition, whose proof is postponed to Appendix~\ref{appn}, shows 
the moment controls of the processes $X$, $Y$, $(\bar{X}_{t})_{t\in[0, T]}$ and $(\bar{Y}_{t})_{t\in[0, T]}$ and the $L^{p}$-strong convergence result for the continuous Euler scheme (\ref{defeulercontinuousx}) and (\ref{defeulercontinuousy}). For convenience, we state the following proposition only for $X$ and  $(\bar{X}_{t})_{t\in[0, T]}$ but the results remain true for $Y$ and $(\bar{Y}_{t})_{t\in[0, T]}$. The proof of Proposition~\ref{cvgeuler} is postponed to Appendix~\ref{appn}.% The proof of Proposition~\ref{cvgeuler} . 
%}
\begin{prop}\label{cvgeuler}
%Let $(\bar{X}_{t_{m}}^{M})_{0\leq m\leq M}$  be random variables  defined by~(\ref{defeulerx}). 
Assume Assumption~\ref{AssumptionI} is in force. 
\begin{enumerate}[$(a)$]
%\item The McKean-Vlasov equation~(\ref{defconvx}) has a unique strong solution $X=(X_{t})_{t\in[0, T]}$ in $L_{\CRD}^{p}(\Omega, \mathcal{F}, \mathbb{P})$, where 
%\begin{equation}\label{pathspacedef}
% \mathcal{C}([0, T], \mathbb{R}^{d})\coloneqq
% \{f: [0, T]\rightarrow \RD \text{ continuous function}\}%\{(f_{t})_{t\in[0, T]}\text{ s.t. } t\mapsto f_{t} \text{ is a continuous application from $[0, T]$ to $\RD$ }\}, %denote the space of continuous applications defined on  
%\end{equation}
%equipped with the uniform norm $\vertii{f}_{\sup}\coloneqq \sup_{t\in[0, T]}|f(t)|$. %, where $|\cdot|$ denotes the norm on $\mathbb{R}^{d}$. 
%.
\item There exists a constant $C$ depending on $p, d, \sigma, \theta, T, L$ such that, for every $t\in[0, T]$ and for every $M\geq 1$, 
\begin{equation}\label{boundedx}
\vertii{\sup_{u\in[0, t]}\left|X_{u}\right|}_{p}\vee \vertii{\sup_{u\in[0, t]}\left|\bar{X}_{u}^{M}\right|}_{p}\leq C(1+\vertii{X_{0}}_{p}).
\end{equation}
Moreover, there exists a constant $\kappa$ depending on $L, b, \sigma, \left\Vert X_{0}\right\Vert_{p}, p, d, T$ such that for any $s,t\in[0, T]$, $s\le t$, 
\[\forall\, \, M\geq 1, \quad\left\Vert \bar{X}^{M}_{t}-\bar{X}^{M}_{s}\right\Vert_{p}\vee\left\Vert X_{t}-X_{s}\right\Vert_{p}\leq \kappa\sqrt{t-s}.\]

\item There exists a constant $\widetilde{C}$ depending on $p, d, T, L, \tilde{L}, \rho, \vertii{X_{0}}_{p}$ such that 
%\begin{align}
%\label{cvgeulerx}
%\mathbb{W}_{p}(\bar{X}, X)\leq 
\[\Big\Vert\sup_{t\in[0, T]}\left|X_{t}-\bar{X}^{M}_{t}\right|\Big\Vert_{p}\leq \widetilde{C} h^{\frac{1}{2}\land \rho}. \]%\nonumber
%\end{align}
\end{enumerate}
\end{prop}
%\noindent 
%\noindent 
%We %refer to \cite[Section 5.1]{liu:tel-02396797}, \cite[Theorem 3.3]{lacker2018mean} for the proof of Proposition~\ref{cvgeuler}-(a)
%and 
%postpone the proof of Proposition~\ref{cvgeuler} to Appendix~\ref{appn}. %{\color{black}(Lemma~\ref{eulertheorical} and Proposition~\ref{Srate})}.% with a more general drift $b(t, x, \mu)$ in the McKean-Vlasov equation.
%\noindent We postpone the proof of Proposition~\ref{cvgeuler} to Appendix A with a more general drift $b(t, x, \mu)$ in the McKean-Vlasov equation

\subsection{Comments on Assumption~\ref{AssumptionII}}\label{cii}

%$\%\%\%$

 %having the form of 
%\begin{equation}\label{baffine}
%\color{ForestGreen} b(t, x, \mu)=a(t)x+\beta\big(t, \int_{\RD}\xi\mu(d\xi)\big) \vspace{-0.1cm}
%\end{equation}
%with $a: [0, T]\rightarrow \mathbb{M}_{d\times d}$ and $\beta: [0, T]\times \RD\rightarrow \RD$.

%$\%\%\%$

 Assumption~\ref{AssumptionII} contains technical conditions. The drift $b$ is assumed to be affine and Lipschitz continuous in $x$, i.e. $b$ has the following form 
\begin{equation}\label{form1}
b(t, x, \mu)=\alpha(t)x+\beta(t, \mu). 
\end{equation}
In fact, Jensen's inequality implies that for every $\mu\in\mathcal{P}_{1}(\RD)$, $\delta_{\int \xi \mu(d\xi)}\conright \mu$ as for every convex function $f$, $f\big(\int \xi \mu(d\xi)\big)\leq\int f(\xi)\mu(d\xi)$. Hence the condition in $(\ref{oldb})$ implies $b(t, x, \mu)=b(t, x, \delta_{\int \xi \mu(d\xi)})$ so that the drift (\ref{form1}) is equivalent to the following drift
\begin{equation}\label{baffine}
\color{black} \tilde{b}(t, x, \mu)=\alpha(t)x+\tilde{\beta}\Big(t, \int_{\RD}\xi\mu(d\xi)\Big) \vspace{-0.1cm}
\end{equation}
with $\tilde{\beta}\big(t, \int_{\RD}\xi\mu(d\xi)\big)\coloneqq \beta(t, \delta_{\int \xi \mu(d\xi)})=\beta(t, \mu)$. 
 
\smallskip
Now we give necessary and sufficient conditions and a criterion based on  the linear functional derivative  to establish  monotonicity with respect to the convex order of a function $\Phi(\mu)$, as  it appears in  Assumption~\ref{AssumptionII}-(3). We will consider the case of probability measures on ${\cal P}_2(\RD)$ fro simplicity but what follows can be straightforwardly adapted to adapted to ${\cal P}_p(\RD)$ for a $p\!\in [1, +\infty)$. The proof of the following proposition is postponed to Appendix~\ref{proofpp}.

\begin{prop}\label{pp} Let $\Phi : \big(\mathcal{P}_{2}(\RD),\mathcal{W}_{2}\big)\rightarrow \RR$ be a continuous function. 

\noindent $(a)$   $\Phi$ is non-decreasing with respect to the convex order if and only if, for every $\mu,\nu\in\mathcal{P}_{2}(\RD),\;\mu\conright\nu$, 
\[
\liminf_{\varepsilon\rightarrow0^{+}} \frac{\Phi\big(\mu+\varepsilon (\nu-\mu)\big)-\Phi (\mu)}{\varepsilon}\geq 0.
\]
\noindent $(b)$ {\em Characterization when  $\Phi$ is smooth}.  Assume $\Phi: \big(\mathcal{P}_{2}(\RD), {\cal W}_2\big)\rightarrow \RR$ is linearly functionally differentiable with linear functional derivative $\frac{\delta\Phi}{\delta m}$ defined on $\mathcal{P}_{2}(\RD)\times \RD$ in the sense of~\cite[Definition~5.43]{MR3752669}~(\footnote{i.e.  $\frac{\delta\Phi}{\delta m}(\mu)(x)$ is jointly  continuous  in $(\mu,x)$ and,  for any ${\cal W}_2$-bounded subset ${\cal K} \subset {\cal P}_2(\RD)$, $x\mapsto \frac{\delta\Phi}{\delta m}(\mu)(x)$ has   at most quadratic growth in $x$ uniformly in  $\mu\!\in {\cal K}$, and satisfies
\[
\Phi(\mu')-\Phi(\mu) = \int_0^1\int_{\RD}\frac{\delta\Phi}{\delta m}(t\mu'+(1-t)\mu)(x)d(\mu'-\mu)(x)dt.
\]
Note that such a quantity is defined up to a real constant (not depending upon $x$).}). 
%Assume further that, for any $m\in \mathcal{P}_{2}(\RD)$, the function $x\mapsto \frac{\delta\Phi}{\delta m}(m)(x)$ is differentiable and the derivative $\partial_{x}\Big[\frac{\delta \Phi}{\delta m}\Big](m, x)$ is jointly continuous in $(m, x)$ and is at most of linear growth in $x$, uniformly in $m\in\mathcal{K}$ for any bounded subset $\mathcal{K}\subset 
%\mathcal{P}_{2}(\RD)$.
 Then,  the following conditions are equivalent. 
\begin{enumerate}[$(i)$]
\item The function $\Phi$ is non-decreasing w.r.t. the convex order on $\mathcal{P}_{2}(\RD)$.
\item For every $\mu, \nu \in \mathcal{P}_{2}(\RD)$ with $\mu\conright\nu$, $\displaystyle \int_{\RD}\frac{\delta \Phi}{\delta m}(\mu)(x)d(\nu-\mu)(x)\geq 0$.
%\item The function $x\mapsto \frac{\delta\Phi}{\delta m}(\mu)(x)$ is convex. (non c'est pas correct)
\end{enumerate}
$(c)$ {\em A convexity based criterion.} In particular, if, for every $\mu \!\in \mathcal{P}_{2}(\RD)$, $x \mapsto \frac{\delta \Phi}{\delta m}(\mu)(x)$ is convex, then $\Phi$ is non-decreasing for  the convex ordering on $\mathcal{P}_{2}(\RD)$.
\end{prop}

\begin{rem}  The converse of the above criterion i.e. $(b)$ implies the convexity of $\frac{\delta \Phi}{\delta m}(\mu)(x)$ in $x$ for every $\mu$ seems  not clear although we have no obvious counterexample.
\end{rem}
\begin{exple} $(a)$ Elementary  examples of such monotonic functions $\Psi$  on ${\cal P}_2(\RD)$ for convex ordering are functions of the form
\[
\Psi(\mu) = \chi\left( \int_{\RD} \psi(\xi)\mu(d\xi)\right)
\]
where $\psi:\RD\to \RR$ is a convex function with at most quadratic growth at infinity and $\chi:\RR\to \RR$ is nondecreasing.  In one dimension, typical examples of  such  functions $\psi$ are $\psi(\xi)= \EE\, [(aZ+b\xi)^+]^2$, $Z^+\!\in L^2$, $a\!\in \RR_+$, $b\!\in \RR$ or $\psi(x)=\EE\, |aZ+b\xi|^\gamma$, $1\le \gamma \le 2$ with $Z\!\in L^\gamma$, $a,\,b\!\in \RR$. Any positive linear combination  of such functions $\Psi$ is of course still non-decreasing for the convex ordering. This leads to consider the more general family of functions 
\begin{equation}\label{eq:2.8}
\Psi(\mu) = \int_{E}\chi\left( \int_{\RD} \psi(\xi,u)\mu(d\xi)\right)\pi(du)
\end{equation}
where $\pi$ is a non-negative $\sigma$-finite measure on a measure space $(E, {\cal E})$, $\psi(\cdot,u)$ is convex with quadratic growth $|\psi(x,u)|\le \kappa(u)(1+|x|^2)$,  $u\!\in E$,   such that  $\chi:\RR\to\RR$ is non-decreasing and $\int_E\Big|\chi\Big(\big(1+\int \!|\xi|^2\mu(d\xi)\big)\kappa(u)\big) \Big|d\pi (u)<+\infty$. 

Note that if $\chi$ in~\eqref{eq:2.8} is continuously differentiable, one has, under appropriate integrability conditions not detailed here, 
\[
\frac{\delta \Psi}{\delta \,m}(\mu)(x)=  \int_{E}\chi'\left( \int_{\RD} \psi(\xi,u)\mu(d\xi)\right)\psi(x,u)\pi(du)
\]
 which is clearly a convex function in $x$ for every distribution $\mu$ since $\chi'\ge 0$.
 
 \smallskip
 \noindent $(b)$   Let $W:\RD \to \RR$ be a convex function  with at most quadratic growth at infinity. Then the function $\Phi$ defined on ${\cal P}_2(\RD)$ by
 \[
 \Phi(\mu) =\tfrac 12 \int_{\RD}\int_{\RD}W(x-y)\mu(dx)\mu(dy)
 \]
is well defined and  non decreasing for convex ordering. This can be easily checked  directly since both $W(x-\cdot)$ and $W(\cdot- y)$ are convex functions. Nevertheless, one can  also check   that its linear functional derivative is given by 
\[
\frac{\delta \Psi}{\delta \,m}(\mu)(x) = \tfrac 12  \int_{\RD} \big(W(x-y)+ W(y-x)\big)\mu(dy),
\]
and is convex in $x$ for every $\mu \! \in {\cal P}_2(\RD)$.
\end{exple}
\section{Applications}\label{appli}
This section contains two applications of the main results. %The first example is 

\subsection{Application:  convex partitioning and convex bounding} \label{application}

Theorem~\ref{mainthmconv} and  Theorem~\ref{mainthmconvdual} show that we can \textit{upper and lower bound}, with respect to the functional convex order, a scaled McKean-Vlasov process by two scaled McKean-Vlasov processes satisfying Assumption~\ref{AssumptionII}-(1), (2), (3), or we can \textit{separate}, with respect to the functional convex order, two scaled McKean-Vlasov processes by a scaled McKean-Vlasov process satisfying Assumption~\ref{AssumptionII}-(1), (2), (3). That is, if we consider the following scaled McKean-Vlasov equations satisfying Assumption~\ref{AssumptionI}
\begin{align}%\label{comparesigma1}
&dX^{\sigma_{1}}_{t}=\big(a(t) X^{\sigma_{1}}_{t}+\beta(t,\,\EE X^{\sigma_{1}}_{t})\big)dt+\sigma_{1}(t, X^{\sigma_{1}}_{t}, \mu^{\sigma_{1}}_{t})dB_{t}, \;\;\;X^{\sigma_{1}}_{0}\in L^{p}(\mathbb{P}),\nonumber\\
%\label{comparetheta1}
&dY^{\theta_{1}}_{t}\,=\big(a(t)\,Y^{\theta_{1}}_{t}\,+\beta(t,\,\EE Y^{\theta_{1}}_{t}\,)\big)dt+\theta_{1}(t, \,Y^{\theta_{1}}_{t}, \,\nu^{\theta_{1}}_{t})\,dB_{t}, \;\;\;Y^{\theta_{1}}_{0}\in L^{p}(\mathbb{P}),\nonumber\\
%\label{comparesigma2}
&dX^{\sigma_{2}}_{t}=\big(a(t) X^{\sigma_{2}}_{t}+\beta(t,\,\EE X^{\sigma_{2}}_{t})\big)dt+\sigma_{2}(t, X^{\sigma_{2}}_{t}, \mu^{\sigma_{2}}_{t})dB_{t}, \;\;\;X^{\sigma_{2}}_{0}\in L^{p}(\mathbb{P}),\nonumber\\
%\label{comparetheta2}
&dY^{\theta_{2}}_{t}\,=\big(a(t)\,Y^{\theta_{2}}_{t}\,+\beta(t,\,\EE Y^{\theta_{2}}_{t})\big)dt+\theta_{2}(t, \,Y^{\theta_{2}}_{t}, \,\nu^{\theta_{2}}_{t})\,dB_{t}, \;\;\;Y^{\theta_{2}}_{0}\in L^{p}(\mathbb{P}),\nonumber
\end{align}
%\begin{align}%\label{comparesigma1}
%&dX^{\sigma_{1}}_{t}=(\alpha X^{\sigma_{1}}_{t}+\beta)dt+\sigma_{1}(t, X^{\sigma_{1}}_{t}, \mu^{\sigma_{1}}_{t})dB_{t}, \;\;\;X^{\sigma_{1}}_{0}\in L^{p}(\mathbb{P}),\nonumber\\
%%\label{comparetheta1}
%&dY^{\theta{1}}_{t}\,=(\alpha\, Y^{\theta{1}}_{t}\,+\beta)dt+\theta_{1}(t, \,Y^{\theta{1}}_{t}, \,\nu^{\theta{1}}_{t})\,dB_{t}, \;\;\;Y^{\theta{1}}_{0}\in L^{p}(\mathbb{P}),\nonumber\\
%%\label{comparesigma2}
%&dX^{\sigma_{2}}_{t}=(\alpha X^{\sigma_{2}}_{t}+\beta)dt+\sigma_{2}(t, X^{\sigma_{2}}_{t}, \mu^{\sigma_{2}}_{t})dB_{t}, \;\;\;X^{\sigma_{2}}_{0}\in L^{p}(\mathbb{P}),\nonumber\\
%%\label{comparetheta2}
%&dY^{\theta_{2}}_{t}\,=(\alpha\, Y^{\theta_{2}}_{t}\,+\beta)dt+\theta_{2}(t, \,Y^{\theta_{2}}_{t}, \,\nu^{\theta_{2}}_{t})\,dB_{t}, \;\;\;Y^{\theta_{2}}_{0}\in L^{p}(\mathbb{P}),\nonumber
%\end{align}
and if $\sigma_{1}$ and $\sigma_{2}$ satisfy Assumption~\ref{AssumptionII}-(2), (3), $X^{\sigma_{1}}_{0}\preceq_{\,cv}Y^{\theta_{1}}_{0}\preceq_{\,cv}X^{\sigma_{2}}_{0}\preceq_{\,cv}Y^{\theta_{2}}_{0}$ and 
\begin{equation}\label{relation}
\sigma_{1}\preceq \theta_{1}\preceq\sigma_{2}\preceq \theta_{2},
\end{equation} 
 %for every $(t, x, \mu)\in\mathbb{R}_{+}\times\RD\times{\color{black}\mathcal{P}_{1}(\RD)}$,
%\[
%\sigma_{1}(t, x, \mu)\preceq \theta_{1}(t, x, \mu)\preceq\sigma_{2}(t, x, \mu)\preceq %\theta_{2}(t, x, \mu),\]
then we have the following two types of inequalities:

$-$  Convex bounding
\begin{equation}
\label{bounding}
%&\hspace{0.7cm}-  \text{Convex bounding}:&\nonumber\\
\;\;\begin{cases}
\EE F(X^{\sigma_{1}}) \leq\EE F(Y^{\theta_{1}})\leq \EE F(X^{\sigma_{2}}), \\
 \EE G\big(X^{\sigma_{1}}, (\mu^{\sigma_{1}}_{t})_{t\in[0, T]}\big)\leq\EE G\big(Y^{\theta_{1}}, (\nu^{\theta_{1}}_{t})_{t\in[0, T]}\big)\leq \EE G\big(X^{\sigma_{2}}, (\mu_{t}^{\sigma_{2}})_{t\in[0, T]}\big),\end{cases}
\end{equation}

$-$  Convex partitioning
\begin{equation}
\label{partitioning}
%&\hspace{0.7cm} -  \text{Convex partitioning}:&\nonumber\\ &\hspace{1.5cm}
\begin{cases}
\EE F(Y^{\theta_{1}})\leq \EE F(X^{\sigma_{2}})\leq\EE F(Y^{\theta_{2}}),\\
\EE G\big(Y^{\theta_{1}}, (\nu^{\theta_{1}}_{t})_{t\in[0, T]}\big)\leq \EE G\big(X^{\sigma_{2}}, (\mu^{\sigma_{2}}_{t})_{t\in[0, T]}\big)\leq \EE G\big(Y^{\theta_{2}}, (\nu^{\theta_{2}}_{t})_{t\in[0, T]}\big)\end{cases}
\end{equation}
\noindent for any two applications $F$ and $G$ satisfying conditions of Theorem~\ref{mainthmconv}. 

%\begin{figure}[!h]
%\centering
%\includegraphics[width=14cm]{convex_figure.eps}
%\caption{Left:  convex partitioning. Right : convex bounding.}
%\label{partitionbound}
%\end{figure}

Remark that Assumption~\ref{AssumptionII}-(2) and (3) are only made on $\sigma$ of the equation of $X$. Consequently, in this application, we can choose two ``simple'' functions $\sigma_{i}, \;i=1,2$ to construct the convex partitioning and convex bounding. For example, we can choose two convex functions $\sigma_{i}, \;i=1,2$  which do not depend on $\mu$ and satisfy (\ref{relation}). In this a case, the results in (\ref{bounding}) and (\ref{partitioning}) make a comparison between  a McKean-Vlasov equation and  regular Brownian diffusions, the latter ones being much easier to simulate.
%\begin{equation}
%\[\begin{cases}
%dX_t^{x, \alpha}=\big[a(t) X_t^{x, \alpha}+ \beta(t, \mathbb{E}X_t^{x, \alpha})\big]dt + \sigma(t, X_t^{x, \alpha}, \alpha_t, \mathbb{E}X_t^{x, \alpha})dB_t, \quad X_0\in L^{2+\varepsilon}(\mathbb{P})\\
%\text{with the control process }\alpha_t=\eta(t) \cdot X_t^{x, \alpha} + h(t, \mathbb{E}X_t^{x, \alpha})
%\end{cases}\]
%If $\alpha, \eta: [0, T]\rightarrow \RD$ are H\"older, $\beta, h: [0, T]\times \RD\rightarrow \RD$ are H\"older in $t$ and Lipschitz in $x$ and 
%\[\sigma: (t, x, y, z)\in [0, T]\times \RD\times \RD\times \RD \mapsto \sigma(t, x, y, z)\in \mathbb{M}_{d\times q} \]
%is H\"older in $t$, Lipschitz in $x, y, z$ and convex in $x$

%\end{equation}
%with the control process $\alpha_t=\eta_t \cdot X_t + h(t, \mathbb{E}X_t)$

\subsection{Application in stochastic control problem and mean field games} \label{application2}

In this subsection, we give two examples to explain how to apply Theorem~\ref{mainthmconv} and Theorem~\ref{mainthmconvdual} in the framework of the stochastic control problem and mean field games. 

\smallskip
$\blacktriangleright\;$The construction of this first example for the stochastic control problem is  based on the dynamic and  cost function in \cite[Section 4]{MR3045029} but the same idea can be applied to other similar examples. 
Consider the following two one-dimensional McKean-Vlasov dynamics: 
\begin{align}
&dX^{x, \alpha}_{t}=\big[a_{t}X^{x, \alpha}_{t}+\bar{a}_{t}\EE X^{x, \alpha}_{t}+b_t\alpha_t+\beta_t\big]dt+\sigma dB_t, \quad X_0=x\in\RR\nonumber\\
&dY^{x, \alpha}_{t}\,=\big[a_{t}Y^{x, \alpha}_{t}+\bar{a}_{t}\EE Y^{x, \alpha}_{t}+b_t\alpha_t+\beta_t\big]dt+\theta(t, Y^{x, \alpha}_t, \nu_t)\, dB_t,\quad Y_0=x\nonumber
\end{align}
equipped with the respective cost function $\mathcal{J}^{X}$ and $\mathcal{J}^{Y}$ defined by 
\begin{align}
\mathcal{J}^{X}(x, \alpha)&=\EE \left[\int_{0}^{T}\Big[ \frac{1}{2}\big(m_t X^{x, \alpha}_t +\bar{m}_t \EE X^{x, \alpha}_t\big)^{2}+\frac{1}{2}n_t \alpha_{t}^{2}\Big]dt +\frac{1}{2}\big( qX^{x, \alpha}_T +\bar{q}\,\EE \! X^{x, \alpha}_T\big)^{2}\right]\nonumber\\
\mathcal{J}^{Y}(x, \alpha)&=\EE \left[\int_{0}^{T}\Big[ \frac{1}{2}\big(m_t Y^{x, \alpha}_t\, +\bar{m}_t \EE Y^{x, \alpha}_t\,\big)^{2}+\frac{1}{2}n_t \alpha_{t}^{2}\Big]dt +\frac{1}{2}\big( qY^{x, \alpha}_T\, +\bar{q}\,\EE Y^{x, \alpha}_T\,\big)^{2}\right]\nonumber
\end{align}
with   admissible controls  $\alpha=(\alpha_t)_{t\in[0, T]}$ (see~\cite[Section~2]{MR3045029}, where for every $t\in[0, T]$, $\nu_t$ is the probability distribution of $Y^{x, \alpha}_t$, $a_t, \bar{a}_t, b_t, \beta_t, m_t, \bar{m}_t, n_t$ are deterministic $\rho$-H\"older continuous functions of $t\in[0, T]$, $q$ and $\bar{q}$ are deterministic constants and the function $t\mapsto n_t$ is positive. Assume moreover that $\theta$ satisfies Assumption~\ref{AssumptionI} and
\[
\forall\, (t, x, \mu)\in [0, T]\times \RD\times \mathcal{P}_{2}(\RD),\quad 0\le \theta(t, x, \mu)\leq \sigma.
\]
We know from \cite[Proposition 4.1 and the remark that follows]{MR3045029} that the optimal control   minimizing $\mathcal{J}^{X}(x, \cdot)$ has a closed form  given by
\begin{equation}\label{oc}
\alpha^{*,X}(t,X_t)= -\frac{b_t}{n_t}\eta_t X_t-\frac{b_t}{n_t}\chi_t
\end{equation}
where $(\eta_t)_{0\leq t\leq T}$ and $(\chi_t)_{0\leq t\leq T}$ are determined by solving a Riccati equation depending only on $b_t, n_t, a_t, \bar{a}_t, m_t, \bar{m}_t, \beta_t$, $q$ and $\bar{q}$ (see (57) still in~\cite{MR3045029}). If the coefficient $\frac{b_t}{n_t}\eta_t $ and $\frac{b_t}{n_t}\chi_t$ in~\eqref{oc} are still $\rho$-H\"older continuous, then Theorem~\ref{mainthmconvdual} directly implies  
 \textcolor{black}{\begin{equation}\label{oj}
\inf_{\alpha}\mathcal{J}^{X}(x, \alpha)\!=\! \mathcal{J}^{X}\big(x, (\alpha^{*,X}(t,X_t))_{t\in [0,T]}\big)\!\geq\! \mathcal{J}^{Y}(x,  (\alpha^{*,X}(t,Y_t))_{t\in [0,T]})\!\geq\! \inf_{\alpha}\mathcal{J}^{Y}(x, \alpha).
\end{equation}
}
Consequently, if we define the value function by 
\[v^{X}(x)\coloneqq \inf_{\alpha}\mathcal{J}^{X}(x, \alpha) \quad \text{and}\quad v^{Y}(x)\coloneqq \inf_{\alpha}\mathcal{J}^{Y}(x, \alpha),\]
then~\eqref{oc} directly implies $v^{X}(x)\geq v^{Y}(x)$.   \textcolor{black}{Moreover, it follows from Corollary~\ref{cormain} that the function $x\mapsto v^X(x)$ is convex. 
}
%\textcolor{red}{dire qu'on suit le cadre de [MR3045029] mais que l'on pourrait sans doute prendre un $\sigma$ plus general.}
%\vspace{4cm}
%We discuss in this subsection an optimal control problem of the following McKean-Vlasov dynamics
%\[dX^{\sigma}_t=\big(\alpha_t+b(\EE X_t)\big)dt + \sigma (X_t, \mu_t)dB_t, \quad X_0=x \in {\color{red}\RR ?}\]
%equipped with the cost function defined by 
%\[\mathcal{J}\big(\alpha\big)=\mathbb{E}\Big[\frac{1}{2}\int_{0}^{T}\left|\alpha_t\right|^{2}dt\Big]+\int_{0}^{T}\mathcal{F}(\mu_t)dt + \mathcal{G}(\mu_T)\]
%with an admissible control $\alpha=(\alpha_t)_{t\in[0, T]}$. 
%We consider one type of control $\alpha_t=\eta_t X_t +h_t$, which is optimal in the linear-quadratic setting if $h_t$ is a linear function of $\EE X_t$ (see e.g. \cite{MR3045029}). If $\sigma, \mathcal{F}$ and $\mathcal{G}$ satisfy the conditions in  Theorem~\ref{mainthmconv}, then we can compare the value of $\alpha\mapsto\EE \mathcal{J}(\alpha)$ for dynamics $X^{\sigma}$ with different coefficients $\sigma$ and different initial values, from which we can get a direct comparison of the value function $v(x)\coloneqq \sup_{\alpha}\EE \mathcal{J}(\alpha)$ for these dynamics $X^{\sigma}$s.  

$\blacktriangleright\;$ Similarly, in the framework of mean field games, a McKean-Vlasov equation appears when the number of agents tends to infinity. If we consider the following one-dimensional (limiting) dynamics 
\[dX_t=\alpha_t dt + \sigma (\mu_t)dB_t, \quad X_0\in L^{p}(\mathbb{P}),\; p\geq 2\]
where $\sigma \ge 0$ is ${\cal W}_p$-Lipschitz continuous and $(\mu_t)_{t\in[0,T]}$ is a flow of square integrable probability meausres on $\RR$,  associated with the following cost function 
\[\mathcal{J}^{X}(\alpha)=\mathbb{E}\left[\frac{1}{2}\Big|c_g X_T + g\Big(\int x \mu_{T}(dx)\Big)\Big|^{2}+\int_0^T \Big[\frac{1}{2}\big|c_f X_t + f\Big(\int x \mu_{t}(dx)\Big)\big|^{2}+\frac{1}{2}\left|\alpha_t\right|^{2}\Big]dt\right].
\]
Theorems in this paper may provide some information on the value of the game. 
Assume that the Lasry-Lions monotonicity condition is satisfied, then the equilibrium of  the mean field game is unique with a McKean-Vlasov  dynamics ($\mu_t$ is the distribution of $X_t$) given by 
\begin{equation}\label{sys1}
dX_t=-\big[ \eta_t X_t + h\big(t, \EE X_t \big)\big]dt + \sigma (\mu_t)dB_t
\end{equation}
by taking $\alpha_t=-\big[\eta_t X_t + h\big(t, \mathbb{E}X_t\big)\big]$ (see e.g. \cite[Section 3.5.2]{MR3753660}). In~\eqref{sys1}, $(\eta_t)_{t\in[0, T]}$ is the solution of a Riccati equation depending only on $c_f$ and $c_g$ and $h$ depends on $f, g, c_f$ and $c_g$. If $\sigma$ is non-decreasing w.r.t. the convex order like in the previous example, results in this paper allow us to compare the value of the cost function $\mathcal{J}(\alpha)$ of the system~\eqref{sys1} with that of  the following system $(Y_t)_{t\in[0, T]}$ %having a different coefficient $\theta$ and different initial value $Y_0$: %defined by
\begin{equation}\label{sys2}
dY_t=-\big[ \eta_t Y_t + h\big(t, \EE Y_t \big)\big]dt + \theta (\nu_t)dB_t, \quad Y_0\in L^{p}(\mathbb{P}),\; p\geq2,
\end{equation}
when 
$$
0\le \sigma \le \theta \quad\mbox{ and }\quad X_0\conright Y_0.
$$ In fact, if we denote by $\nu_t$ the probability distribution of $Y_t$ for every $t\in[0, T]$, by applying Theorem~\ref{mainthmconv} and~\eqref{equalE}, we know that $\mu_t \preceq \nu_t$ for every $t\!\in [0,T]$ so that $\EE X_t=\EE Y_t$. Thus $g\big(\int x \mu_{T}(dx)\big)=g\big(\int x \,\nu_{T}(dx)\big)$ and $f\big(\int x \mu_{t}(dx)\big)=f\big(\int x \,\nu_{t}(dx)\big),\; t\in[0, T]$. Let $\alpha \mapsto \mathcal{J}^{X}(\alpha)$ and $\alpha \mapsto \mathcal{J}^{Y}(\alpha)$ denote  the cost function for the system~\eqref{sys1} and~\eqref{sys2} respectively.  Then the difference between $\mathcal{J}^{X}(\alpha)$ and $\mathcal{J}^{Y}(\alpha)$ only depends on   terms containing $\EE [X_t^{2}]$, $\EE [Y_t^{2}]$, $t\!\in [0, T]$, that is,  on   the variances of $X_t$ and $Y_t$ since  they have the same mean.

\section{Convex order results for the Euler scheme}\label{eulerconv}

In this section, we will discuss the convex order results for the random variables $\bar{X}_{t_{m}}^{M}$ and $\bar{Y}_{t_{m}}^{M}, m=0, \ldots, M$ defined by the Euler scheme~(\ref{defeulerx}) and~(\ref{defeulery}). In order to simplify the notations, we rewrite~(\ref{defeulerx}) and~(\ref{defeulery}) by setting 
$$
\bar{X}_{m}\coloneqq \bar{X}_{t_{m}}^{M},\quad\bar{Y}_{m}\coloneqq \bar{Y}_{t_{m}}^{M},\quad\bar{\mu}_{m}\coloneqq \bar{\mu}_{t_{m}}^{M}\quad \mbox{and}\quad\bar{\nu}_{m}\coloneqq \bar{\nu}_{t_{m}}^{M}.
$$ It reads %Let $M\in\mathbb{N}^{*}$ and let $ h= \frac{T}{M}$. For $m=0, \ldots, M$, we write $t_{m}^{M}\coloneqq h\cdot m=\frac{T}{M}\cdot m$.
%We rewrite the Euler scheme of $(X_{t})_{t\in[0, T]}$ and $(Y_{t})_{t\in[0, T]}$ defined in~(\ref{defxtyt}) as follows,
\begin{align}\label{EulerX}
&\bar{X}_{m+1}={\color{black}b_{m}(\bar{X}_{m}, \bar{\mu}_{m})}+\sigma_{m}(\bar{X}_{m}, \bar{\mu}_{m})Z_{m+1}, \quad \bar{X}_{0}=X_{0},\\
%&\bar{X}_{m+1}=\bar{\alpha}_{m}\bar{X}_{m}+\bar{\beta}_{m}\,\EE\bar{X}_{m} +\bar{\gamma}_{m}+\sigma_{m}(\bar{X}_{m}, \bar{\mu}_{m})Z_{m+1}, \hspace{1cm} \bar{X}_{0}=X_{0},\\
\label{EulerY}&\bar{Y}_{m+1}\,=\,{\color{black}b_{m}(\bar{Y}_{m}\,, \bar{\nu}_{m}\,)}+\,\theta_{m}(\bar{Y}_{m},\, \bar{\nu}_{m}\,)Z_{m+1}, \quad \;\bar{Y}_{0}=Y_{0},
%&\bar{Y}_{m+1}\,=\,\bar{\alpha}_{m}\bar{Y}_{m}\,+\bar{\beta}_{m}\,\EE\bar{Y}_{m}\,+\bar{\gamma}_{m}+\theta_{m}(\bar{Y}_{m},\, \bar{\nu}_{m}\,)Z_{m+1}, \hspace{1.1cm} \bar{Y}_{0}=Y_{0},
\end{align}
where 
%\begin{equation}\label{defalpham}
%\bar{\alpha}_{m}\coloneqq h\,\alpha(t_{m}) +\mathbf{I}_{d}, \:\bar{\beta}_{m}\coloneqq h\,\beta(t_{m})\:\text{and}\: \bar{\gamma}_{m}\coloneqq h\,\gamma(t_{m}),
%\end{equation}
%and 
for every $m=0, \ldots, M$, %$\bar{\mu}_{m}$ and $\bar{\nu}_{m}$ respectively denote the probability distributions of $\bar{X}_{m}$ and $\bar{Y}_{m}$,  and
\begin{align}\label{defbm}
{\color{black}b_m(x, \mu)\coloneqq x+h\cdot b(t_m, x, \mu)},\;\;
\sigma_{m}(x, \mu)\coloneqq \sqrt{ h}\cdot\sigma(t_{m}, x, \mu),\;\; \theta_{m}(x, \mu)\coloneqq \sqrt{ h}\cdot\theta(t_{m}, x, \mu).
\end{align}
\textcolor{black}{
First note that it follows from Proposition~\ref{cvgeuler}$(a)$ that
\[
\forall\, m=0,\ldots, M,\quad \bar \mu_m ,\; \bar \nu_m \! \in  \mathcal{P}_{p}(\RD)
\]
(hence lie in $\mathcal{P}_{1}(\RD)$).}  Then it follows from Assumption~\ref{AssumptionII} that $X_{0}$, $Y_{0}$, ${\color{black}b_{m}}, \sigma_{m}, \theta_{m}, m=0, \ldots, M,$ satisfy its discrete time counterpart. 
\begin{manualtheorem}{II$_{disc}$}\label{IIprime}
{\color{black}$(1)$ The function $b_m,\, m=0, \dots, M,$ are affine in $x$ and }% having the form %can be written as 
%\[b_m(x, \mu)=a_m x+\beta_{m}(\int_{\RD}\xi\mu(d\xi))\vspace{-0.1cm}\] 
%with $a_m\in\mathbb{M}_{d\times d}$ and $\beta_m: \RD\rightarrow\RD$ so that by ~\eqref{equalE}},
\begin{equation}\label{assmbm}
\color{black}\forall\, \;\mu, \nu\in\mathcal{P}_{p}(\RD), \;\mu\conright\nu\quad \quad b_{m}(x, \mu)=b_m(x, \nu).
\end{equation}
\vspace{-0.5cm}

\noindent$(2)$ The functions $\sigma_m,\; m=0, \dots, M,$ are convex in $x$ :%\textit{Convex in $x$ :}
\begin{equation}\label{conm1}
\forall\, \, x,y\in\RD, \forall\lambda\in[0,1],\hspace{0.7cm}\sigma_{m}\big(\lambda x+(1-\lambda)y, \mu\big)\preceq\lambda\sigma_{m}(x, \mu)+(1-\lambda)\sigma_{m}(y, \mu).  
\end{equation}
$(3)$ The functions $\sigma_m,\, m=0, \dots, M,$ are non-decreasing in $\mu$ with respect to the convex order :%\textit{Non-decreasing in $\mu$ with respect to the convex order: }
\begin{equation}\label{conm2}
\forall\, \, \mu, \nu\in\mathcal{P}_{p}(\RD), \;\mu\conright\nu, \quad\quad\quad \sigma_{m}(x, \mu)\preceq\sigma_{m}(x, \nu).
\end{equation}
$(4)$ We have the following order between $\sigma_m$ and $\theta_m$, $m=0, \dots, M$ :%\textit{Order of $\sigma_{m}$ and $\theta_{m}$: }
\begin{equation}\label{conm3}
\forall(x, \mu)\in\mathbb{R}^{d}\times \mathcal{P}_{p}(\RD), \hspace{1.3cm} \sigma_{m}(x, \mu)\preceq\theta_{m}(x, \mu). 
\end{equation}
$(5)$ $\bar{X}_{0}\conright\bar{Y}_{0}$~(\footnote{Since $\bar{X}_{0}=X_{0}$ and $\bar{Y}_{0}=Y_{0}$ \big(see the definition~\eqref{defeulerx} and~\eqref{defeulery}\big).}).

\end{manualtheorem}

\textcolor{black}{At this stage let us mention that we will extensively use the following elementary  characterization of convex ordering between two integrable $\RD$-valued random variables or their distributions.}
\begin{lem}[Lemma A.1 in \cite{alfonsi:hal-01589581}]\label{onlylineargrowth}
Let $\mu, \nu\in\mathcal{P}_{1}(\RD)$. We have $\mu\conright\nu$ if and only if for every convex function $\varphi: \RD \rightarrow \RR$ with (at most)  linear growth in the sense that there exists a real constant $C>0$ such that, for every $x\!\in\RD$,  $|\varphi(x)|\le C(1+|x|))$,
\[
\int_{\RD}\varphi(x)\mu(dx)\leq\int_{\RD}\varphi(x)\nu(dx).
\]
\end{lem}
\noindent This characterization allows us to restrict the proofs to convex functions with linear growth to establish the convex ordering.

\medskip
The main result of this section is the following proposition. 
\begin{prop}\label{funconvexEuler}
Under Assumption~\ref{AssumptionI} and~\ref{AssumptionII}, for any convex function $F:(\mathbb{R}^{d})^{M+1}\rightarrow \RR$ with  $p$-polynomial growth in the sense that 
\begin{equation}\label{rpolygrowth2}
\forall\, \, x=(x_{0}, \ldots, x_{M})\in(\RD)^{M+1},\; \exists\, C>0, \; \text{ such that } \; \left|F(x)\right|\leq C\big(1+\sup_{0\leq i\leq M}\left|x_{i}\right|^{p}\big),
\end{equation}
 we have
$\EE F(\bar{X}_{0}, \ldots, \bar{X}_{M})\leq \EE F(\bar{Y}_{0}, \ldots, \bar{Y}_{M}).$
\end{prop}

Before proving Proposition~\ref{funconvexEuler}, we first show in the next section  that the Euler scheme defined in~(\ref{EulerX}) and~(\ref{EulerY}) propagates the marginal convex order step by step, i.e. $\bar{X}_{m}\conright \bar{Y}_{m},$ for any fixed $m\in\{0, \ldots, M\}$ by a forward induction.% established in Section~\ref{marginalorder}. 

%\bar{\alpha}x+\bar{\beta}
\subsection{Marginal convex order for the Euler scheme}\label{marginalorder}
Let 
%\begin{align}
$\mathcal{C}_{cv} (\RD,\RR)\coloneqq\big\{\varphi: \RD\rightarrow\RR \text{ convex function}\big\}$. For every {\color{black}$m=0, \ldots, M-1$}, we define an operator $Q_{m+1}\!: \mathcal{C}_{cv}(\mathbb{R}^{d}, \mathbb{R})\rightarrow\mathcal{C}\big(\RD\times \mathcal{P}_{1}(\RD)\times\MDQ, \RR\big)$ associated with $Z_{m+1}$ defined in ~(\ref{EulerX}) and~(\ref{EulerY}) by 
\begin{equation}\label{defqm}
 (x,\mu, u)\!\in \RD\times \mathcal{P}_{1}(\RD)\times\MDQ \longmapsto (Q_{m+1}\,\varphi) (x,\mu, u)\coloneqq \mathbb{E}\Big[\varphi\big(b_{m} (x, \mu)+u Z_{m+1}\big)\Big]. 
\end{equation}
%\begin{equation}\label{defqm}
%(x,\mu, u)\!\in \RD\times \mathcal{P}_{1}(\RD)\times\MDQ \longmapsto (Q_{m}\,\varphi) (x,\mu, u)\coloneqq \mathbb{E}\Big[\varphi\big(\bar{\alpha}_{m}x+\bar{\beta}_{m}\int_{\RD}\xi\mu(d\xi)+\bar{\gamma}_{m}+u Z_{m}\big)\Big]. 
%\end{equation}
For every $m=0, \ldots, M$, let $\mathcal{F}_{m}$ %$\mathcal{F}_{m}\coloneqq \sigma(X_{0}, Z_{1}, \ldots, Z_{m})$, 
denote the $\sigma$-algebra generated by $(X_{0}, {\color{black}Y_0}, \, Z_{1}, \ldots, Z_{m})$. The main result in this section is the following. 

\begin{prop}\label{marginalconvex}
Let $(\bar{X}_{m})_{m=0, \ldots, M}$, $(\bar{Y}_{m})_{m=0, \ldots, M}$ be random variables defined by~(\ref{EulerX}) and~(\ref{EulerY}). Under Assumption~\ref{IIprime}, we have %If for every $m=0, \ldots, M$, $\sigma_{m}$ and $\theta_{m}$ satisfy Assumption 2',  then  
\[\bar{X}_{m}\conright\bar{Y}_{m}, \quad m=0, \ldots, M.\]
\end{prop}
The proof of Proposition~\ref{marginalconvex} relies on the following  two lemmas whose proofs are postponed to Appendix~\ref{appB}.
\begin{lem}[Revisited Jensen's Lemma]\label{revisiedjensen}
 Let $\varphi\in C_{cv}(\RD, \RR)$ \textcolor{black}{with linear growth and let $\mu\!\in\mathcal{P}_{1}(\RD)$}. Then, for every $m=1, \dots, M$,
%Let $\varphi: \RD\rightarrow\RR$ be a convex function having an $r$-polynomial growth.  %{\color{red}such that $\mathbb{E}\varphi(uZ)<+\infty$ for every $u\in\mathbb{R}^{d}$}
%Then,
\begin{enumerate}[$(i)$] 
\item   the function $(x,u)\mapsto(Q_m\varphi )(x, \mu, u)$ is (finite and) convex.
\item for any fixed $x\!\in \RD$, the function $u\mapsto(Q_m\varphi)(x, \mu, u)$ attains its minimum at $\mathbf{0}_{d\times q}$, where $\mathbf{0}_{d\times q}$ is the zero-matrix of size $d\times q$, 
\item for any fixed $x\!\in \RD$, the function $u\mapsto(Q_m\varphi)(x, \mu, u)$ is non-decreasing with respect to the partial order of $d\times q$ matrix (\ref{matrixorder}).
\end{enumerate}
\end{lem}

\begin{lem}\label{marglemma}
\textcolor{black}{Let $\varphi\in\mathcal{C}_{cv}(\RD, \RR)$ with linear growth}.
Then for a fixed $\mu\in\mathcal{P}_{1}(\mathbb{R}^{d})$,  the function $x\mapsto \mathbb{E}\Big[\varphi\big(b_{m}(x, \mu)+\sigma_{m}(x, \mu)  Z_{m+1}\big)\Big]$ is convex with linear growth for every $m=0, \ldots, M-1$. 
\end{lem}
\begin{proof}[Proof of Proposition~\ref{marginalconvex}]
Assumption~\ref{IIprime} directly implies $\bar{X}_{0}\conright \bar{Y}_{0}$.  
%It is obvious that $X_{0}\conright Y_{0}$ since $\bar{X}_{0}=\bar{Y}_{0}=x$.  
Assume $\bar{X}_{m}\conright \bar{Y}_{m}$ or, equivalently, $\bar \mu_m\preceq_{cv}\bar \nu_m$.  \textcolor{black}{Let $\varphi: \RD \to \mathbb{R}$ be a convex function with linear growth}. Then,
\begin{align}
\EE [\varphi(\bar{X}_{m+1})] &= \EE\Big[ \varphi\big( b_{m}(\bar{X}_{m}, \bar{\mu}_{m})+\sigma_{m}(\bar{X}_{m}, \bar{\mu}_{m})Z_{m+1}\big)\Big]\nonumber\\
&= \EE \Big[ \EE \big[\varphi\big( b_{m}(\bar{X}_{m}, \bar{\mu}_{m})+\sigma_{m}(\bar{X}_{m}, \bar{\mu}_{m})Z_{m+1}\big)\mid \mathcal{F}_{m}   \big]\Big]\nonumber\\
&=\int_{\mathbb{R}^{d}}\bar{\mu}_{m}(dx)\,\EE \Big[ \varphi \big( b_{m}(x, \bar{\mu}_{m})+\sigma_{m}(x, \bar{\mu}_{m})Z_{m+1}\big)\Big]\nonumber\\
&\hspace{0.5cm}\text{(the integrability is due to Proposition~\ref{cvgeuler} and Lemma~\ref{marglemma}})\nonumber\\
&\leq \int_{\mathbb{R}^{d}}\bar{\mu}_{m}(dx)\,\EE \Big[  \varphi \big(b_{m}(x, \bar{\nu}_{m})+\sigma_{m}(x, \bar{\nu}_{m})Z_{m+1}\big)\Big]\nonumber\\ 
&\hspace{0.5cm}\text{(by Assumption~(\ref{assmbm}), (\ref{conm2}) and Lemma~\ref{revisiedjensen}, since $\bar{\mu}_{m}\conright\bar{\nu}_{m}$)}\nonumber\\
&\leq \int_{\mathbb{R}^{d}}\bar{\nu}_{m}(dx)\,\EE \Big[  \varphi \big(b_{m}(x, \bar{\nu}_{m})+\sigma_{m}(x, \bar{\nu}_{m})Z_{m+1}\big)\Big]\nonumber\\
&\hspace{0.5cm}\text{(by Lemma~\ref{marglemma},  since $\bar{\mu}_{m}\conright\bar{\nu}_{m}$)}\nonumber\\
&\leq \int_{\mathbb{R}^{d}}\bar{\nu}_{m}(dx)\,\EE \Big[  \varphi \big(b_{m}(x, \bar{\nu}_{m})+\theta_{m}(x, \bar{\nu}_{m})Z_{m+1}\big)\Big]\nonumber\\
&\hspace{0.5cm}\text{\big(by Assumption~(\ref{conm3}) and Lemma~\ref{revisiedjensen}}\big)\nonumber\\
&=\EE [\varphi(\bar{Y}_{m+1})]\nonumber.
\end{align}
Thus $\bar{X}_{m+1}\conright \bar{Y}_{m+1}$ by applying Lemma~\ref{onlylineargrowth}. One concludes by a forward induction. 
\end{proof}

\begin{proof}[Proof of Corollary~\ref{cormain}]$(a)$ This is a straightforward consequence of Theorem~\ref{mainthmconv}$(a)$ by setting, $f$ being a convex function with linear growth and $t\!\in [0,T]$,  $F(\alpha)= f(\alpha(t))$, $\alpha \!\in {\cal C}([0,T], \RR)$. 

A direct proof is as follows: we only need to prove $\mu_T\conright\nu_{\,T}$ since for any $t\in[0, T]$ the proof for $\mu_t\conright\nu_{t}$ is the same by considering $(X_s)_{s\in[0, t]}$ and $(Y_s)_{s\in[0, t]}$. Let $\varphi: \RD\rightarrow \RR$ be a convex function with linear growth. Proposition~\ref{marginalconvex} implies that for every $h=\frac{T}{M}\geq 0$, $\bar{X}_{M}\conright \bar{Y}_{M}$. Thus for every $h>0$, $\EE \varphi (\bar{X}_{M})\leq \EE \varphi (\bar{Y}_{M})$. 
It follows from Proposition~\ref{cvgeuler} that
%\[\vertii{\bar{X}_{M}-X_T}_{p}\vee\vertii{\bar{Y}_{M}-Y_T}_{p}\rightarrow \;\text{as}\; h\rightarrow0.\]
%Then we have 
$\EE \varphi (\bar{X}_{M})\rightarrow \EE \varphi (X_T)$ and $\EE \varphi (\bar{Y}_{M})\rightarrow \EE \varphi (Y_T)$ as $h\rightarrow 0$ since we assumed $p\geq 2$. Hence $\EE \varphi (X_T)\leq \EE \varphi (Y_T)$ by letting $h\rightarrow 0$, i.e. $X_T\conright Y_T$ by applying Lemma~\ref{onlylineargrowth}.

\noindent $(b)$ For every $x, y\!\in \RD$ and $\lambda\!\in [0,1]$, we have $\delta_{\lambda x + (1-\lambda)y}\conright \lambda \delta_{x}+ (1-\lambda)\delta_{y}$ and one  concludes  by applying Theorem~\ref{mainthmconv}  with $\theta=\sigma$, $X_0\sim \delta_{\lambda x + (1-\lambda)y}$ and $Y_0\sim \lambda \delta_{x}+ (1-\lambda)\delta_{y}$.
\end{proof}

%Now we prove the marginal convex order result in Theorem~\ref{mainthmconv} - $(a)$. 
The next proposition prove that we can dissociate the assumption on the convexity and monotonicity in Assumption~\ref{AssumptionII} - (2), (3) to obtain the same marginal convex order as in Corollary~\ref{cormain} - $(a)$. 
%assume separately that $\sigma$ is convex and $\theta$ is monotone with respect to the convex order to obtain 
This seems to be specific to marginal convex ordering.
\begin{prop}\label{marginal2}Let $\mu_t, \;\nu_t,\;t\in [0, T]$,  respectively denote the marginal distributions of the solution processes $X$ and $Y$ in (\ref{defconvx}), (\ref{defconvy}). 
If we replace Assumption~\ref{AssumptionII}-(3) by the following condition:

\noindent$(3')$ For every fixed $(t,x)\in\RR_{+}\times\RD$, the function $\theta(t, x, \cdot)$ is non-decreasing in $\mu$ with respect to the convex order in the sense that
\begin{equation}\label{conmtheta}
\forall\, \, \mu, \nu\in\mathcal{P}_{p}(\RD), \quad\mu\conright\nu \Longrightarrow \theta(t, x, \mu)\conright\theta(t, x, \nu),
\end{equation}

%\vspace{-0.5cm}

\noindent then for every $t\in[0, T]$, $\mu_t\conright\nu_t$. 
\end{prop}
\begin{proof}[Proof of Proposition~\ref{marginal2}]
First, remark that the proofs of Lemma~\ref{revisiedjensen} and~\ref{marglemma} do not depend on the condition (\ref{assumptionorder}) and (\ref{conmtheta}) and the induction step of Proposition~\ref{marginalconvex} remains true under the new assumption $(3')$.   Assume that $\bar{X}_{m}\conright\bar{Y}_{m}$.  If we consider a convex function $\varphi: \RD\rightarrow \RR$ with linear growth, then 
\begin{align}
\EE [\varphi(\bar{X}_{m+1})] &= \EE\Big[ \varphi\big( b_{m}(\bar{X}_{m}, \bar{\mu}_{m})+\sigma_{m}(\bar{X}_{m}, \bar{\mu}_{m})Z_{m+1}\big)\Big]\nonumber\\
&= \EE \Big[ \EE \big[\varphi\big( b_{m}(\bar{X}_{m}, \bar{\mu}_{m})+\sigma_{m}(\bar{X}_{m}, \bar{\mu}_{m})Z_{m+1}\big)\mid \mathcal{F}_{m}   \big]\Big]\nonumber\\
&=\int_{\mathbb{R}^{d}}\bar{\mu}_{m}(dx)\,\EE \Big[ \varphi \big( b_{m}(x, \bar{\mu}_{m})+\sigma_{m}(x, \bar{\mu}_{m})Z_{m+1}\big)\Big]\nonumber\\
&\leq \int_{\mathbb{R}^{d}}\bar{\nu}_{m}(dx)\,\EE \Big[ \varphi \big( b_{m}(x, \bar{\mu}_{m})+\sigma_{m}(x, \bar{\mu}_{m})Z_{m+1}\big)\Big]\nonumber\\
&\hspace{0.5cm}\text{(by Lemma~\ref{marglemma}, as $\bar{\mu}_m\conright \bar{\nu}_m$)}\nonumber\\
&\leq \int_{\mathbb{R}^{d}}\bar{\nu}_{m}(dx)\,\EE \Big[ \varphi \big( b_{m}(x, \bar{\mu}_{m})+\theta_{m}(x, \bar{\mu}_{m})Z_{m+1}\big)\Big]\nonumber\\
&\hspace{0.5cm}\text{\big(by Assumption~(\ref{conm3}) and Lemma~\ref{revisiedjensen}-$(iii)$}\big)\nonumber\\
&\leq \int_{\mathbb{R}^{d}}\bar{\nu}_{m}(dx)\,\EE \Big[ \varphi \big( b_{m}(x, \bar{\nu}_{m})+\theta_{m}(x, \bar{\nu}_{m})Z_{m+1}\big)\Big]\nonumber\\
&\hspace{0.5cm}\text{(by Assumption~(\ref{assmbm}), (\ref{conmtheta}) and Lemma~\ref{revisiedjensen}, since $\bar{\mu}_{m}\conright\bar{\nu}_{m}$)}\nonumber\\
&=\EE [\varphi(\bar{Y}_{m+1})]\nonumber.
\end{align}
Thus one has $\bar{X}_{m}\conright\bar{Y}_{m}$ for every $m=0, \dots, M$ by a forward induction. Hence we can conclude by applying the convergence of the Euler scheme as in the proof of Corollary~\ref{cormain}$(a)$.
\end{proof}

\subsection{Global convex order for the Euler scheme}
%In this section, we will prove $(\bar{X}_{0}, \ldots, \bar{X}_{M})\conright(\bar{Y}_{0}, \ldots, \bar{Y}_{M})$. 
%The main goal of this section is to prove Proposition~\ref{funconvexEuler}. %{\color{red}
We will prove Proposition~\ref{funconvexEuler} in this section. 
%In this section, 
For any $K\in\mathbb{N}^{*}$, we consider the norm on $(\mathbb{R}^{d})^{K}$ defined by $\vertii{x}\coloneqq \sup_{1\leq i \leq K}\left|x_{i}\right|$ for every $x=(x_{1}, \dots, x_{K})\in(\mathbb{R}^{d})^{K}$, where $|\cdot|$ denotes the canonical Euclidean norm on $\RD$.
%}
For any $m_{1}, m_{2}\in\mathbb{N}^{*}$ with $m_{1}\leq m_{2}$, we denote by $x_{m_1:m_{2}}\coloneqq (x_{m_{1}}, x_{m_{1}+1}, \ldots, x_{m_{2}})\in(\RD)^{m_{2}-m_{1}+1}$. Similarly, we denote by $\mu_{m_{1}: \;m_{2}}\coloneqq(\mu_{m_{1}}, \ldots, \mu_{m_{2}})\in\big(\mathcal{P}_{1}(\RD)\big)^{m_{2}-m_{1}+1}$.
We recursively define a sequence of functions~(\footnote{We formally consider the case where $\Phi_M$ may depend on $\mu_M$ in view of the proof of Proposition~\ref{convschemG} and item $(b)$ of Theorem~\ref{mainthmconv}.}) 
%$\Phi_M: (\RD)^{M+1}\rightarrow \RR$ and 
\[
\Phi_{m}: (\RD)^{m+1}\times \big(\mathcal{P}_{1}(\RD)\big)^{M-m+1}\rightarrow \RR, \quad m=0, \ldots, M
\] 
in a backward way as follows: 

$\blacktriangleright\quad$Set 
\begin{equation}\label{defPhi1}
\Phi_{M}(x_{0: M}; \mu_{M})\coloneqq F(x_{0}, \ldots, x_{M})
\end{equation} 
where $F:(\mathbb{R}^{d})^{M+1}\rightarrow \RR$ is a convex function with $p$-polynomial growth (\ref{rpolygrowth2}).%with the same $F$ as in Proposition~\ref{funconvexEuler}. 

$\blacktriangleright\quad$For $m=0, \ldots, M-1$, set
\begin{align}
&\Phi_{m}(x_{0: m}; \mu_{m: M}) \coloneqq\big(Q_{m+1}\Phi_{m+1}(x_{0: m}, \,\cdot\,; \mu_{m+1: M})\big)\big(x_{m}, \mu_{m}, \sigma_{m}(x_{m}, \mu_{m})\big)&\nonumber\\
\label{defPhi2}
&\hspace{0.7cm}=\EE \Big[\Phi_{m+1}\big(x_{0: m}, {\color{black}b_{m}(x_{m}, \mu_{m})}+\sigma_{m}(x_{m}, \mu_{m})Z_{m+1}; \mu_{m+1: M}\big)\Big].&
%&\hspace{0.7cm}=\EE \Big[\Phi_{m+1}\big(x_{0: m}, \bar{\alpha}_{m}x_{m}+\bar{\beta}_{m}\int_{\RD}\xi\mu_{m}(d\xi)+\bar{\gamma}_{m}+\sigma_{m}(x_{m}, \mu_{m})Z_{m+1}, \mu_{m+1: M}\big)\Big],&
\end{align}
%where $\bar{\alpha}_{m}, \bar{\beta}_{m}$ and $\bar{\gamma}_{m}$ are defined in (\ref{defalpham}).
The functions  $\Phi_{m}, m=0, \ldots, M$, %are well-defined and 
 share the following properties.  
\begin{lem}\label{propphi}
%Let $\Phi_{m}, m=0, \ldots, M$ be functions defined by~(\ref{defPhi1}) and~(\ref{defPhi2}). 
For every $m=0, \ldots, M$,  %the functions  $\Phi_{m}$ satisfy the following conditions. 
\begin{enumerate}[$(i)$]
\item for a fixed $\mu_{m:M}\in\big(\mathcal{P}_{1}(\RD)\big)^{M-m+1}$, the function $\Phi_{m}(\;\cdot\;; \mu_{m: M})$ is convex and has a  $p$-polynomial growth in $x_{0: m}$ so that $\Phi_{m}$ is well-defined. %and
\item for a fixed $x_{0:m}\in(\RD)^{m+1}$, the function $\Phi_{m}(x_{0: m}\,; \;\cdot\;)$ is non-decreasing in $\mu_{m:M}$ with respect to the convex order  in the sense that for any $\mu_{m:M}, \nu_{m:M}\in \big(\mathcal{P}_{1}(\RD)\big)^{M-m+1}$ such that $\mu_{i}\conright \nu_{i}, i= m, \ldots, M$, 
\begin{align}
& \Phi_{m}(x_{0:m}\,; \mu_{m: M})\leq \Phi_{m}(x_{0:m}\, ;\nu_{m: M}). 
\end{align}
\end{enumerate}
\end{lem}

%\vspace{-1cm}
\begin{proof} 
\noindent  $(i)$ The function $\Phi_{M}$ is convex in $x_{0:M}$ owing to the hypotheses on $F$. Now assume that $x_{0: m+1}\mapsto\Phi_{m+1}(x_{0:m+1}\,;\mu_{m+1:M})$ is convex. For any $x_{0:m}, y_{0:m}\in (\RD)^{m+1}$ and $\lambda\in[0, 1]$, it follows that 
\begin{align}
&\Phi_{m}\big(\lambda x_{0:m}+(1-\lambda)y_{0:m}\,; \mu_{m:M}\big)\nonumber\\
&\hspace{0.5cm}=\EE \Phi_{m+1}\Big( \lambda x_{0:m}+(1-\lambda)y_{0:m},\; {\color{black}b_{m}\big(\lambda x_{m}+(1-\lambda)y_{m}, \mu_{m}\big)}\nonumber\\
&\hspace{7cm}+\sigma_{m}\big(\lambda x_{m}+(1-\lambda)y_{m}, \mu_{m}\big)Z_{m+1}\,; \mu_{m+1:M}\Big)\nonumber\\
%&\hspace{1cm}\bar{\alpha}_{m}\big(\lambda x_{m}+(1-\lambda)y_{m}\big)+\bar{\beta}_{m}\int_{\RD}\xi\mu_{m}(d\xi)+\bar{\gamma}_{m}+\sigma_{m}\big(\lambda x_{m}+(1-\lambda)y_{m}, \mu_{m}\big)Z_{m+1}, \mu_{m+1:M}\Big)\nonumber\\
%%
&\hspace{0.5cm}\leq\EE \Phi_{m+1}\Big( \lambda x_{0:m}+(1-\lambda)y_{0:m}, {\color{black}\lambda\, b_{m}\big( x_{m}, \mu_{m}\big)+(1-\lambda)\,b_{m}\big(y_{m}, \mu_{m}\big)}\nonumber\\
&\hspace{3cm}+\big[\lambda \sigma_{m}(x_{m}, \mu_{m})+(1-\lambda)\sigma_{m}(y_{m}, \mu_{m})\big]Z_{m+1}\,;\mu_{m+1:M}\Big) \nonumber\\
&\hspace{1cm}\text{(by Assumption~(\ref{conm1}) and Lemma~\ref{revisiedjensen} since $\Phi_{m+1} (x_{0:m}, \cdot\,; \mu_{m+1:M})$ is convex)}\nonumber\\
&\hspace{0.5cm}\leq \lambda \EE\Big[\Phi_{m+1}\big(x_{0:m}, {\color{black}b_{m}(x_{m}, \mu_{m})}+\sigma_{m}(x_{m}, \mu_{m})Z_{m+1}\,; \mu_{m+1:M}\big)\Big]\nonumber\\
&\hspace{0.5cm}\hspace{0.5cm}+(1-\lambda)\EE\Big[\Phi_{m+1}\big(y_{0:m}, {\color{black}b_{m}\big(y_{m}, \mu_{m}\big)}+\sigma_{m}(y_{m}, \mu_{m})Z_{m+1}\,;  \mu_{m+1:M}\big)\Big]\nonumber\\
&\hspace{1cm}\text{(since $x_{0:m+1}\mapsto\Phi_{m+1} (x_{0:m+1}\,;  \mu_{m+1:M})$ is convex)}\nonumber\\
&\hspace{0.5cm}=\lambda \Phi_{m}(x_{0:m}\,; \mu_{m:M})+(1-\lambda) \Phi_{m}(y_{0:m}\,; \mu_{m:M}).\nonumber
\end{align}
Thus the function $\Phi_{m}(\,\cdot\,;  \mu_{m:M})$ is convex and one concludes by a backward induction.

The function $\Phi_{M}$ has a $p$-polynomial growth by the assumption made on $F$. Now assume that $\Phi_{m+1}$ has a $p$-polynomial growth. As Assumption~\ref{AssumptionI} implies that {\color{black}$b_m$ and $\sigma_{m}$ have} linear growth (see further (\ref{lineargrowth})), it is obvious that $\Phi_{m}$ has  $p$-polynomial growth. %of  Lemma~\ref{marglemma} and the following obvious inequality
%\[\forall\, \,C_{1}, C_{2}\in\mathbb{R}_{+}^{*}, \quad C_{1} \sup_{1\leq i\leq m}\left|x_{i}\right|^{r} + C_{2}\left|x_{m}\right|^{r}\leq (C_{1}+C_{2}) \sup_{1\leq i\leq m}\left|x_{i}\right|^{r}.\] 
Thus one concludes by a backward induction.

\noindent$(ii)$ Firstly, it is obvious that for any $\mu_{M}, \nu_{M}\in\mathcal{P}_{1}(\RD)$ such that $\mu_{M}\conright \nu_{M}$, we have
%\forall\, \, x_{0:M}\in\big(\RD\big)^{M+1},\hspace{1cm}
\[\Phi_{M}(x_{0:M}\,;  \mu_{M})=F(x_{0:M})=\Phi_{M}(x_{0:M}\,;  \nu_{M}).\]
Assume that $\Phi_{m+1}(x_{0:m+1}\,;  \cdot\,)$ is non-decreasing with respect to the convex order of $\mu_{m+1:M}$. For any $\mu_{m:M}, \nu_{m:M}\in\big(\mathcal{P}_{1}(\RD)\big)^{M-m+1}$ such that $\mu_{i}\conright\nu_{i}, i=m, \ldots, M,$ we have
\begin{align}
&\Phi_{m}(x_{0:m}\,;  \mu_{m:M})=\EE \Big[ \Phi_{m+1}\big(x_{0:m}, {\color{black}b_m(x_{m}, \mu_{m})}+\sigma_{m}(x_{m}, \mu_{m})Z_{m+1}\,;  \mu_{m+1: M}\big)\Big]\nonumber\\
&\hspace{0.5cm}\leq \EE \Big[ \Phi_{m+1}\big(x_{0:m}, {\color{black}b_m(x_{m}, \nu_{m})}+\sigma_{m}(x_{m}, \nu_{m})Z_{m+1}\,;  \mu_{m+1: M}\big)\Big]\nonumber\\
&\hspace{1cm}\text{(by Assumption~(\ref{conm2}), (\ref{assmbm}) and Lemma~\ref{revisiedjensen} since $\Phi{_{m+1}} (x_{0:m}, \cdot\,;  \mu_{m+1:M})$ is convex)}\nonumber\\
&\hspace{0.5cm}\leq \EE\Big[ \Phi_{m+1}\big(x_{0:m},  {\color{black}b_m(x_{m}, \nu_{m})}+\sigma_{m}(x_{m}, \nu_{m})Z_{m+1}\,;  \nu_{m+1: M}\big)\Big]\nonumber\\
&\hspace{1cm}\hspace{0.2cm}\text{(by the assumption on $\Phi_{m+1}$)}\nonumber\\
&\hspace{0.5cm}=\Phi_{m}(x_{0:m}\,;  \nu_{m:M}).\nonumber
\end{align}
Then one concludes by a backward induction. 
\end{proof}

As $F$ has an  $p$-polynomial growth, then the integrability of $F(\bar{X}_{0}, \ldots, \bar{X}_{M}) \text{ and } F(\bar{Y}_{0}, \ldots, \bar{Y}_{M})$ is guaranteed by Proposition~\ref{cvgeuler} since $X_{0}, Y_{0}\in L^{p}(\mathbb{P})$.
%Before we start to prove Proposition~\ref{funconvexEuler}, we firstly define a martingale by 
We define for every $m=0, \ldots, M$, 
\[\mathcal{X}_{m}\coloneqq \EE \big[F(\bar{X}_{0}, \ldots, \bar{X}_{M})\;\big|\;\mathcal{F}_{m}\big].\]
%where the function $F$ is the same as in Proposition~\ref{funconvexEuler}. For every $m=0, \ldots, M$, let $\bar{\mu}_{m}$ denote the probability distribution of $\bar{X}_{m}$. 
Recall the notation $\bar{\mu}_{m}\coloneqq \mathbb{P}_{\bar{X}_{m}}\coloneqq \mathbb{P}\circ \bar{X}_{m}^{-1}, \,m=0, \ldots, M.$

\smallskip
\begin{lem}\label{martingale}
For every $m=0, \ldots, M$, $\Phi_{m}(\bar{X}_{0:m}\,; \bar{\mu}_{m:M})=\mathcal{X}_{m}$.
\end{lem}
\begin{proof}
%We prove by a backward induction. 
It is obvious that $\Phi_{M}(\bar{X}_{0:M}\,; \bar{\mu}_{M})=F(\bar{X}_{0}, \ldots, \bar{X}_{M})=\mathcal{X}_{M}.$
Assume that $\Phi_{m+1}(\bar{X}_{0:m+1}\,; \bar{\mu}_{m+1:M})=\mathcal{X}_{m+1}$. Then %It follows that 
\begin{align}
\mathcal{X}_{m}&=\EE \big[\mathcal{X}_{m+1}\mid\mathcal{F}_{m}\big]
=\EE \big[\Phi_{m+1}(\bar{X}_{0:m+1}\,;\bar{\mu}_{m+1:M}) \mid \mathcal{F}_{m}\big]\nonumber\\
&=\EE \big[\Phi_{m+1}(\bar{X}_{0:m},  b_{m}(\bar{X}_{m}, \bar{\mu}_{m})+\sigma_{m}(\bar{X}_{m}, \bar{\mu}_{m})Z_{m+1}\,; \bar{\mu}_{m+1:M})\mid \mathcal{F}_{m}\big]\nonumber\\
&=\big(Q_{m+1}\Phi_{m+1}(\bar{X}_{0:m}, \cdot\,; \bar{\mu}_{m+1:M})\big)\big(\bar{X}_{m}, \bar{\mu}_{m}, \sigma_{m}(\bar{X}_{m}, \bar{\mu}_{m})\big)=\Phi_{M}(\bar{X}_{0:m}\,; \bar{\mu}_{m:M}).\nonumber
\end{align}
Then a backward induction completes the proof.
\end{proof}

%\begin{proof}
%%%%We prove by a backward induction. 
%It is obvious that $\Phi_{M}(\bar{X}_{0:M}, \bar{\mu}_{M})=F(\bar{X}_{0}, \ldots, \bar{X}_{M})=\mathcal{X}_{M}.$
%Assume that $\Phi_{m+1}(\bar{X}_{0:m+1}, \bar{\mu}_{m+1:M})=\mathcal{X}_{m+1}$. Then %It follows that 
%\begin{align}
%\mathcal{X}_{m}&=\EE \big[\mathcal{X}_{m+1}\mid\mathcal{F}_{m}\big]
%=\EE \big[\Phi_{m+1}(\bar{X}_{0:m+1}, \bar{\mu}_{m+1:M}) \mid \mathcal{F}_{m}\big]\nonumber\\
%%%%
%&=\EE \big[\Phi_{m+1}(\bar{X}_{0:m}, \bar{\alpha}_{m}\bar{X}_{m}+\bar{\beta}_{m}\EE\bar{X}_{m}+\bar{\gamma}_{m}+\sigma_{m}(\bar{X}_{m}, \bar{\mu}_{m})Z_{m+1}, \bar{\mu}_{m+1:M})\mid \mathcal{F}_{m}\big]\nonumber\\
%&=\big(Q_{m+1}\Phi_{m+1}(\bar{X}_{0:m}, \cdot, \bar{\mu}_{m+1:M})\big)\big(\bar{X}_{m}, \bar{\mu}_{m}, \sigma_{m}(\bar{X}_{m}, \bar{\mu}_{m})\big)=\Phi_{M}(\bar{X}_{0:m}, \bar{\mu}_{m:M}).\nonumber
%\end{align}
%Then a backward induction completes the proof.
%\end{proof}
Similarly, we define $\Psi_{m}:(\RD)^{m+1}\times \big(\mathcal{P}_{1}(\RD)\big)^{M-m+1}\rightarrow \RR,  \,m=0, \ldots, M$ by
\begin{align}
&\Psi_{M}(x_{0:M}\,; \mu_{M})\hspace{0.25cm}\coloneqq F(x_{0:M})\nonumber\\
&\Psi_{m}(x_{0:m}\,; \mu_{m:M})\coloneqq \big(Q_{m+1}\Psi_{m+1}(x_{0:m}, \;\cdot\;\,; \mu_{m+1:M})\big)
\big(x_{m}, \mu_{m}, \theta_{m}(x_{m}, \mu_{m})\big)\nonumber\\
\label{defpsi}
&\hspace{2.7cm}=\EE \Big[\Psi_{m+1}\big(x_{0:m},  b_{m}(x_{m}, \mu_{m})+\theta_{m}(x_{m}, \mu_{m})Z_{m+1}\,; \mu_{m+1:M}\big)\Big]. 
\end{align}
Recall the notation $\bar{\nu}_{m}\coloneqq \mathbb{P}_{\bar{Y}_{m}}$.  By using the same  method of  proof as for Lemma~\ref{martingale}, we get
%It follows from the same reasoning as in Lemma~\ref{martingale} that 
%Let $\bar{\nu}_{m}$ denote the probability distribution of $\bar{Y}_{m}, m=0, \ldots, M.$ %It follows from the same reasoning of Lemma~\ref{martingale} that
%Hence 
\[\Psi_{m}(\bar{Y}_{0:m}\,; \bar{\nu}_{m:M})=\EE \big[F(\bar{Y}_{0}, \ldots, \bar{Y}_{M})\mid \mathcal{F}_{m}\big].\] 
%by to the same reasoning as in Lemma~\ref{martingale}. 

\begin{proof}[Proof of Proposition~\ref{funconvexEuler}]

We first prove by a backward induction that for every $m=0, \ldots, M$, $\Phi_{m}\leq\Psi_{m}$. 

It follows from the definition of $\Phi_{M}$ and $\Psi_{M}$ that $\Phi_{M}=\Psi_{M}$. Assume now $\Phi_{m+1}\leq \Psi_{m+1}$. For any $x_{0:m}\in(\RD)^{m+1}$ and $\mu_{m:M}\in\big(\mathcal{P}_{1}(\RD)\big)^{M-m+1}$, we have 
\begin{align}
&\Phi_{m}(x_{0:m}\,;\mu_{m:M})\nonumber\\
&=\EE\big[\Phi_{m+1}\big(x_{0:m}, b_{m}(x_{m}, \mu_{m})+\sigma_{m}(x_{m}, \mu_{m})Z_{m+1}\,; \mu_{m+1:M}\big)\big]\nonumber\\
&\leq\EE\big[\Phi_{m+1}\big(x_{0:m}, b_{m}(x_{m}, \mu_{m})+\theta_{m}(x_{m}, \mu_{m})Z_{m+1}\,; %\bar{\alpha}x_{m}+\bar{\beta}+\theta_{m}(x_{m}, \mu_{m})Z_{m+1}, 
\mu_{m+1:M}\big)\big]\nonumber\\
&\quad\text{(by Assumption~(\ref{conm3}) and Lemma~\ref{revisiedjensen}, since Lemma~\ref{propphi} shows that $\Phi_{m+1}$ is  convex in $x_{0: m+1}$) }\nonumber\\
&\leq\EE\big[\Psi_{m+1}\big(x_{0:m}, b_{m}(x_{m}, \mu_{m})+\theta_{m}(x_{m}, \mu_{m})Z_{m+1}\,; \mu_{m+1:M}\big)\big]=\Psi_{m}(x_{0:m}\,; \mu_{m:M}).\nonumber
\end{align}
%\begin{align}
%&\Phi_{m}(x_{0:m}, \mu_{m:M})\nonumber\\
%%%%
%&=\EE\big[\Phi_{m+1}\big(x_{0:m}, \bar{\alpha}_{m}x_{m}+\bar{\beta}_{m}\int_{\RD}\xi\mu_{m}(d\xi)+\bar{\gamma}_{m}+\sigma_{m}(x_{m}, \mu_{m})Z_{m+1}, \mu_{m+1:M}\big)\big]\nonumber\\
%&\leq\EE\big[\Phi_{m+1}\big(x_{0:m}, \bar{\alpha}_{m}x_{m}+\bar{\beta}_{m}\int_{\RD}\xi\mu_{m}(d\xi)+\bar{\gamma}_{m}+\theta_{m}(x_{m}, \mu_{m})Z_{m+1}, %%%%\bar{\alpha}x_{m}+\bar{\beta}+\theta_{m}(x_{m}, \mu_{m})Z_{m+1}, 
%\mu_{m+1:M}\big)\big]\nonumber\\
%&\quad\text{(by Assumption~(\ref{conm3}) and Lemma~\ref{revisiedjensen}, since Lemma~\ref{propphi} shows that $\Phi_{m+1}$ is  convex in $x_{0: m+1}$) }\nonumber\\
%&\leq\EE\big[\Psi_{m+1}\big(x_{0:m}, \bar{\alpha}_{m}x_{m}+\bar{\beta}_{m}\int_{\RD}\xi\mu_{m}(d\xi)+\bar{\gamma}_{m}+\theta_{m}(x_{m}, \mu_{m})Z_{m+1}, \mu_{m+1:M}\big)\big]=\Psi_{m}(x_{0:m}, \mu_{m:M}).\nonumber
%\end{align}
Thus,  the backward induction is completed and 
\begin{equation}\label{phipsiorder}
\forall\, \, m=0, \ldots, M, \quad \Phi_{m}\leq\Psi_{m}.
\end{equation}
%By the same reasoning as in Lemma~\ref{martingale}, we have $\Psi_{m}(\bar{Y}_{0:m}, \bar{\nu}_{m:M})=\EE \big[F(\bar{Y}_{0}, \ldots, \bar{Y}_{M})\mid \mathcal{F}_{m}\big], m=0, \ldots, M$. 
Consequently, 
\begin{align}
\EE\big[ F(\bar{X}_{0}, \ldots, \bar{X}_{M})\big]&=\EE \Phi_{0}(\bar{X}_{0}\,; \bar{\mu}_{0:M})\hspace{0.9cm}\text{(by Lemma~\ref{martingale})}\nonumber\\
&\leq\EE \Phi_{0}(\bar{Y}_{0}\,; \bar{\mu}_{0:M})\hspace{1cm}\text{(by Lemma~\ref{propphi}-$(i)$ since }  \bar{X}_{0}=X_0\conright Y_0=\bar{Y}_{0})\nonumber\\
&\leq \EE \Phi_{0}(\bar{Y}_{0}\,; \bar{\nu}_{0:M})  \hspace{1cm}\,\text{(by Lemma~\ref{propphi}-$(ii)$ and Proposition~\ref{marginalconvex})}\nonumber\\
&\leq \EE \Psi_{0}(\bar{Y}_{0}\,; \bar{\nu}_{0:M}) \hspace{1cm}\,\text{(by~(\ref{phipsiorder}))}\nonumber\\
&=\EE \big[F(\bar{Y}_{0}, \ldots, \bar{Y}_{M})\big].\nonumber \hspace{7cm}\qedhere
\end{align}
%owing to the martingale property. 
\end{proof}

%%%%%%%%%%%%%%%%%%%%%%
%%%%%%%%%%%%%%%%%%%%%%
%%%%%%%%%%%%%%%%%%%%%% 

\section{Functional convex order for the McKean-Vlasov process}\label{convprocess}

This section is devoted to prove Theorem~\ref{mainthmconv}-$(a)$.
Recall that $t_{m}^{M}=m\cdot \frac{T}{M}, m=0,  \ldots, M$. We define two interpolators as follows. %The \textit{piecewise affine interpolator} is defined as follows. 

\begin{defn}
%\begin{enumerate}[$(i)$]
%\item 
$(i)$ For every integer $M\geq1$, we define the piecewise affine interpolator $i_{M}: x_{0:M}\in(\RD)^{M+1}\mapsto i_{M}(x_{0:M})\in\mathcal{C}([0, T], \RD)$ by 
\begin{flalign}
%&\forall\, \, x_{0:M}\in(\RD)^{M+1},\;
& \forall\, m=0, \ldots, M-1, \;\forall\, t\in[t_{m}^{M}, t_{m+1}^{M}],\hspace{0.3cm}i_{M}(x_{0:M})(t)=\frac{M}{T}\big[(t_{m+1}^{M}-t)x_{m}+(t-t^{M}_{m})x_{m+1}\big].&\nonumber
\end{flalign}
%\item 
$(ii)$ For every $M\geq 1$, we define the functional interpolator $I_{M}: \mathcal{C}\big([0, T], \RD\big)\rightarrow\mathcal{C}\big([0, T], \RD\big)$ by 
\[\forall\, \, \alpha \in \mathcal{C}([0, T], \RD), \hspace{1cm} I_{M}(\alpha)=i_{M}\big(\alpha(t_{0}^{M}), \ldots, \alpha(t_{M}^{M})\big).\]
%\end{enumerate}
\end{defn}

It is obvious that 
\begin{equation}\label{supinterpolator}
\forall\, \, x_{0:M}\in(\RD)^{M+1},\hspace{0.5cm} \vertii{i_{M}(x_{0:M})}_{\sup}\textcolor{black}{=} \max_{0\leq m\leq M}\left|x_{m}\right|
\end{equation}
since the norm $\left|\cdot\right|$ is convex. Consequently, 
\begin{equation}\label{supinterpolator2}
\forall\, \, \alpha\in\mathcal{C}([0, T], \RD), \hspace{0.8cm} \vertii{I_{M}(\alpha)}_{\sup}\leq\vertii{\alpha}_{\sup}.
\end{equation}
Moreover, for any  $\alpha\in\mathcal{C}([0, T], \RD)$, we have
\begin{equation}\label{imconv}
\vertii{I_{M}(\alpha)-\alpha}_{\sup}\leq w(\alpha, \tfrac{T}{M}), 
\end{equation}
where $w$ denotes the uniform continuity modulus of $\alpha$. 
%It is obvious that $\sup_{t\in[0, T]}\left| I_{M}(\alpha)_{t}\right|\leq \sup_{t\in[0, T]}\left|\alpha(t)\right|$ since the norm $\left|\cdot\right|$ is convex. 
%We recall Lemma 2.2 in~\cite{pages2016convex}.
The proof of Theorem~\ref{mainthmconv}-$(a)$ relies on the following lemma. 
\begin{lem}[Lemma 2.2 in~\cite{pages2016convex}]\label{Imlemma}
Let $X^{M}, M\geq1$, be a sequence of continuous processes weakly converging towards $X$ as $M\rightarrow +\infty$ for the $\left\Vert\cdot\right\Vert_{sup}$-norm topology. Then, the sequence of interpolating processes $\widetilde{X}^{M}=I_{M}(X^{M}), M\geq 1$ is weakly converging toward $X$ for the $\left\Vert\cdot\right\Vert_{\sup}$-norm topology.
\end{lem}

\begin{proof}[Proof of Theorem~\ref{mainthmconv}-$(a)$]
Let $M\in\mathbb{N}^{*}.$ %, h=\frac{T}{M}, t_{m}^{M}=m\cdot h=m\cdot \frac{T}{M}$. 
Let $(\bar{X}_{t_{m}}^{M})_{m=0, \ldots, M}$ and $(\bar{Y}_{t_{m}}^{M})_{m=0, \ldots, M}$ denote the Euler scheme defined in~(\ref{defeulerx}) and~(\ref{defeulery}). Let $\bar{X}^{M}\coloneqq(\bar{X}^{M}_{t})_{t\in[0, T]}$, $\bar{Y}^{M}\coloneqq(\bar{Y}^{M}_{t})_{t\in[0, T]}$  denote the continuous Euler scheme of $(X_{t})_{t\in[0, T]}$, $(Y_{t})_{t\in[0, T]}$ defined by (\ref{defeulercontinuousx}) and (\ref{defeulercontinuousy}). 
By Proposition~\ref{cvgeuler}, there exists a constant $\tilde{C}$ such that 
\begin{align}\label{supxy}
&\vertii{\sup_{t\in[0, T]}\left|\bar{X}_{t}^{M}\right|}_{p}\vee\vertii{\sup_{t\in[0, T]}\left|X_{t}\right|}_{p}\leq \tilde{C}(1+\vertii{X_{0}}_{p})<+\infty,\nonumber\\
&\vertii{\sup_{t\in[0, T]}\left|\bar{Y}_{t}^{M}\right|}_{p}\vee\vertii{\sup_{t\in[0, T]}\left|Y_{t}\right|}_{p}\leq \tilde{C}(1+\vertii{Y_{0}}_{p})<+\infty
\end{align}
as $1\leq r\leq p$ and $X_{0}, Y_{0}\in L^{p}(\mathbb{P})$. 
Hence, $F(X)$ and $F(Y)$ are in $L^{1}(\mathbb{P})$ since $F$ has a $r$-polynomial growth.

We define a function $F_{M}: (\RD)^{M+1}\rightarrow \RR$ by 
\begin{equation}
x_{0:M}\in(\RD)^{M+1}\mapsto F_{M}(x_{0:M})\coloneqq F\big(i_{M}(x_{0:M})\big).
\end{equation}
The function $F_{M}$ is obviously convex since $i_{M}$ is a linear application. Moreover, $F_{M}$ has also an $r$-polynomial growth (on $\RR^{M+1}$) by~(\ref{supinterpolator}).% since $\sup_{t\in[0, T]}\left| i_{M}(x_{0:M})\right|\leq \max_{i=0:M}\left|x_{i}\right|$ as any norm $\left|\cdot\right|$ on $\RD$ is convex. 

Furthermore, we have $I_{M}(\bar{X}^{M})=i_{M}\big((\bar{X}_{t_{0}}^{M}, \ldots, \bar{X}_{t_{M}}^{M})\big)$ by the definition of the continuous Euler scheme and the interpolators $i_{M}$ and $I_{M}$, so that 
\[F_{M}(\bar{X}_{t_{0}}^{M}, \ldots, \bar{X}_{t_{M}}^{M})=F\Big(i_{M}\big((\bar{X}_{t_{0}}^{M}, \ldots, \bar{X}_{t_{M}}^{M})\big)\Big)=F\big(I_{M}(\bar{X}^{M})\big).\]
It follows from Proposition~\ref{funconvexEuler} that 
\begin{align}\label{im}
\EE &F\big(I_{M}(\bar{X}^{M})\big)=\EE F\big(i_{M}(\bar{X}_{0}^{M}, \ldots, \bar{X}_{M}^{M})\big)=\EE F_{M}\big(\bar{X}_{0}^{M}, \ldots, \bar{X}_{M}^{M}\big)\nonumber\\
&\hspace{1cm}\leq \EE F_{M}\big(\bar{Y}_{0}^{M}, \ldots, \bar{Y}_{M}^{M}\big)=\EE F\big(i_{M}(\bar{Y}_{0}^{M}, \ldots, \bar{Y}_{M}^{M})\big)=\EE F\big(I_{M}(\bar{Y}^{M})\big).
\end{align}

%The function $F$ is $\vertii{\cdot}_{\sup}-$continuous since it is convex with $\vertii{\cdot}_{\sup}-$polynomial growth (see Lemma 2.1.1 in~\cite{MR2179578}). 
%Moreover the process $\bar{X}^{M}$ weakly converges to $X$ as $M\rightarrow+\infty$ by Proposition~\ref{cvgeuler}. Then $I_{M}(\bar{X}^{M})$ weakly converges to $X$ by applying Lemma~\ref{Imlemma}. Hence the inequality~(\ref{im}) implies that 
%\[\EE F(X)\leq \EE F(Y),\]
%by letting $M\rightarrow +\infty$ {\color{red} as the random variable sequences $\Big(F\big(I_{M}(\bar{X}^{M})\big)\Big)_{M\geq 1}$, $\Big(F\big(I_{M}(\bar{Y}^{M})\big)\Big)_{M\geq 1}$ are uniformly integrable by (\ref{supinterpolator2}) and (\ref{supxy}) since $F$ has an $r$-polynomial growth, $r<p$.} %{\color{red}by applying the Lebesgue dominated convergence theorem }owing to~(\ref{supxy}) since $F$ has an $r$-polynomial growth. 
%%\color{black}
%%F continuous ??
%%\color{black}
The function $F$ is $\vertii{\cdot}_{\sup}$-continuous since it is convex with $\vertii{\cdot}_{\sup}$-polynomial growth (see Lemma 2.1.1 in~\cite{MR2179578}). 
Moreover the process $\bar{X}^{M}$ weakly converges for the $\sup$-norm topology to $X$ as $M\rightarrow+\infty$ as a consequence of Proposition~\ref{cvgeuler}.  Then $I_{M}(\bar{X}^{M})$ weakly converges  for the $\sup$-norm topology   to $X$ owing to Lemma~\ref{Imlemma}.  This proves that  $F\big(I_{M}(\bar{X}^{M})\big)$ weakly converges toward $F(X)$ and, similarly, that  $F\big(I_{M}(\bar{Y}^{M})\big)$ weakly converges toward $F(Y)$. Moreover, as $F$ has a $p$-polynomial growth, we have 
\[
F\big(I_M(\bar{X}^{M})\big)\leq C\big( 1 + \vertii{I_M(\bar{X}^{M})}_{\sup}^{p}\big)\leq C\big( 1 + \vertii{\bar{X}^{M}}_{\sup}^{p}\big)
\]
where the last inequality follows from \eqref{supinterpolator2}. It follows from Proposition \ref{cvgeuler} that 
\[\EE \vertii{\bar{X}^{M}}_{\sup}^{p}\rightarrow \EE \vertii{X}_{\sup}^{p}\quad \text{as}\: M\rightarrow +\infty. \]
Then one derives that $\EE\, F\big(I_{M}(\bar{X}^{M})\big)\to \EE\, F(X)$ as $M\to +\infty$. 
%\textcolor{blue}{Moreover
%\[
%\EE\, \vertii{ I_{M}(\bar{X}^{M})}^p_{\sup} = \EE \max_{k=0:M}|\bar X_{t^M_k}|^p \to  \EE \sup_{t\in [0,T]}| X_{t}|^p\quad \mbox{as} \quad M\to +\infty
%\]
%again by Proposition~\ref{cvgeuler}. Hence, as $\alpha \mapsto \vertii{\alpha}_{\sup}^p$ is norm-$\sup$-continuous and  $|F(\alpha)|\le C(1+\vertii{\alpha}_{\sup}^p)$, one derives that $\EE\, F\big(I_{M}(\bar{X}^{M})\big)\to \EE\, F(X)$ as $M\to +\infty$. }
The same reasoning shows that $\EE\, F\big(I_{M}(\bar{Y}^{M})\big)\to \EE\, F(Y)$.
%It follows from~\eqref{supinterpolator2}) and~\eqref{supxy} that both sequences $\big(\vertii{ I_{M}(\bar{X}^{M})}_{\sup}\big)_{M\ge 1}$ and $\big(\vertii{ I_{M}(\bar{Y}^{M})}_{\sup}\big)_{M\ge 1}$ are $L^p$-bounded which in turn implies that both $\big(F\big(I_{M}(\bar{X}^{M})\big)\big)_{M\ge 1}$ and $\big(F\big(I_{M}(\bar{Y}^{M})\big)\big)_{M\ge 1}$  are uniformly integrable since 
%$F$ has at most $\vertii{\cdot}_{\sup}$-norm $r$-polynomial growth  for some $r\!\in [0,p)$. 
%
Finally, one derives   by letting $M\to +\infty$ in inequality~(\ref{im}) that 
\[\EE F(X)\leq \EE F(Y). \hfill\qedhere\]
%{\color{red}
%[NEW...  from ``Moreover '']

%Moreover, as $F$ has a $p$-polynomial growth, we have 
%\[
%F\big(I_M(\bar{X}^{M})\big)\leq C\big( 1 + \vertii{I_M(\bar{X}^{M})}_{\sup}^{p}\big)\leq C\big( 1 + \vertii{\bar{X}^{M}}_{\sup}^{p}\big)
%\]
%where the last inequality follows from \eqref{supinterpolator2}. It follows from Proposition \ref{cvgeuler} that 
%\[\EE \vertii{\bar{X}^{M}}_{\sup}^{p}\rightarrow \EE \vertii{X}_{\sup}^{p}\quad \text{as}\: M\rightarrow +\infty. \]
%Then one derives that $\EE\, F\big(I_{M}(\bar{X}^{M})\big)\to \EE\, F(X)$ as $M\to +\infty$. 

%[END]}
%by letting $M\rightarrow +\infty$ {\color{red} as the random variable sequences $\Big(F\big(I_{M}(\bar{X}^{M})\big)\Big)_{M\geq 1}$, $\Big(F\big(I_{M}(\bar{Y}^{M})\big)\Big)_{M\geq 1}$ %are uniformly integrable by (\ref{supinterpolator2}) and (\ref{supxy}) since $F$ has an $r$-polynomial growth, $r<p$.} %{\color{red}by applying the Lebesgue dominated convergence theorem }owing to~(\ref{supxy}) since $F$ has an $r$-polynomial growth. 
%\color{black}
%F continuous ??
%\color{black}
\end{proof}

\begin{rem}
The functional convex order result, in a general setting, can be used to establish a robust option price bound (see e.g.~\cite{alfonsi2018sampling}). However, in the McKean-Vlasov setting, the functional convex order result Theorem~\ref{mainthmconv} is established by using the \textit{theoretical} Euler scheme \eqref{defeulerx} and \eqref{defeulery} which can not be directly simulated so that there are still some work to do to produce simulable approximations which are consistent for the convex order. %Several computable methods, such as the particle method, are discussed in the literature,
%In the next chapter, we propose the computable  particle method for~(\ref{defconvx}) and~(\ref{defconvy}), 
One simulable approximation of the McKean-Vlasov equation is the particle method (see e.g. \cite{bossy1997stochastic}, \cite{MR1910635} and \cite[Section 7.1]{liu:tel-02396797} among many other references), which, in the context of this paper, can be written as follows:  for  $n=1, \ldots, N $, 
\begin{equation}\label{computable}
\begin{cases}
\bar{X}^{n, N}_{t_{m+1}}\!=\bar{X}^{n, N}_{t_{m}}\!+h\cdot b(t_{m}, \bar{X}^{n, N}_{t_{m}}\!, \frac{1}{N}\sum_{n=1}^{N}\delta_{\bar{X}^{n, N}_{t_{m}}}) +\sqrt{ h}\cdot \sigma(t_{m}, \bar{X}^{n, N}_{t_{m}}\!, \frac{1}{N}\sum_{n=1}^{N}\delta_{\bar{X}^{n, N}_{t_{m}}}) Z_{m+1}^{n}, \\
\bar{Y}^{n, N}_{t_{m+1}}=\bar{Y}^{n, N}_{t_{m}}+h\cdot b(t_{m}, \bar{Y}^{n, N}_{t_{m}}, \frac{1}{N}\sum_{n=1}^{N}\delta_{\bar{Y}^{n, N}_{t_{m}}})+\sqrt{ h}\cdot \theta(t_{m}, \bar{Y}^{n, N}_{t_{m}}, \frac{1}{N}\sum_{n=1}^{N}\delta_{\bar{Y}^{n, N}_{t_{m}}}) Z_{m+1}^{n},%\,\text{ with }\bar{\mu}^{N}_{t_{m}}\coloneqq \frac{1}{N}\sum_{n=1}^{N}\delta_{\bar{X}^{n, N}_{t_{m}}}, \\%\text{ and } \bar{X}^{n, N}_{0}\widesim{\,\text{i.i.d}\,} X_{0}
%\bar{Y}^{n, N}_{t_{m+1}}=\bar{Y}^{n, N}_{t_{m}}+h(\alpha\bar{Y}^{n, N}_{t_{m}}+ \beta)+\sqrt{ h}\theta(\bar{Y}^{n, N}_{t_{m}}, \bar{\nu}^{N}_{t_{m}}) Z_{m+1}^{n}\,\text{ with }\bar{\nu}^{N}_{t_{m}}\coloneqq \frac{1}{N}\sum_{n=1}^{N}\delta_{\bar{Y}^{n, N}_{t_{m}}},
\end{cases}%\]
\end{equation}
where $\bar{X}_{0}^{n, N}, 1\leq n\leq N,\widesim{\,i.i.d\,}X_{0}$, $\bar{Y}_{0}^{n, N}, 1\leq n\leq N\widesim{\,i.i.d\,}Y_{0}$, $t_{m}=t_{m}^{M}\coloneqq m\cdot\frac{T}{M},\, M\in\mathbb{N}^{*}$ and $Z_{m}^{n}, 0\leq n\leq N, 0\leq m\leq M, \widesim{\,i.i.d\,}\mathcal{N}(0, \mathbf{I}_{q})$. %, $\,\bar{X}^{n, N}_{0}$ are i.i.d copies of $X_{0}$ and $\bar{Y}^{n, N}_{0}$ are i.i.d copies of $Y_{0}$.

Unfortunately, this scheme (\ref{computable}) 
based on particles does not propagate nor preserve the convex order as in Proposition~\ref{marginalconvex} since we cannot obtain for a convex function $\varphi$ that,
\[\frac{1}{N}\sum_{n=1}^{N}\varphi\big(X_{t_{m}}^{n, N}(\omega)\big)\leq\frac{1}{N}\sum_{n=1}^{N}\varphi\big(Y_{t_{m}}^{n, N}(\omega)\big), \quad a.s.\]
under the condition that  $X_{t_{m}}^{n, N}\conright Y_{t_{m}}^{n, N}$, $n=1, \ldots, N$,
even if the random variables $X_{t_{m}}^{n, N}$, $n=1, \ldots, N$ and  $Y_{t_{m}}^{n, N}$, $n=1, \ldots, N$ were both i.i.d. (see again~\cite{alfonsi2018sampling}).
\end{rem}

%%%%%%%%%%%%%%%%%%%%%%
%%%%%%%%%%%%%%%%%%%%%%
%%%%%%%%%%%%%%%%%%%%%%
 
\subsection{Extension of the functional convex order result}

This section is devoted to the proof of  Theorem~\ref{mainthmconv}-$(b)$. 
%We prove Theorem~\ref{mainthmconv}-(b) in this section. %the functional convex order result of both the path of process and its marginal distributions. 
We first discuss the marginal distribution space for the strong solutions $X=(X_{t})_{t\in [0, T]}$ and $Y=(Y_{t})_{t\in [0, T]}$ of  equations~(\ref{defconvx}) and~(\ref{defconvy}). % under consideration. 
By Proposition~\ref{cvgeuler}, $X, Y\in L_{\CRD}^{p}(\Omega, \mathcal{F}, \mathbb{P})$ then their probability distributions $\mu, \nu$  naturally lie in 
\begin{align}
&\mathcal{P}_{p}\big(\mathcal{C}([0, T], \mathbb{R}^{d})\big) \nonumber\\
&\quad\coloneqq\left\{\mu \text{ probability distribution on }\mathcal{C}([0, T], \mathbb{R}^{d}) \text{ s.t. } \int_{\mathcal{C}([0, T], \mathbb{R}^{d})}\left\Vert \alpha\right\Vert_{\sup}^{p}\mu(d\alpha)<+\infty\right\}.\nonumber
\end{align}
We define an $L^{p}$-\textit{Wasserstein distance} $\mathbb{W}_{p}$ on $\mathcal{P}_{p}\big(\mathcal{C}([0, T], \mathbb{R}^{d})\big)$ by%: for any  $\mu, \nu\in\mathcal{P}_{p}\big( \mathcal{C}([0, T], \mathbb{R}^{d})\big),$% as~(\ref{defwas2})
\begin{align}\label{defwasC}
&\forall\, \, \mu, \nu\in\mathcal{P}_{p}\big( \mathcal{C}([0, T], \mathbb{R}^{d})\big),\nonumber\\
&\hspace{2cm}\mathbb{W}_{p}(\mu, \nu)\coloneqq\Big{[}\inf_{\pi\in\Pi(\mu,\nu)}\int_{\mathcal{C}([0, T], \mathbb{R}^{d})\times \mathcal{C}([0, T], \mathbb{R}^{d})}\left\Vert x-y\right\Vert_{\sup}^{p}\pi(dx,dy)\Big{]}^{\frac{1}{p}},
\end{align}
where $\Pi(\mu,\nu)$ denotes the set of probability measures on $\mathcal{C}([0, T], \mathbb{R}^{d})\times \mathcal{C}([0, T], \mathbb{R}^{d})$ %\text{Bor}\big( \mathcal{C}([0, T], \mathbb{R}^{d})\big)^{\otimes2}\Big)$
with respective marginals  $\mu$ and $\nu$. The space $\mathcal{P}_{p}\big(\mathcal{C}([0, T], \mathbb{R}^{d})\big)$ equipped with $\mathbb{W}_{p}$ is complete and separable since $\big(\mathcal{C}([0, T], \mathbb{R}^{d}), \left\Vert\cdot\right\Vert_{\sup}\big)$ is a Polish space (see~\cite{bolley2008separability}). 

Now, we prove that for any stochastic process $X=(X_{t})_{t\in[0, T]}\in L_{\CRD}^{p}(\Omega, \mathcal{F}, \mathbb{P})$, its marginal distribution $(\mu_{t})_{t\in[0, T]}$ lies in $\CPP$. 
For any $t\in[0, T]$, we define $\pi_{t}: \CRD\rightarrow\mathbb{R}^{d}$ by $\alpha\mapsto\pi_{t}(\alpha)=\alpha_{t}$ and we define %$\rightarrow $ by 
\[\iota: \mu\in\PPC \,\longmapsto \,\iota(\mu)\coloneqq(\mu\circ\pi_{t}^{-1})_{t\in[0, T]}=(\mu_{t})_{t\in[0, T]}\in \CPP.\]

\begin{lem}\label{injectionmeasure}
The application $\iota$ %: \PPC \rightarrow \CPP$ defined by \[\mu\mapsto \iota(\mu)=(\mu\circ\pi_{t}^{-1})_{t\in[0, T]}=(\mu_{t})_{t\in[0, T]}\] 
is well-defined. %, where $\pi_{t}$ is an application defined on $\CRD\rightarrow\mathbb{R}^{d}$ by $ \alpha\mapsto\pi(\alpha)=\alpha_{t}$.
\end{lem}

%The proof of Theorem~\ref{mainthmconv}-(b) is very similar to that of Theorem~\ref{mainthmconv}-(a).  
{\color{black}The proof of Lemma~\ref{injectionmeasure} is postponed in Appendix~\ref{appc}.} For the functional convex order result for the Euler scheme, like in Section~\ref{eulerconv}, we have the following proposition. 
\begin{prop}\label{convschemG}
Let $\bar{X}_{0:M}, \bar{Y}_{0:M}, \bar{\mu}_{0:M}, \bar{\nu}_{0:M}$ be respectively random variables and probability distributions defined by (\ref{defeulerx}) and (\ref{defeulery}). Under Assumption~\ref{AssumptionI} and~\ref{AssumptionII}, for any function 
\[\tilde{G}: (x_{0:M},\eta_{0:M})\in(\RD)^{M+1}\times \big(\PPRD\big)^{M+1}\:\longmapsto \:\tilde{G}(x_{0:M},\eta_{0:M})\in \RR\]
satisfying the following conditions (i), (ii) and (iii)
\begin{enumerate}[$(i)$]
\item $\tilde{G}$ is convex in $x_{0:M}$,
\item $\tilde{G}$ is non-decreasing in $\mu_{0:M}$ with respect to the convex order in the sense that 
\begin{flalign}
&\forall\, x_{0:M}\!\in(\RD)^{M+1} \text{ and } \forall\mu_{0:M}, \nu_{0:M}\!\in\big(\mathcal{P}_{p}(\RD)\big)^{M+1}\; \text{ s.t. }\;\mu_{i}\conright\nu_{i},\; 0\leq i \leq M,&\nonumber\\
&\hspace{5cm}\tilde{G}(x_{0:M}, \mu_{0:M})\leq\tilde{G}(x_{0:M}, \nu_{0:M}),&\nonumber
\end{flalign}
\item $\tilde{G}$ has a  $p$-polynomial growth in the sense that 
\begin{flalign}
&\exists\,  C\in \mathbb{R}_{+}\text{ s.t. } \forall\, (x_{0:M}, \mu_{0:M})\in(\RD)^{M+1}\times \big(\mathcal{P}_{p}(\RD)\big)^{M+1},&\nonumber\\
&\hspace{3cm}\tilde{G}(x_{0:M}, \mu_{0:M})\leq C \big[1+\sup_{0\leq m\leq M}\left|x_{m}\right|^{p}+\sup_{0\leq m\leq M}\mathcal{W}_{p}^{p}(\mu_{m}, \delta_{0})\big],&
\end{flalign}
\end{enumerate}
we have 
\begin{equation}\label{convgeuler}
\EE \tilde{G}(\bar{X}_{0}, \ldots, \bar{X}_{m}, \bar{\mu}_{0}, \ldots, \bar{\mu}_{M})\leq \EE \tilde{G}(\bar{Y}_{0}, \ldots, \bar{Y}_{m}, \bar{\nu}_{0}, \ldots, \bar{\nu}_{M}).
\end{equation}
\end{prop}
The proof of Proposition~\ref{convschemG} is quite similar to that of Proposition~\ref{funconvexEuler}. 
%Firstly, in order to prove the functional convex order result for the Euler scheme{s}~(\ref{EulerX}) and~(\ref{EulerY})
%\begin{equation}\label{convgeuler}
%\EE \tilde{G}(\bar{X}_{0}, \ldots, \bar{X}_{m}, \bar{\mu}_{0}, \ldots, \bar{\mu}_{M})\leq \EE \tilde{G}(\bar{Y}_{0}, \ldots, \bar{Y}_{m}, \bar{\nu}_{0}, \ldots, \bar{\nu}_{M})
%\end{equation}
%with 
%\[\tilde{G}: (x_{0:M},\eta_{0:M})\in(\RD)^{M+1}\times \big(\PPRD\big)^{M+1}\mapsto \tilde{G}(x_{0:M},\eta_{0:M})\in \RR\] 
%convex in $x_{0:M}$, non-decreasing in $\eta_{0:M}$ with respect to the convex order and having an $r$-polynomial growth,   
We just need to replace the definition of $\Phi_{m}$ and $\Psi_{m}$ in~(\ref{defPhi1}), (\ref{defPhi2}) and (\ref{defpsi}) by the following $\Phi_{m}', \Psi_{m}': (\RD)^{m+1}\times\big(\PPRD\big)^{M+1}\rightarrow\RR$,  $m=0, \ldots, M,$ defined by %for every $(x_{0:M},\eta_{0:M})\in (\RD)^{M+1}\times \big(\PPRD\big)^{M+1}$ by
\begin{align}
&\forall\, \, (x_{0:m},\mu_{0:M})\in (\RD)^{m+1}\times \big(\PPRD\big)^{M+1},&\nonumber\\
&\quad\quad\Phi'_{M}(x_{0:M}\,; \mu_{0:M})=\tilde{G}(x_{0:M}, \mu_{0:M}),&\nonumber\\
&\quad\quad \Phi'_{m}(x_{0:m}\,; \mu_{0:M})=\big(Q_{m+1}\Phi'_{m+1}(x_{0:m}, \;\cdot\;\,; \mu_{0:M})\big)\big(x_{m}, \mu_{m}, \sigma_{m}(x_{m}, \mu_{m})\big),&\nonumber \\
&\quad\quad\Psi'_{M}(x_{0:M}\,; \mu_{0:M})=\tilde{G}(x_{0:M}, \mu_{0:M}),&\nonumber\\
&\quad\quad \Psi'_{m}(x_{0:m}\,; \mu_{0:M})=\big(Q_{m+1}\Psi'_{m+1}(x_{0:m}, \;\cdot\;\,; \mu_{0:M})\big)\big(x_{m}, \mu_{m},\theta_{m}(x_{m}, \mu_{m})\big).\nonumber 
\end{align}

\smallskip The key step to prove Theorem~\ref{mainthmconv}$(b)$ starting from (\ref{convgeuler}) is how to define the ``interpolator'' of the marginal distributions $(\bar{\mu}_{t})_{t\in [0, T]}$ and $(\bar{\nu}_{t})_{t\in [0, T]}$.
%Now we discuss the key step from the functional convex order for the Euler scheme~(\ref{convgeuler}) to the joint functional convex order of process and its marginal probability distribution~(\ref{convgpro}). 
Let $\lambda\in[0, 1]$. For any two random variables $X_{1}, X_{2}$ with respective probability distributions $\mu_{1}, \mu_{2}\in\PPRD$, we define the linear combination of $\mu_{1}, \mu_{2}$, denoted by $\lambda \mu_{1}+(1-\lambda)\mu_{2}$, by
\begin{equation}\label{defcombmesure}
\forall\, \, A\in\mathcal{B}(\RD), \hspace{0.5cm} \big(\lambda \mu_{1}+(1-\lambda)\mu_{2}\big)(A)\coloneqq \lambda\mu_{1}(A)+(1-\lambda)\mu_{2}(A). 
\end{equation}
In fact, $\lambda \mu_{1}+(1-\lambda)\mu_{2}$ is  the distribution of the random variable
\begin{equation}\label{aid}
\mathbbm{1}_{\{U\leq \lambda\}}X_{1}+\mathbbm{1}_{\{U>\lambda\}}X_{2},
\end{equation}
where $U$ is a random variable with probability distribution $\mathcal{U}([0,1])$, independent of $(X_{1}, X_{2})$. Then it is obvious that $\lambda \mu_{1}+(1-\lambda)\mu_{2}\in \PPRD$. %In fact, $\lambda \mu_{1}+(1-\lambda)\mu_{2}$ is  the distribution of the random variable
%\[\mathbbm{1}_{\{U\leq \lambda\}}X_{1}+\mathbbm{1}_{\{U>\lambda\}}X_{2},\]
%where $U$ is a random variable with probability distribution $\mathcal{U}([0,1])$ and independent to $(X_{1}, X_{2})$. 
Moreover, with the help of the random variable (\ref{aid}), one proves that 
the application $\lambda\in[0, 1]\mapsto \lambda \mu_{1}+(1-\lambda)\mu_{2}\in\PPRD$ is continuous with respect to $\mathcal{W}_{p}$ for a fixed $(\mu_{1}, \mu_{2})\in\big(\PPRD\big)^{2}$.% 

From (\ref{defcombmesure}), we can extend the definition of the interpolator $i_{M}$ (\textit{respectively} $I_{M}$) on the probability distribution space $\big(\PPRD\big)^{M+1}$ (\textit{resp.} $\mathcal{C}\big([0, T], \PPRD\big)$ ) as follows 
\begin{align}
%i_{M} &: \mu_{0:M}\in(\PPRD)^{M+1}\mapsto i_{M}(\mu_{0:M}) \in \CRD\nonumber\\
%&\forall\, \, t\in[t_{m}^{M}, t_{m+1}^{M}], \;\\
 &\forall\, m=0, \ldots, M-1, \;\forall\, t\in[t_{m}^{M}, t_{m+1}^{M}], \nonumber\\
&\quad\forall\, \, \mu_{0:M}\in\big(\PPRD\big)^{M+1},\hspace{1.25cm}i_{M}(\mu_{0:M})(t)=\frac{M}{T}\big[(t_{m+1}^{M}-t)\mu_{m}+(t-t^{M}_{m})\mu_{m+1}\big],&\nonumber\\
&\quad\forall\, \, (\mu_{t})_{t\in[0, T]}\in\mathcal{C}\big([0, T], \PPRD\big),\hspace{1cm}I_{M}\big((\mu_{t})_{t\in[0, T]}\big)=i_{M}\big(\mu_{t_{0}^{M}}, \ldots, \mu_{t_{M}^{M}}\big).\nonumber
\end{align}

We consider now $\bar{X}^{M}=(\bar{X}_{t}^{M})_{t\in[0, T]}$, $\bar{Y}^{M}=(\bar{Y}_{t}^{M})_{t\in[0, T]}$ defined  by~(\ref{defeulercontinuousx}) and~(\ref{defeulercontinuousy}) with respective probability distributions $\bar{\mu}^{M}, \bar{\nu}^{M}\in\mathcal{P}_{p}\big(\mathcal{C}([0, T], \RD)\big)$ (see (\ref{boundedx})). Let $(\bar{\mu}^{M}_{t})_{t\in[0, T]}=\iota(\bar{\mu}^{M})$ and $(\bar{\nu}^{M}_{t})_{t\in[0, T]}=\iota(\bar{\nu}^{M})$. We define now for every $t\in[0, T]$, $\widetilde{\mu}^{M}_{t}\coloneqq I_{M}\big((\bar{\mu}^{M}_{t})_{t\in[0, T]}\big)_{t}$. By the same idea as (\ref{aid}),  {\color{black}for every $t\in[t_{m}^{M}, t_{m+1}^{M}]$},\, $\widetilde{\mu}^{M}_{t}$ is the probability distribution of the random variable
\[\widetilde{X}_{t}^{M}\coloneqq \mathbbm{1}_{\left\{U_{m}\leq \frac{M\left(t_{m+1}^{M}-t\right)}{T}\right\}}\bar{X}^{M}_{t_{m}}+\mathbbm{1}_{\left\{U_{m}> \frac{M\left(t_{m+1}^{M}-t\right)}{T}\right\}}\bar{X}^{M}_{t_{m+1}},\]
where $(U_{0}, \ldots, U_{M})$ is independent to the Brownian motion $(B_{t})_{t\in[0, T]}$ in~(\ref{defconvx}),~(\ref{defconvy}) and then to $(Z_{0}, \ldots, Z_{M})$ in~(\ref{defeulerx}),~(\ref{defeulery}).

Now we prove that $(\widetilde{\mu}^{M}_{t})_{t\in[0, T]}$ converges to the weak solution $(\mu_{t})_{t\in[0, T]}$ of (\ref{defconvx}) with respect to the distance $d_{\mathcal{C}}$ defined in (\ref{distancedc}). 
%Now we prove that $\sup_{t\in[0, T]}\mathcal{W}_{p}(\bar{\mu}_{t}^{M}, \widetilde{\mu}^{M}_{t})\rightarrow0$ as $M\rightarrow+\infty$. 
%Let $\bar{\mu}^{M}$ and $\bar{\nu}^{M}$ denote the probability distribution of $\bar{X}^{M}=(\bar{X}_{t}^{M})_{t\in[0, T]}$ and $\bar{Y}^{M}=(\bar{Y}_{t}^{M})_{t\in[0, T]}$ defined by~(\ref{defeulercontinuousx}) and~(\ref{defeulercontinuousy}). For every $t\in[0, T]$, let $\widetilde{\mu}^{M}_{t}\coloneqq I_{M}\big((\bar{\mu}^{M}_{t})_{t\in[0, T]}\big)_{t}$. 
We know from Proposition~\ref{cvgeuler}-(c) %Lemma~\ref{relationd} and Corollary~\ref{coleulercontinuous} 
that for any $p\geq2$
\begin{equation}\label{conG1}
\sup_{t\in[0, T]}\mathcal{W}_{p}(\mu_{t}, \bar{\mu}^{M}_{t})\rightarrow 0 \,\,\text{ as }\, M\rightarrow+\infty.
%\sup_{t\in[0, T]}\mathcal{W}_{p}(\mu_{t}, \bar{\mu}^{M}_{t})\leq \mathbb{W}_{p}(\mu, \bar{\mu}^{M})\rightarrow 0 \,\,\text{ as }\, M\rightarrow+\infty.
\end{equation}
%Now we prove that $\sup_{t\in[0, T]}\mathcal{W}_{p}(\bar{\mu}_{t}^{M}, \widetilde{\mu}^{M}_{t})\rightarrow0$ as $M\rightarrow+\infty$. For every $t\in[t_{m}^{M}, t_{m+1}^{M}]$, let 
%\[\widetilde{X}_{t}^{M}\coloneqq \mathbbm{1}_{\left\{U_{m}\leq \frac{M\left(t_{m+1}^{M}-t\right)}{T}\right\}}\bar{X}^{M}_{t_{m}}+\mathbbm{1}_{\left\{U_{m}> \frac{M\left(t_{m+1}^{M}-t\right)}{T}\right\}}\bar{X}^{M}_{t_{m+1}},\]
%where $(U_{0}, \ldots, U_{M})$ is independent to the Brownian Motion $(B_{t})_{t\in[0, T]}$ in~(\ref{defconvx}),~(\ref{defconvy}) and $(Z_{0}, \ldots, Z_{M})$ in~(\ref{defeulerx}), ~(\ref{defeulery}). Thus,  for every $t\in[t_{m}^{M}, t_{m+1}^{M}]$, $\widetilde{X}_{t}^{M}$ has the probability distribution $\widetilde{\mu}_{t}^{M}$. 
It follows that 
\begin{flalign}
&\forall\, \, m\in\{0, \ldots, M\}, \forall\, \, t\in[t_{m}^{M}, t_{m+1}^{M}], &\nonumber\\
&\hspace{1cm}\mathcal{W}_{p}^{p}(\bar{\mu}^{M}_{t}, \widetilde{\mu}^{M}_{t})\leq \EE \left|\bar{X}^{M}_{t}-\widetilde{X}^{M}_{t}\right|^{p}\nonumber\\
&\hspace{3cm}=\EE \left|\bar{X}^{M}_{t}-\mathbbm{1}_{\left\{U_{m}\leq \frac{M\left(t_{m+1}^{M}-t\right)}{T}\right\}}\bar{X}^{M}_{t_{m}}-\mathbbm{1}_{\left\{U_{m}> \frac{M\left(t_{m+1}^{M}-t\right)}{T}\right\}}\bar{X}^{M}_{t_{m+1}}\right|^{p}&\nonumber\\
&\hspace{3cm}\leq \EE \left|\bar{X}^{M}_{t}-\bar{X}^{M}_{t_{m}}\right|^{p}+\EE \left|\bar{X}^{M}_{t}-\bar{X}^{M}_{t_{m+1}}\right|^{p}.&\nonumber
\end{flalign}
%where $C_{p}$ is a constant depending only on $p$.
We derive from Proposition~\ref{cvgeuler}-$(b)$ that 
\begin{align}
\forall\, \, s, t\in[t_{m}^{M}, t_{m+1}^{M}], \;s<t, \hspace{1cm}\EE \left|\bar{X}^{M}_{t}-\bar{X}^{M}_{s}\right|^{p}\leq(\kappa\sqrt{\;t-s\;})^{p}\leq \kappa^{p}(\tfrac{T}{M})^{\frac{p}{2}}\rightarrow 0, \text{ as } M\rightarrow+\infty.\nonumber
\end{align}
Thus, we have $\sup_{t\in[0, T]}\mathcal{W}_{p}^{p}(\bar{\mu}^{M}_{t}, \widetilde{\mu}^{M}_{t})\rightarrow 0$ as $M\rightarrow +\infty$. Hence,  %and by~(\ref{conG1}), 
\begin{align}\label{cvgdc}
d_{\mathcal{C}}\big((\widetilde{\mu}^{M}_{t})_{t\in[0, T]}, (\mu_{t})_{t\in[0, T]}\big)=&\sup_{t\in[0, T]}\mathcal{W}_{p}^{p}(\mu_{t}, \widetilde{\mu}^{M}_{t})\leq \sup_{t\in[0, T]}\mathcal{W}_{p}^{p}(\bar{\mu}^{M}_{t}, \mu_{t})+\sup_{t\in[0, T]}\mathcal{W}_{p}^{p}(\bar{\mu}^{M}_{t}, \widetilde{\mu}^{M}_{t})\nonumber\\
&\,\longrightarrow 0 \,\text{ as }\, M\rightarrow+\infty. 
\end{align}

\begin{proof}[Proof of Theorem~\ref{mainthmconv}-$(b)$]

We define for every $(x_{0:M}, \eta_{0:M})\in(\RD)^{M+1}\times\big(\PPRD\big)^{M+1}$, $G_{M}(x_{0:M}, \eta_{0:M})\coloneqq G\big(i_{M}(x_{0:M}), i_{M}(\eta_{0:M})\big)$.  For every fixed $M\in\mathbb{N}^{*}$, we have
\begin{align}
\EE G&\big(I_{M}(\bar{X}^{M}), (\widetilde{\mu}^{M}_{t})_{t\in[0, T]}\big)=\EE G\Big(I_{M}(\bar{X}^{M}),I_{M}\big((\bar{\mu}_{t}^{M})_{t\in[0, T]}\big) \Big)\nonumber\\
&=\EE G\big(i_{M}(\bar{X}_{t_{0}}^{M}, \ldots, \bar{X}_{t_{M}}^{M}), i_{M}(\bar{\mu}_{t_{0}}^{M}, \ldots, \bar{\mu}_{t_{M}}^{M})\big)=\EE G_{M}\big(\bar{X}_{t_{0}}^{M}, \ldots, \bar{X}_{t_{M}}^{M}, \bar{\mu}_{0}^{M}, \ldots, \bar{\mu}_{t_{M}}^{M}\big)\nonumber\\
&\leq \EE G_{M}\big(\bar{Y}_{t_{0}}^{M}, \ldots, \bar{Y}_{t_{M}}^{M}, \bar{\nu}_{t_{0}}^{M}, \ldots, \bar{\nu}_{t_{M}}^{M}\big)\quad\text{(by Proposition~\ref{convschemG})}\nonumber\\
&=\EE G\big(i_{M}(\bar{Y}_{t_{0}}^{M}, \ldots, \bar{Y}_{t_{M}}^{M}), i_{M}(\bar{\nu}_{t_{0}}^{M}, \ldots, \bar{\nu}_{t_{M}}^{M})\big)=\EE G\big(I_{M}\big(\bar{Y}^{M}), (\widetilde{\nu}_{t}^{M})_{t\in[0, T]}\big).\nonumber%=\EE G\big(I_{M}(\bar{Y}^{M}), (\tilde{\nu}_{t})_{t\in[0, T]}\big),
\end{align}
%Consequently, for a fixed $M\in\mathbb{N}^{*}$, we have
%\begin{align}
%\EE G&\big(I_{M}(\bar{X}^{M}), (\widetilde{\mu}_{t})_{t\in[0, T]}\big)=\EE G\big(I_{M}(\bar{X}^{M}),I_{M}\big((\bar{\mu}_{t}^{M})_{t\in[0, T]}\big) \big)\nonumber\\
%&=\EE G\big(i_{M}(\bar{X}_{t_{0}}^{M}, \ldots, \bar{X}_{t_{M}}^{M}), i_{M}(\bar{\mu}_{t_{0}}^{M}, \ldots, \bar{\mu}_{t_{M}}^{M})\big)=\EE G_{M}\big(\bar{X}_{t_{0}}^{M}, \ldots, \bar{X}_{t_{M}}^{M}, \bar{\mu}_{0}^{M}, \ldots, \bar{\mu}_{t_{M}}^{M}\big)\nonumber\\
%&\leq \EE G_{M}\big(\bar{Y}_{t_{0}}^{M}, \ldots, \bar{Y}_{t_{M}}^{M}, \bar{\nu}_{t_{0}}^{M}, \ldots, \bar{\nu}_{t_{M}}^{M}\big)=\EE G\big(i_{M}(\bar{Y}_{t_{0}}^{M}, \ldots, \bar{Y}_{t_{M}}^{M}), i_{M}(\bar{\nu}_{t_{0}}^{M}, \ldots, \bar{\nu}_{t_{M}}^{M})\big)\nonumber\\
%&=\EE G\big(I_{M}\big(\bar{Y}^{M}), (\widetilde{\nu}_{t}^{M})_{t\in[0, T]}\big)\big),%=\EE G\big(I_{M}(\bar{Y}^{M}), (\tilde{\nu}_{t})_{t\in[0, T]}\big),
%\end{align}
%where for any $(x_{0:M}, \eta_{0:M})\in(\RD)^{M+1}\times\big(\PPRD\big)^{M+1}$, $G_{M}(x_{0:M}, \eta_{0:M})\coloneqq G\big(i_{M}(x_{0:M}), i_{M}(\eta_{0:M})\big)$.  
%Then one checks that $\Big(G\big(I_{M}(\bar{X}^{M}), (\widetilde{\mu}^{M}_{t})_{t\in[0, T]}\big)\Big)_{M\geq1}$, $\Big(G\big(I_{M}(\bar{Y}^{M}), (\widetilde{\nu}^{M}_{t})_{t\in[0, T]}\big)\Big)_{M\geq1}$ are  uniformly integrable. 
\textcolor{black}{Using the continuity assumption on $G$ (see Theorem~\ref{mainthmconv}-$(b)$-$(iii)$) and the convergence in (\ref{cvgdc}) imply  weak convergence of both sequences of random variables. Then using that $G$ has at most $p$-polynomial growth  in both space and measure arguments, on concludes like for claim $(a)$ that $\EE\,  G\big(I_{M}(\bar{X}^{M}), (\widetilde{\mu}^{M}_{t})_{t\in[0, T]}\big) \to \EE\, G\big(X, (\mu_t)_{t\in [0,T]}\big)$ (idem for $Y$)}.  Combining these two properties,  we finally obtain~\eqref{convgpro} by letting $M\rightarrow +\infty$.
%{\color{red}by applying the Lebesgue dominated convergence theorem,} the continuity assumption on $G$ (see Theorem~\ref{mainthmconv}-(b)-$(iii)$) and the convergence in (\ref{cvgdc}).
\end{proof}

%\noindent{\bf \color{red}Provisional remark.} 

%In dimension 1, if we consider the monotone convex order
%Using a similar strategy of the proof in this paper,  it seems likely to generalize Hajek's theorem (\cite{MR771469}) about (non-decreasing)  monotone convex order for Brownian diffusions to McKean-Vlasov equations with drifts  i.e. establish the propagation of (non-decreasing)  monotone convex order by such equations under appropriate assumptions on their drifts.

\noindent\textbf{Acknowledgement.} The authors thank both the anonymous reviewer and the associate editor for their careful reading of the paper. They are grateful for their constructive and insightful comments and suggestions. The first author would like to thank Dr. Julien Claisse for his helpful advice. 
%%%%%%%%%%%%%%%%%%%%%%%%%%%%%%%%%%%%%%%%%%%%%%the reviewers for all of their careful, constructive and insightful comments
%% Example with single Appendix:            %%
%%%%%%%%%%%%%%%%%%%%%%%%%%%%%%%%%%%%%%%%%%%%%%

\bibliographystyle{alpha}
\bibliography{convex_order}

\appendixtitleon
\appendixtitletocon

\begin{appendices}

\section{Convergence rate of the Euler scheme for the McKean-Vlasov equation}\label{appn} %% if no title is needed, leave empty \section*{}.

We prove Proposition~\ref{cvgeuler} in this section. 
Under Assumption~\ref{AssumptionI}, the functions $b$ and $\sigma$ have a linear growth in $x$ and in $\mu$ in the sense that there exists a constant $C_{b,\sigma, L, T}$ depending on $b, \sigma$, $L$ and $T$ such that for any $(t, x, \mu)\in[0, T]\times \mathbb{R}^{d}\times \mathcal{P}_{p}(\mathbb{R}^{d})$,
\begin{equation}\label{lineargrowth}
%\forall\, \, t\in[0, T], \forall\, \, x\in\mathbb{R}^{d}, \forall\, \, \mu\in\mathcal{P}_{p}(\mathbb{R}^{d}), \;
\left|b(t, x,\mu)\right|\vee\vertiii{\sigma(t, x,\mu)}\leq C_{b,\sigma, L, T}(1+\left|x\right|+\mathcal{W}_{p}(\mu,\delta_{0})),
\end{equation}
since for any $x\in\mathbb{R}^{d}$ and for any $\mu\in\mathcal{P}_{p}(\mathbb{R}^{d})$, we have for every $t\in[0, T]$,
\[\left|b(t, x, \mu)\right|\leq\left|b(t, 0, \delta_{0})\right| + L\big(\left|x\right|+\mathcal{W}_{p}(\mu, \delta_{0})\big)\leq (\left|b(t, 0, \delta_{0})\right| \vee L)(1+\left|x\right|+\mathcal{W}_{p}(\mu, \delta_{0}))\]
and $\vertiii{\sigma(t, x, \mu)}\leq (\vertiii{\sigma(t, 0,  \delta_{0})} \vee L)(1+\left|x\right|+\mathcal{W}_{p}(\mu,  \delta_{0}))$ by applying (\ref{assumplip}) so that one can take \textit{e.g.} $C_{b,\sigma, L, T}\coloneqq\sup_{t\in[0, T]}\left|b(t, 0, \delta_{0})\right|\vee\sup_{t\in[0, T]}\vertiii{\sigma(t, 0, \delta_{0})}\vee L$.

Moreover, the definition of continuous time Euler scheme (\ref{defeulercontinuousx}) implies that 
 $\bar{X}\coloneqq(\bar{X}_{t})_{t\in[0, T]}$ is a $\CRD$-valued stochastic process. %since $(B_{t})_{t\in[0, T]}$ is pointwise continuous, which ensures the (left)-continuity . 
Let $\bar{\mu}$ denote the probability distribution of $\bar{X}$ and for every $t\in[0, T]$, let $\bar{\mu}_{t}$ denote the marginal distribution of $\bar{X}_{t}$. 
 %let $(\bar{\mu}_{t})_{t\in[0, T]}=\iota(\bar{\mu})$. 
Then $(\bar{X}_{t})_{t\in[0, T]}$ is the solution of 
\begin{equation}\label{eulerconti2}
\begin{cases}
d\bar{X}_{t}=b(\underline{t}, \bar{X}_{\underline{t}}, \bar{\mu}_{\underline{t}})dt+\sigma(\underline{t}, \bar{X}_{\underline{t}}, \bar{\mu}_{\underline{t}})dB_{t},\\
\bar{X}_{0}=X_{0},
\end{cases}
\end{equation}
where for every $t\in[t_{m}, t_{m+1})$, $\underline{t}\coloneqq t_{m}$. 

Now we recall a variant of Gronwall's Lemma (see Lemma~7.3 in~\cite{pages2018numerical} for a proof) and two important technical tools used throughout the proof: the generalized Minkowski Inequality and the Burk\"older-Davis-Gundy Inequality.  We refer to~\cite[Section 7.8]{pages2018numerical} and \cite[Chapter IV - Section 4]{MR1725357} for proofs  (among many others).
\begin{lem}[``\`A la Gronwall'' Lemma]\label{Gronwall}
Let $f : [0, T]\rightarrow\mathbb{R}_{+}$ be a Borel, locally bounded, non-negative and non-decreasing function and let $\psi: [0, T]\rightarrow\mathbb{R}_{+}$ be a non-negative non-decreasing function satisfying 
\[\forall\, \, t\in[0, T],\,\, f(t)\leq A\int_{0}^{t}f(s)ds+B\left(\int_{0}^{t}f^{2}(s)ds\right)^{\frac{1}{2}}+\psi(t),\]
where $A, B$ are two positive real constants. Then, for any $t\in[0, T],$ \[ f(t)\leq 2e^{(2A+B^{2})t}\psi(t).\]
\end{lem}
\begin{prop}[The Generalized Minkowski Inequality]\label{gemin}
For any (bi-measurable) process $X=(X_{t})_{t\geq0}$, for every $p\in[1,\infty)$ and for every $ T\in[0, +\infty],$
\begin{equation}\label{eq:GMinl}\left\Vert \int_{0}^{T}X_{t}dt\right\Vert_{p}\leq\int_{0}^{T}\left\Vert X_{t}\right\Vert_{p}dt.
\end{equation}
\end{prop}
\begin{thm}[Burk\"older-Davis-Gundy Inequality (continuous time)]\label{BDGin}
For every $p\in(0, +\infty)$, there exists two real constants $C_{p}^{BDG}>c_{p}^{BDG}>0$ such that, for every continuous local martingale $(X_{t})_{t\in[0, T]}$ null at 0 with sharp bracket process $\big(\langle X \rangle_{t}\big)_{t\in[0, T]}$,
\[c_{p}^{BDG}\left\Vert \sqrt{\langle X \rangle_{T}}\right\Vert_{p}\leq\left\Vert \sup_{t\in[0,T]}\left|X_{t}\right|\right\Vert_{p}\leq C_{p}^{BDG}\left\Vert \sqrt{\langle X\rangle_{T}}\right\Vert_{p}.\]
\end{thm}
In particular, if $(B_{t})$ is an $(\mathcal{F}_{t})$-standard Brownian motion and $(H_{t})_{t\geq0}$ is an $(\mathcal{F}_{t})$-progressively measurable process having values in $\mathbb{M}_{d\times q}(\mathbb{R})$ such that $\int_{0}^{T}\left\Vert H_{t}\right\Vert^{2}dt<+\infty$ $\;\mathbb{P}-a.s.$, then the $d$-dimensional local martingale $\int_{0}^{\cdot}H_{s}dB_{s}$ satisfies
\begin{equation}\label{BDGinequality}
\left\Vert\sup_{t\in[0, T]}\left|\int_{0}^{t}H_{s}dB_{s}\right|\right\Vert_{p}\leq C_{d,p}^{BDG}\left\Vert\sqrt{\int_{0}^{T}\left\Vert H_{t}\right\Vert^{2}dt}\right\Vert_{p}.
\end{equation}
where $C_{d,p}^{BDG}$ only depends on $p,d$. %The proof of Proposition~\ref{Srate} relies on the following lemma. 
%\begin{lem}\label{eulertheorical}
%Under Assumption~\ref{AssumptionI}, let $X$ be the unique strong solution of~(\ref{part2mckeandef}) and let $(\bar{X}_{t})_{t\in[0, T]}$ be the process defined in~(\ref{eulercontinuous}). Then 

%\begin{enumerate}[(a)]
%\item %\begin{equation}\label{moments}
%There exists a constant $C_{p,d,b,\sigma}$ depending on $p, d, b, \sigma$ %and not on $M$ 
%such that for every $t\in[0, T]$,
%\[\forall\, \, M\geq 1, \quad \left\Vert \sup_{u\in[0, t]}\left| X_{u}\right|\right\Vert_{p}\vee\vertii{\sup_{u\in[0,t]}\left|\bar{X}^{M}_{u}\right|}_{p}\leq C_{p,d,b,\sigma}e^{C_{p,d,b,\sigma}t}(1+\left\Vert X_{0}\right\Vert_{p}).\]
%\item There exists a constant $\kappa$ depending on $L, b, \sigma, \left\Vert X_{0}\right\Vert, p, d, T$ such that for any $s,t\in[0, T], s<t$, 
%\[\forall\, \, M\geq 1, \quad\left\Vert \bar{X}^{M}_{t}-\bar{X}^{M}_{s}\right\Vert_{p}\vee\left\Vert X_{t}-X_{s}\right\Vert_{p}\leq \kappa\sqrt{t-s}.\]
%\item Let $(\bar{\mu}_{t})_{t\in[0, T]}$ denote the marginal distributions of $(\bar{X}_{t})_{t\in[0, T]}$ defined in~(\ref{eulercontinuous}), then $(\bar{\mu}_{t})_{t\in[0, T]}\in\mathcal{P}_{p}\big(\mathcal{C}([0, T], \mathbb{R}^{d})\big)$. {\color{red}but it is not correct}
%\end{enumerate}
%\end{lem}
%The proof of Lemma~\ref{eulertheorical} is postponed to Appendix B.

%We first prove 
\begin{proof}[Proof of Proposition~\ref{cvgeuler}-$(a)$]
(a) If $X$ is the unique strong solution of $~(\ref{defconvx})$, then its probability distribution $\mu$ is the unique weak solution. We define two new coefficient functions depending on $\iota(\mu)=(\mu_{t})_{t\in[0, T]}$ by 
\[\tilde{b}(t, x)\coloneqq b(t, x, \mu_{t}) \quad\text{ and }\quad \tilde{\sigma}(t, x)\coloneqq \sigma(t, x, \mu_{t}).\]
Now we discuss the continuity in $t$ of $\tilde{b}$ and $\tilde{\sigma}$. 
%Then $\tilde{b}$ and $\tilde{\sigma}$ are still continuous in $t$. 
In fact,  
\begin{align}
\left|\tilde{b}(t, x)-\tilde{b}(s, x)\right|&\leq \left|b(t, x, \mu_{t})-b(s, x, \mu_{s})\right|\nonumber\\
&\leq  \left|b(t, x, \mu_{t})-b(s, x, \mu_{t})\right|+ \left|b(s, x, \mu_{t})-b(s, x, \mu_{s})\right|\nonumber\\
&\leq  \left|b(t, x, \mu_{t})-b(s, x, \mu_{t})\right|+\mathcal{W}_{p}(\mu_{t}, \mu_{s}),
\end{align}
and we have a similar inequality for $\tilde{\sigma}$. 
Moreover, we know from Assumption~\ref{AssumptionI} that $b$ and $\sigma$ are continuous in $t$ and from Lemma~\ref{injectionmeasure} that $\iota(\mu)=(\mu_{t})_{t\in[0, T]}\in\mathcal{C}\big([0, T], \PPRD\big)$. Hence, $\tilde{b}$ and $\tilde{\sigma}$ are continuous in $t$. Moreover, it is obvious that $\tilde{b}$ and $\tilde{\sigma}$ are still Lipschitz continuous in $x$. 
Consequently, $X$ is also the unique strong solution of the following stochastic differential equation
\[
dX_{t}=\tilde{b}(t, X_{t})dt + \tilde{\sigma}(t, X_{t})dB_{t}, \quad X_{0} \text{ same as in ~(\ref{defconvx}).}
\]
%with $X_{0}$ is like  in $$. 
Hence, the inequality 
\[\left\Vert \sup_{u\in[0, t]}\left| X_{u}\right|\right\Vert_{p}\leq C_{p,d,b,\sigma}e^{C_{p,d,b,\sigma}t}(1+\left\Vert X_{0}\right\Vert_{p})\]
can be obtained by the usual method for the regular stochastic differential equation for which we refer to~\cite[Proposition 7.2 and (7.12)]{pages2018numerical} among many other references. 

\smallskip
Next, we prove the inequality for $\vertii{\sup_{u\in[0,t]}\left|\bar{X}_{u}^{M}\right|}_{p}$. 

We go back to the discrete Euler scheme \[
%\begin{cases}
%M\in\mathbb{N}, h=\frac{T}{M}, t_{m}=m\cdot h=m\cdot \frac{T}{M}\\
\bar{X}_{t_{m+1}}^{M}=\bar{X}_{t_{m}}^{M}+h\cdot b(t_{m}, \bar{X}_{t_{m}}^{M},\bar{\mu}_{t_{m}}^{M})+\sqrt{h\,} \sigma(t_{m}, \bar{X}_{t_{m}}^{M}, \bar{\mu}_{t_{m}}^{M})Z_{m+1}.\]
We write $\bar{X}_{t_{m}}$ instead of $\bar{X}_{t_{m}}^{M}$ in the following. 
By Minkowski's inequality, we have 
\[\vertii{\bar{X}_{t_{m+1}}}_{p}=\vertii{\bar{X}_{t_{m}}}_{p}+h\vertii{ b(t_{m}, \bar{X}_{t_{m}},\bar{\mu}_{t_{m}})}_{p}+\sqrt{h\,} \Big\Vert\vertiii{\sigma(t_{m}, \bar{X}_{t_{m}}, \bar{\mu}_{t_{m}})}\left|Z_{m+1}\right|\Big\Vert_{p}.\]
As $Z_{m+1}$ is independent of the $\sigma-$algebra generated by $\bar{X}_{t_{0}}, \ldots, \bar{X}_{t_{m}}$, one can apply the linear growth result in  (\ref{lineargrowth}) and obtain %that there exist two constant $C$ and $c_{p}$ such that 
\[\vertii{\bar{X}_{t_{m+1}}}_{p}=\vertii{\bar{X}_{t_{m}}}_{p}+C_{b,\sigma, L, T}(h+c_{p} h^{1/2})\big(1+\vertii{\bar{X}_{t_{m}}}_{p}+\mathcal{W}_{p}(\delta_{0}, \bar{\mu}_{t_{m}})\big),\]
where $C_{b,\sigma, L, T}$ and $c_{p}$ are two real constants. 
As $\mathcal{W}_{p}(\delta_{0}, \bar{\mu}_{t_{m}})\leq \vertii{\bar{X}_{t_{m}}}_{p}$, {\color{black}there exists  constants $C_1$ and $C_2$ such that 
\[\vertii{\bar{X}_{t_{m+1}}}_{p}\leq C_1 \vertii{\bar{X}_{t_{m}}}_{p}+C_2,\]}
which in turn implies by induction that $\displaystyle\max_{m=0, \ldots, M}\vertii{\bar{X}_{t_{m}}}_{p}<+\infty$ since \[\vertii{\bar{X}_{0}}_{p}=\vertii{X_{0}}_{p}<+\infty.\] 

For every $t\in[t_{m}, t_{m+1})$, it follows from the definition~(\ref{defeulercontinuousx}) that 
\[
\vertii{\bar{X}^{M}_{t}}_{p}\leq\vertii{\bar{X}_{t_{m}}}_{p}+(t-t_{m})\vertii{b(t_{m}, \bar{X}_{t_{m}}, \bar{\mu}_{t_{m}})}_{p}+\Big\Vert\vertiii{\sigma(t_{m},\bar{X}_{t_{m}}, \bar{\mu}_{t_{m}})}\left|B_{t}-B_{t_{m}}\right|\Big\Vert_{p}.
\]
We write $\bar{X}_{t}$ instead of $\bar{X}^{M}_{t}$ in the following when there is no ambiguity. 

As $B_{t}-B_{t_{m}}$ is independent to $\sigma(\mathcal{F}_{s}, s\leq t_{m})$, it follows that 
\begin{align}
\vertii{\bar{X}_{t}}_{p}&\leq\vertii{\bar{X}_{t_{m}}}_{p}+C_{b,\sigma, L, T}\big(1+\vertii{\bar{X}_{t_{m}}}_{p}+ \mathcal{W}_{p}(\delta_{0}, \bar{\mu}_{t_{m}})\big)\big(h+c_{d, p}(t-t_{m})^{1/2}\big)\nonumber\\
& \leq C_{1} \vertii{\bar{X}_{t_{m}}}_{p}+C_{2}, \nonumber
\end{align}
where $C_{1}$ and $C_{2}$ are two constants. Finally, {\color{black} for a fixed $M\geq1$,}
\begin{equation}\label{whatlost}
\sup_{t\in[0, T]}\vertii{\bar{X}_{t}^{M}}_{p}<+\infty.
\end{equation}

Consequently, 
\begin{align}
&\left\Vert\sup_{u\in[0, t]}\left| \bar{X}_{u}^{M}\right|\right\Vert_{p} \nonumber\\
&\leq\left\Vert X_{0}\right\Vert_{p}+\left\Vert \int_{0}^{t}\left|b(s, \bar{X}_{\underline{s}}, \bar{\mu}_{\underline{s}})\right|ds\right\Vert_{p}+\left\Vert \sup_{u\in[0, t]}\left|\int_{0}^{u}\sigma(s, \bar{X}_{\underline{s}}, \bar{\mu}_{\underline{s}})dB_{s}\right|\right\Vert_{p}\nonumber\\
&\quad\text{(by Minkowski's Inequality)}\nonumber\\
&\leq \left\Vert X_{0}\right\Vert_{p}+\int_{0}^{t}\left\Vert b(s, \bar{X}_{\underline{s}}, \bar{\mu}_{\underline{s}})\right\Vert_{p}ds+C_{d,p}^{BDG}\left\Vert \sqrt{\int_{0}^{t}\vertiii{\sigma(s, \bar{X}_{\underline{s}}, \bar{\mu}_{\underline{s}})}^{2}ds}\right\Vert_{p}\nonumber\\
&\quad\text{(by Lemma~\ref{gemin} and~\eqref{BDGinequality})}\nonumber\\
&\leq  \left\Vert X_{0}\right\Vert_{p} + \int_{0}^{t} C_{b,\sigma, L, T}\left\Vert 1+\left|\bar{X}_{\underline{s}}\right|+\mathcal{W}_{p}(\bar{\mu}_{\underline{s}}, \delta_{0})\right\Vert_{s}ds \nonumber\\
&\hspace{3cm}+ C_{d,p, L}^{BDG}\left\Vert \sqrt{\int_{0}^{t}\big| 1+\left| \bar{X}_{\underline{s}}\right|+\mathcal{W}_{p}(\bar{\mu}_{\underline{s}},\delta_{0})\big|^{2}ds}\right\Vert_{p}\;(\text{by }(\ref{lineargrowth}))\nonumber\\
&\leq \left\Vert X_{0}\right\Vert_{p} +\int_{0}^{t}C_{b,\sigma, L, T}(1+2\left\Vert \bar{X}_{\underline{s}}\right\Vert_{p})ds\nonumber\\
&\hspace{3cm}+C_{d,p, L}^{BDG}\left\Vert \sqrt{\int_{0}^{t}4\big(1+\left| \bar{X}_{\underline{s}}\right|^{2}+\mathcal{W}_{p}^{2}(\bar{\mu}_{\underline{s}},\delta_{0})\big)ds}\right\Vert_{p}\nonumber\\
&\leq \left\Vert X_{0}\right\Vert_{p} +\int_{0}^{t}C_{b,\sigma, L, T}(1+2\left\Vert \bar{X}_{\underline{s}}\right\Vert_{p})ds\nonumber\\
&\hspace{3cm}+C_{d,p, L}^{BDG}\left\Vert \sqrt{4\big[t+\int_{0}^{t}\left|\bar{X}_{\underline{s}}\right|^{2}ds+\int_{0}^{t}\mathcal{W}_{p}^{2}(\bar{\mu}_{\underline{s}}, \delta_{0})ds\big]}\right\Vert_{p}\nonumber\\
&\leq \left\Vert X_{0}\right\Vert_{p} +\int_{0}^{t}C_{b,\sigma, L, T}(1+2\left\Vert \bar{X}_{\underline{s}}\right\Vert_{p})ds\nonumber\\
&\hspace{3cm}+C_{d,p, L}^{BDG'}\left\Vert \sqrt{t}+\sqrt{\int_{0}^{t}\left|\bar{X}_{\underline{s}}\right|^{2}ds}+\sqrt{\int_{0}^{t}\mathcal{W}_{p}^{2}(\bar{\mu}_{\underline{s}}, \delta_{0})ds}\right\Vert_{p}\nonumber\\
&\leq \left\Vert X_{0}\right\Vert_{p} +\int_{0}^{t}C_{b,\sigma, L, T}(1+2\left\Vert \bar{X}_{\underline{s}}\right\Vert_{p})ds\nonumber\\
&\hspace{3cm}+C_{d,p, L}^{BDG'}\left[ \sqrt{t}+\left\Vert\sqrt{\int_{0}^{t}\left|\bar{X}_{\underline{s}}\right|^{2}ds}\right\Vert_{p}+\sqrt{\int_{0}^{t}\mathcal{W}_{p}^{2}(\bar{\mu}_{\underline{s}}, \delta_{0})ds} \;\right]\nonumber\\
&\leq \left\Vert X_{0}\right\Vert_{p} +\int_{0}^{t}C_{b,\sigma, L, T}(1+2\left\Vert \bar{X}_{\underline{s}}\right\Vert_{p})ds\nonumber\\
&\hspace{3cm}+C_{d,p, L}^{BDG'}\left[ \sqrt{t}+\left\Vert\int_{0}^{t}\left|\bar{X}_{\underline{s}}\right|^{2}ds\right\Vert_{\frac{p}{2}}^{\frac{1}{2}}+\left(\int_{0}^{t}\mathcal{W}_{p}^{2}(\bar{\mu}_{\underline{s}}, \delta_{0})ds\right)^{\frac{1}{2}}\;\right]\nonumber\\
&\leq \left\Vert X_{0}\right\Vert_{p} +\int_{0}^{t}C_{b,\sigma, L, T}(1+2\left\Vert \bar{X}_{\underline{s}}\right\Vert_{p})ds\nonumber\\
&\hspace{3cm}+C_{d,p, L}^{BDG'}\left[ \sqrt{t}+\Big[\int_{0}^{t}\left\Vert \left|\bar{X}_{\underline{s}}\right|^{2}\right\Vert_{\frac{p}{2}}ds\Big]^{\frac{1}{2}}+\Big[\int_{0}^{t}\mathcal{W}_{p}^{2}(\bar{\mu}_{\underline{s}}, \delta_{0})ds\Big]^{\frac{1}{2}}\right]\nonumber\\
%&\;\;(p\geq2, \text{ then } \frac{p}{2}\geq1, \text{then one can use the generalized Minkowski Inequality})
& \quad \quad(\text{by Lemma~\ref{gemin} since } \frac{p}{2}\geq1).\nonumber
\end{align}
It follows from  $\left\Vert\left|\bar{X}_{\underline{s}}\right|^{2}\right\Vert_{\frac{p}{2}}=\big[\mathbb{E}\left|\bar{X}_{\underline{s}}\right|^{2\cdot\frac{p}{2}}\big]^{\frac{2}{p}}=\left\Vert \bar{X}_{\underline{s}}\right\Vert_{p}^{2}$ and 
\[
\left[\int_{0}^{t}\mathcal{W}_{p}^{2}(\bar{\mu}_{\underline{s}}, \delta_{0})ds\right]^{\frac{1}{2}}\leq  \left[\int_{0}^{t}\left\Vert\mathcal{W}_{p}(\bar{\mu}_{\underline{s}}, \delta_{0})\right\Vert_{p}^{2}ds\right]^{\frac{1}{2}}\leq \left[\int_{0}^{t}\left\Vert \bar{X}_{\underline{s}}\right\Vert_{p}^{2}ds\right]^{\frac{1}{2}}
\]
that
\begin{align}
\left\Vert\sup_{u\in[0, t]}\left| \bar{X}_{u}^{M}\right|\right\Vert_{p}\leq &\left\Vert X_{0}\right\Vert_{p}+\int_{0}^{t}C_{b,\sigma, L, T}\big(1+2\left\Vert \bar{X}_{\underline{s}}\right\Vert_{p}\big)ds\nonumber\\
&\label{whatlost2}\hspace{3cm}+C_{d,p, L}^{BDG'}\left(\sqrt{t}+\left[\int_{0}^{t}\left\Vert \bar{X}_{\underline{s}}\right\Vert_{p}^{2}ds\right]^{\frac{1}{2}}\right).
\end{align}
Hence,~(\ref{whatlost2}) implies that, for every $M\geq1$, one has
$\left\Vert\sup_{u\in[0, T]}\left| \bar{X}_{u}^{M}\right|\right\Vert_{p}<+\infty$ 
by applying (\ref{whatlost}).

In order to establish the uniformity in $M$, we come back to~(\ref{whatlost2}). As $\vertii{\bar{X}_{\underline{s}}}_{p}\leq \vertii{\sup_{u\in[0,s]}\left|\bar{X}_{u}\right|}_{p}$, it follows that
%It follows that 
\begin{align}
\left\Vert\sup_{u\in[0, t]}\left| \bar{X}_{u}^{M}\right|\right\Vert_{p}\leq&\left\Vert X_{0}\right\Vert_{p}+C_{b,\sigma, L, T}\big(t+C_{d,p, L}^{BDG'}\sqrt{t}\big).\nonumber\\
&+C_{b,\sigma, L, T}\left\{\int_{0}^{t}\Big\Vert\sup_{u\in[0,s]}\left|\bar{X}_{u}\right|\Big\Vert_{p}ds+C_{d,p, L}^{BDG'}\Big[\int_{0}^{t}\Big\Vert\sup_{u\in[0,s]}\left|\bar{X}_{u}\right|\Big\Vert_{p}^{2}ds\Big]^{\frac{1}{2}}\right\}.\nonumber
\end{align}
%Then after Lemma~\ref{Gronwall}, 
Hence, 
\[\left\Vert\sup_{u\in[0, t]}\left| \bar{X}_{u}^{M}\right|\right\Vert_{p}\leq 2e^{(2\,C_{b,\sigma, L, T}+C_{d,p, L}^{BDG'^{2}})t}(\left\Vert X_{0}\right\Vert_{p}+C_{b,\sigma, L, T}\big(t+C_{d,p, L}^{BDG}\sqrt{t})\big),\]
by applying  Lemma~\ref{Gronwall}. Thus one concludes the proof by taking \[C_{p,d,b,\sigma}=\big(2\,C_{b,\sigma, L, T}+C_{d,p, L}^{BDG'^{2}}\big)\vee 2\,C_{b,\sigma, L, T}\big(T+C_{d, p, L}^{BDG}\sqrt{T}\big)\vee 2.\] %to conclude the proof. 

%\noindent(b) 
Next, it follows from $\left|X_{t}-X_{s}\right|=\left|\int_{s}^{t}b(u, X_{u}, \mu_{u})du+\int_{s}^{t}\sigma(u, X_{u}, \mu_{u})dB_{u}\right|$ that, 
\begin{align}
&\left\Vert X_{t}-X_{s}\right\Vert_{p}\leq\left\Vert\int_{s}^{t}b(u, X_{u}, \mu_{u})du\right\Vert_{p}+\left\Vert\int_{s}^{t}\sigma(u, X_{u}, \mu_{u})dB_{u}\right\Vert_{p}\nonumber\\
&\quad\leq\int_{s}^{t}\left\Vert b(u, X_{u}, \mu_{u})\right\Vert_{p}du+ C_{d,p}^{BDG}\vertii{\int_{s}^{t}\vertiii{\sigma(u, X_{u}, \mu_{u})}^{2}du}_{\frac{p}{2}}^{\frac{1}{2}}\nonumber\\
&\hspace{1cm}(\text{by Lemma~\ref{gemin} and Lemma~\ref{BDGin}})\nonumber\\
&\quad\leq\int_{s}^{t}C_{b, \sigma, L, T}\Big{[}1+\left\Vert X_{u}\right\Vert_{p}+\vertii{\mathcal{W}_{p}(\mu_{p}, \delta_{0})}_{p}\Big{]}du\nonumber\\
&\quad\hspace{1cm}+ C_{d,p}^{BDG}\vertii{\int_{s}^{t}C_{b, \sigma, L, T}\Big{[}1+\left\Vert X_{u}\right\Vert_{p}+\vertii{\mathcal{W}_{p}(\mu_{p}, \delta_{0})}_{p}\Big{]}^{2}du}_{\frac{p}{2}}^{\frac{1}{2}}\;\;(\text{by (\ref{lineargrowth})})\nonumber\\
&\quad\leq\int_{s}^{t}C_{b, \sigma, L, T}\Big{[}1+2\left\Vert X_{u}\right\Vert_{p}\Big{]}du+4\,C_{d,p}^{BDG}\cdot C_{b, \sigma, L, T}\vertii{\int_{s}^{t}\left[1+\left\Vert X_{u}\right\Vert_{p}^{2}+\mathcal{W}_{p}^{2}(\mu_{p}, \delta_{0})\right]du}_{\frac{p}{2}}^{\frac{1}{2}}\nonumber\\
&\quad\leq\int_{s}^{t}C_{b, \sigma, L, T}\Big{[}1+2\left\Vert X_{u}\right\Vert_{p}\Big{]}du\nonumber\\
&\quad\hspace{1cm}+4\,C_{d,p}^{BDG}\cdot C_{b, \sigma, L, T}\left[(t-s)+ \vertii{\int_{s}^{t}\left| X_{u}\right|^{2}du}_{\frac{p}{2}}+\vertii{\int_{s}^{t}\mathcal{W}_{p}^{2}(\mu_{u}, \delta_{0})du}_{\frac{p}{2}}\right]^{\frac{1}{2}}\nonumber\\
&\quad\leq\int_{s}^{t}C_{b, \sigma, L, T}\Big{[}1+2\left\Vert X_{u}\right\Vert_{p}\Big{]}du\nonumber\\
&\quad\hspace{1cm}+4\,C_{d,p}^{BDG}\cdot C_{b, \sigma, L, T}\left[\sqrt{t-s}+\left[\int_{s}^{t}\left\Vert \left|X_{u}\right|^{2}\right\Vert_{\frac{p}{2}}du\right]^{\frac{1}{2}}+\left[\int_{s}^{t}\left\Vert \mathcal{W}_{p}^{2}(\mu_{u}, \delta_{0})\right\Vert_{\frac{p}{2}}du\right]^{\frac{1}{2}}\right]\nonumber\\
&\quad\leq\int_{s}^{t}C_{b, \sigma, L, T}\left[1+2\,\Big\Vert \sup_{u\in[0,T]}\left|X_{u}\right|\Big\Vert_{p}\right]du\nonumber\\
&\quad\hspace{1cm}+4\,C_{d,p}^{BDG}\cdot C_{b, \sigma, L, T}\left\{ \sqrt{t-s} +\sqrt{\int_{s}^{t}\left\Vert X_{u}\right\Vert_{p}^{2}du\,}+\sqrt{\int_{s}^{t}\vertii{\mathcal{W}_{p}(\mu_{u}, \delta_{0})}_{p}^{2}du}\,\right\}\nonumber\\
&\quad\leq\,C_{b, \sigma, L, T}\left[1+2\,\Big\Vert \sup_{u\in[0,T]}\left|X_{u}\right|\Big\Vert_{p}\right](t-s)\nonumber\\
&\quad\hspace{1cm}+4\,C_{d,p}^{BDG}\cdot C_{b, \sigma, L, T}\left\{ \sqrt{t-s} +2\sqrt{t-s}\,\,{\color{black}\Big\Vert \sup_{u\in[0,T]}\left|X_{u}\right|\Big\Vert_{p}}\,\right\}\nonumber\\
&\quad\leq\Big\{C_{b, \sigma, L, T}\Big[1+2\,\big\Vert \sup_{u\in[0,T]}\left|X_{u}\right|\big\Vert_{p}\Big]\sqrt{T}\nonumber\\
&\quad\hspace{1cm}+4\,C_{d,p}^{BDG}\cdot C_{b, \sigma, L, T}\Big[1+2\,{\color{black}\big\Vert \sup_{u\in[0,T]}\left|X_{u}\right|\big\Vert_{p}}\,\Big]\Big\}\sqrt{t-s}.\nonumber
\end{align}
%\textcolor{black}{pbs de tailles de [ ] ou de parenth\`eses. Pas pu tout corriger}
Owing to the result in $(a)$, $\big\Vert \sup_{u\in[0,T]}\left|X_{u}\right|\big\Vert_{p}\leq C_{p,d,b,\sigma}e^{C_{p,d,b,\sigma}t}\big(1+\left\Vert X_{0}\right\Vert_{p}\big)$, then one can conclude by setting
\begin{align}
\kappa=C_{L, b, \sigma, \left\Vert X_{0}\right\Vert, p, d, T}\coloneqq&\:C_{b, \sigma, L, T}\Big{[}1+2\,C_{p,d,b,\sigma}e^{C_{p,d,b,\sigma}t}(1+\left\Vert X_{0}\right\Vert_{p})\Big{]}\sqrt{T}\nonumber\\
&+ 4\,C_{d,p}^{BDG}\cdot C_{b, \sigma, L, T}\Big[1+{\color{black}2\,C_{p,d,b,\sigma}e^{C_{p,d,b,\sigma}t}\big(1+\left\Vert X_{0}\right\Vert_{p}\big)}\,\Big]\nonumber.\quad\quad\quad\hfill\qedhere
\end{align}
\renewcommand{\qedsymbol}{}
\end{proof}

\begin{proof}[Proof of Proposition~\ref{cvgeuler}-$(b)$]
We write $\bar{X}_{t}$ and $\bar{\mu}_{t}$ instead of $\bar{X}_{t}^{M}$ and $\bar{\mu}_{t}^{M}$ to simplify the notation in this proof. 
For every $s\in[0, T]$, set \[\varepsilon_{s}\coloneqq X_{s}-\bar{X}_{s}=\int_{0}^{s}\big[ b(u, X_{u},\mu_{u})-b(\underline{u}, \bar{X}_{\underline{u}}, \bar{\mu}_{\underline{u}})\big]du+\int_{0}^{s}\big[ \sigma(u, X_{u}, \mu_{u})-\sigma(\underline{u}, \bar{X}_{\underline{u}}, \bar{\mu}_{\underline{u}})\big]dB_{u},\]
%then 
%\[\sup_{s\in[0,t]}\left|\varepsilon_{s}\right|\leq \int_{0}^{t}\left|b(X_{s}, \mu_{s})-b(\bar{X}_{\underline{s}}, \bar{\mu}_{\underline{s}})\right|ds+\sup_{s\in[0,t]}\left|\int_{0}^{s}\big(\sigma(X_{u}, \mu_{u})-\sigma(\bar{X}_{\underline{u}}, \bar{\mu}_{\underline{u}})\big)dB_{u}\right|\]
and let \[f(t)\coloneqq \left\Vert\sup_{s\in[0,t]}\left|\varepsilon_{s}\right|\right\Vert_{p}=\left\Vert\sup_{s\in[0,t]}\left|X_{s}-\bar{X}_{s}\right|\right\Vert_{p}.\]

It follows from Proposition~\ref{cvgeuler}-$(a)$ that $\bar{X}=(\bar{X}_{t})_{t\in[0, T]}\in L_{\mathcal{C}([0, T], \mathbb{R}^{d})}^{p}(\Omega, \mathcal{F}, \mathbb{P})$. Consequently, $\bar{\mu}\in\PPC$ and $\iota(\mu)=(\mu_{t})_{t\in[0, T]}\in\CPP$ by applying Lemma~\ref{injectionmeasure}.
%Lemma~\ref{eulertheorical}-(a) implies that $(\bar{X}_{t})_{t\in[0, T]}\in L_{\mathcal{C}([0, T], \mathbb{R}^{d})}^{p}(\Omega, \mathcal{F}, \mathbb{P})$ and consequently, $(\bar{\mu}_{t})_{t\in[0, T]}\in\mathcal{P}_{p}\big(\mathcal{C}([0, T], \mathbb{R}^{d})\big)\subset \mathcal{C}\big([0, T], \mathcal{P}_{p}(\mathbb{R}^{d})\big)$. 
Hence, 
\begin{align}&f(t)=\left\Vert\sup_{s\in[0,t]}\left|X_{s}-\bar{X}_{s}\right|\right\Vert_{p}\nonumber\\
& \leq \vertii{\int_{0}^{t}\left|b(s, X_{s}, \mu_{s})-b(\underline{s}, \bar{X}_{\underline{s}}, \bar{\mu}_{\underline{s}})\right|ds+\sup_{s\in[0,t]}\left|\int_{0}^{s}\big(\sigma(u, X_{u}, \mu_{u})-\sigma(\underline{u}, \bar{X}_{\underline{u}}, \bar{\mu}_{\underline{u}})\big)dB_{u}\right|}_{p}\nonumber\\
&\leq\int_{0}^{t}\vertii{b(s, X_{s}, \mu_{s})-b(\underline{s},\bar{X}_{\underline{s}}, \bar{\mu}_{\underline{s}})}_{p}ds+C_{d,p}^{BDG}\vertii{\sqrt{\int_{0}^{t}\vertiii{\sigma(s, X_{s}, \mu_{s})-\sigma(\underline{s}, \bar{X}_{\underline{s}}, \bar{\mu}_{\underline{s}})}^{2}ds}}_{p}\nonumber\\
&=\int_{0}^{t}\vertii{b(s, X_{s}, \mu_{s})-b(\underline{s}, \bar{X}_{\underline{s}}, \bar{\mu}_{\underline{s}})}_{p}ds+C_{d,p}^{BDG}\vertii{\int_{0}^{t}\vertiii{\sigma(s, X_{s}, \mu_{s})-\sigma(\underline{s}, \bar{X}_{\underline{s}}, \bar{\mu}_{\underline{s}})}^{2}ds}_{\frac{p}{2}}^{\frac{1}{2}}\nonumber\\
&\leq \int_{0}^{t}\vertii{b(s, X_{s}, \mu_{s})-b(\underline{s}, \bar{X}_{\underline{s}}, \bar{\mu}_{\underline{s}})}_{p}ds+C_{d,p}^{BDG}\left[\int_{0}^{t}\vertii{\vertiii{\sigma(s, X_{s}, \mu_{s})-\sigma(\underline{s}, \bar{X}_{\underline{s}}, \bar{\mu}_{\underline{s}})}^{2}}_{\frac{p}{2}}ds\right]^{\frac{1}{2}}\nonumber\\
&= \int_{0}^{t}\vertii{b(s, X_{s}, \mu_{s})-b(\underline{s}, \bar{X}_{\underline{s}}, \bar{\mu}_{\underline{s}})}_{p}ds+C_{d,p}^{BDG}\left[\int_{0}^{t}\big{\Vert}\vertiii{\sigma(s, X_{s}, \mu_{s})-\sigma(\underline{s}, \bar{X}_{\underline{s}}, \bar{\mu}_{\underline{s}})}\big{\Vert}_{p}^{2}ds\right]^{\frac{1}{2}}\nonumber\\
&\leq \int_{0}^{t}\vertii{b(s, X_{s}, \mu_{s})-b(\underline{s}, X_{s}, \mu_{s})}_{p}ds+\int_{0}^{t}\vertii{b(\underline{s}, X_{s}, \mu_{s})-b(\underline{s}, \bar{X}_{\underline{s}}, \bar{\mu}_{\underline{s}})}_{p}ds\nonumber\\
&\label{big1}
+C_{d,p}^{BDG}\left[\int_{0}^{t}\big{\Vert}\vertiii{\sigma(s, X_{s}, \mu_{s})-\sigma(\underline{s}, X_{s}, \mu_{s})}+\vertiii{\sigma(\underline{s}, X_{s}, \mu_{s})-\sigma(\underline{s}, \bar{X}_{\underline{s}}, \bar{\mu}_{\underline{s}})}\big{\Vert}_{p}^{2}ds\right]^{\frac{1}{2}}, 
\end{align}
where the last term of~(\ref{big1}) can be upper-bounded by 
\begin{align}\label{big2}
&C_{d,p}^{BDG}\left[\int_{0}^{t}\big{\Vert}\vertiii{\sigma(s, X_{s}, \mu_{s})-\sigma(\underline{s}, X_{s}, \mu_{s})}+\vertiii{\sigma(\underline{s}, X_{s}, \mu_{s})-\sigma(\underline{s}, \bar{X}_{\underline{s}}, \bar{\mu}_{\underline{s}})}\big{\Vert}_{p}^{2}ds\right]^{\frac{1}{2}}\nonumber\\
&\quad\leq C_{d,p}^{BDG}\left[\int_{0}^{t}\big[\big{\Vert}\vertiii{\sigma(s, X_{s}, \mu_{s})-\sigma(\underline{s}, X_{s}, \mu_{s})}\big{\Vert}_{p}+\big{\Vert}\vertiii{\sigma(\underline{s}, X_{s}, \mu_{s})-\sigma(\underline{s}, \bar{X}_{\underline{s}}, \bar{\mu}_{\underline{s}})}\big{\Vert}_{p}\big]^{2}ds\right]^{\frac{1}{2}}\nonumber\\
&\quad\leq \sqrt{2}C_{d,p}^{BDG}\left[\int_{0}^{t}\big{\Vert}\vertiii{\sigma(s, X_{s}, \mu_{s})-\sigma(\underline{s}, X_{s}, \mu_{s})}\big{\Vert}_{p}^{2}ds\right]^{\frac{1}{2}}\nonumber\\
&\hspace{1cm}+\sqrt{2}C_{d,p}^{BDG}\left[\int_{0}^{t}\big{\Vert}\vertiii{\sigma(\underline{s}, X_{s}, \mu_{s})-\sigma(\underline{s}, \bar{X}_{\underline{s}}, \bar{\mu}_{\underline{s}})}\big{\Vert}_{p}^{2}ds\right]^{\frac{1}{2}}. 
\end{align}
It follows that
\begin{align}\label{bigpart1}
\int_{0}^{t}&\vertii{b(s, X_{s}, \mu_{s})-b(\underline{s}, X_{s}, \mu_{s})}_{p}ds+\sqrt{2}C_{d,p}^{BDG}\left[\int_{0}^{t}\big{\Vert}\vertiii{\sigma(s, X_{s}, \mu_{s})-\sigma(\underline{s}, X_{s}, \mu_{s})}\big{\Vert}_{p}^{2}ds\right]^{\frac{1}{2}}\nonumber\\
&\leq \int_{0}^{t}\vertii{(s-\underline{s})^{\rho}\tilde{L}\big(1+\left|X_{s}\right|+\mathcal{W}_{p}(\mu_{s}, \delta_{0})\big)}_{p}ds\nonumber\\
&\hspace{1cm}+\sqrt{2}C_{d,p}^{BDG}\left[\int_{0}^{t}\big{\Vert}(s-\underline{s})^{\rho}\tilde{L}\big(1+\left|X_{s}\right|+\mathcal{W}_{p}(\mu_{s}, \delta_{0})\big)\big{\Vert}_{p}^{2}ds\right]^{\frac{1}{2}}\;\;\;\text{(by Assumption~\ref{AssumptionI})}\nonumber\\
&\leq  h^{\rho}T\tilde{L}(1+2\,\big{\Vert}\sup_{s\in[0, T]}\left|X_{s}\right|\big{\Vert}_{p})+\sqrt{2} h^{\rho}\tilde{L}C_{d,p}^{BDG}\left[T(2+4\,\big{\Vert}\sup_{s\in[0, T]}\left|X_{s}\right|\big{\Vert}_{p}^{2})\right]^{\frac{1}{2}}\nonumber\\
&\leq  h^{\rho}T\tilde{L}(1+2\,\big{\Vert}\sup_{s\in[0, T]}\left|X_{s}\right|\big{\Vert}_{p})+\sqrt{2} h^{\rho}\tilde{L}C_{d,p}^{BDG}\left[\sqrt{2T}+2\sqrt{T}\big{\Vert}\sup_{s\in[0, T]}\left|X_{s}\right|\big{\Vert}_{p}\right]
\end{align}
%\[\left\Vert \sup_{u\in[0, t]}\left| X_{u}\right|\right\Vert_{p}\vee\vertii{\sup_{u\in[0,t]}\left|\bar{X}_{u}\right|}_{p}\leq C_{p,d,b,\sigma}e^{C_{p,d,b,\sigma}t}(1+\left\Vert X_{0}\right\Vert_{p}).\]
and
\begin{align}
\int_{0}^{t}\big\Vert b(\underline{s}&, X_{s}, \mu_{s})-b(\underline{s}, \bar{X}_{\underline{s}}, \bar{\mu}_{\underline{s}})\big\Vert_{p}ds+\sqrt{2}C_{d,p}^{BDG}\left[\int_{0}^{t}\big{\Vert}\vertiii{\sigma(\underline{s}, X_{s}, \mu_{s})-\sigma(\underline{s}, \bar{X}_{\underline{s}}, \bar{\mu}_{\underline{s}})}\big{\Vert}_{p}^{2}ds\right]^{\frac{1}{2}}\nonumber\\
&\leq \int_{0}^{t}\vertii{ L\big(\left|X_{s}-\bar{X}_{\underline{s}}\,\right|+\mathcal{W}_{p}(\mu_{s},\bar{\mu}_{\underline{s}})\big)}_{p}ds\nonumber\\
&\hspace{1.5cm}+\sqrt{2}C_{d,p}^{BDG}\left[\int_{0}^{t}\vertii{L\big(\left|X_{s}-\bar{X}_{\underline{s}}\,\right|+\mathcal{W}_{p}(\mu_{s},\bar{\mu}_{\underline{s}})\big)}_{p}^{2}ds\right]^{\frac{1}{2}}\nonumber\\
&\leq\int_{0}^{t}2L\vertii{X_{s}-\bar{X}_{\underline{s}}}_{p}ds+\sqrt{2}C_{d,p}^{BDG}\left[\int_{0}^{t}4L^{2}\left\Vert X_{s}-\bar{X}_{\underline{s}}\,\right\Vert_{p}^{2} ds\right]^{\frac{1}{2}}\nonumber\\
&\leq\int_{0}^{t}2L\left[%\underset{\color{black}{\leq \kappa \sqrt{s-\underline{s}}\leq \kappa\sqrt{ h}}}{\underbrace{
\vertii{X_{s}-X_{\underline{s}}}_{p}
%}}
+\vertii{X_{\underline{s}}-\bar{X}_{\underline{s}}}_{p}\right]ds\nonumber\\
&\hspace{1.5cm}+\sqrt{2}C_{d,p}^{BDG}\left[\int_{0}^{t}4L^{2}\left[%\underset{\color{black}{\leq \kappa \sqrt{s-\underline{s}}\leq \kappa\sqrt{ h}}}{\underbrace
\vertii{X_{s}-X_{\underline{s}}}_{p}+\vertii{X_{\underline{s}}-\bar{X}_{\underline{s}}}_{p}\right]^{2} ds\right]^{\frac{1}{2}}\nonumber\\
&\leq\int_{0}^{t}2L\left[\kappa\sqrt{ h}+\vertii{X_{\underline{s}}-\bar{X}_{\underline{s}}}_{p}\right]ds+\sqrt{2}C_{d,p}^{BDG}\left[\int_{0}^{t}4L^{2}\left[\kappa\sqrt{ h}+\vertii{X_{\underline{s}}-\bar{X}_{\underline{s}}}_{p}\right]^{2} ds\right]^{\frac{1}{2}}\nonumber\\
&\quad\text{(by applying Proposition~\ref{cvgeuler}-$(a)$)}\nonumber\\
%&\leq\underset{\color{black}{\eqqcolon \psi(t)\leq \psi(T)}}{\underbrace{2Lt\kappa\sqrt{ h}+C_{d,p}^{BDG}2\sqrt{2}L\sqrt{t}\kappa}}\sqrt{ h}+\underset{\color{black}{\eqqcolon A}}{\underbrace{2L}}\int_{0}^{t}f(s)ds+\underset{\color{black}{\eqqcolon B}}{\underbrace{C_{d,p}^{BDG}4L}}\left[\int_{0}^{t}f(s)^{2}ds\right]^{\frac{1}{2}}
&\label{bigpart2}~\leq \!2Lt\kappa\sqrt{ h}\!+\!4\,C_{d,p}^{BDG}L\sqrt{t}\kappa\sqrt{ h}\!+\!2L\int_{0}^{t}f(s)ds\!+\!\!\sqrt{2}C_{d,p}^{BDG}4L\!\!\left[\int_{0}^{t}f(s)^{2}ds\right]^{\frac{1}{2}}\!\!.
\end{align}
Let $\tilde{\kappa}(T, \vertii{X_{0}}_{p})=C_{p,d,b,\sigma}e^{C_{p,d,b,\sigma}t}(1+\left\Vert X_{0}\right\Vert_{p})$, which is the right hand side of results in Proposition~\ref{cvgeuler}-$(a)$. A combination of~(\ref{big1}),~(\ref{big2}),~(\ref{bigpart1}) and~(\ref{bigpart2}) leads to 
\begin{align}
f(t)&=\left\Vert\sup_{s\in[0,t]}\left|X_{s}-\bar{X}_{s}\right|\right\Vert_{p}\nonumber\\
&\leq  h^{\rho}T\tilde{L}(1+2\,\big{\Vert}\sup_{s\in[0, T]}\left|X_{s}\right|\big{\Vert}_{p})+\sqrt{2} h^{\rho}\tilde{L}C_{d,p}^{BDG}\left[\sqrt{2T}+2\sqrt{T}\big{\Vert}\sup_{s\in[0, T]}\left|X_{s}\right|\big{\Vert}_{p}\right]\nonumber\\
&\hspace{0.5cm}+2Lt\kappa\sqrt{ h}+\sqrt{2}C_{d,p}^{BDG}2\sqrt{2}L\sqrt{t}\kappa\sqrt{ h}+2L\int_{0}^{t}f(s)ds+\sqrt{2}C_{d,p}^{BDG}4L\left[\int_{0}^{t}f(s)^{2}ds\right]^{\frac{1}{2}}.\nonumber\\
&\leq  h^{\frac{1}{2}\land\rho}\psi(T)+2L\int_{0}^{t}f(s)ds+\sqrt{2}C_{d,p}^{BDG}4L\left[\int_{0}^{t}f(s)^{2}ds\right]^{\frac{1}{2}},\nonumber
\end{align}
where 
\begin{align}
\psi(T)=&\,T^{\rho-\rho\land\frac{1}{2}}\left[T\tilde{L}\big(1+2\tilde{\kappa}(T, \vertii{X_{0}}_{p})\big)+\sqrt{2}\tilde{L}C_{d, p}^{BDG}\big(\sqrt{2T}+2\sqrt{T}\tilde{\kappa}(T, \vertii{X_{0}}_{p})\big)\right]\nonumber\\
&+T^{\frac{1}{2}-\rho\land\frac{1}{2}}\left[2LT\kappa+4C_{d, p}^{BDG}L \sqrt{T}\kappa\right]. \nonumber
\end{align}
%Let $\psi(t)\coloneqq2Lt\kappa\sqrt{ h}+C_{d,p}^{BDG}2\sqrt{2}L\sqrt{t}\kappa$. 
Then it follows from lemma~\ref{Gronwall} that $f(t)\leq 2e^{(4L+16C_{d, p}^{BDG^{2}}L^{2})T}\cdot\psi(T) h^{\rho\land\frac{1}{2}}$. Then we can conclude the proof by letting  %$\varphi(T)
$\tilde{C}= 2e^{(4L+16C_{d, p}^{BDG^{2}}L^{2})T}\cdot\psi(T)$. % to conclude the proof. 
\end{proof}

The proof of Proposition~\ref{cvgeuler}-$(b)$ directly derives the following result.
\begin{cor}\label{coleulercontinuous}
Let $\bar{X}\coloneqq(\bar{X}_{t})_{t\in[0, T]}$ denote the process defined by the continuous time Euler scheme~(\ref{defeulercontinuousx}) with step $h=\frac{T}{M}$ and let $X\coloneqq(X_{t})_{t\in[0, T]}$ denote the unique solution of the McKean-Vlasov equation $~(\ref{defconvx})$. Then under Assumption~\ref{AssumptionI}, one has 
\begin{equation}
\mathbb{W}_{p}(\bar{X}, X)\leq \vertii{\sup_{t\in[0, T]}\left|X_{t}-\bar{X}_{t}\right|}_{p}\leq \tilde{C} h^{\frac{1}{2}\land \rho}, 
\end{equation}
where $\tilde{C}$ is the same as in Proposition~\ref{cvgeuler}-$(a)$.
\end{cor}

%%%%%%%%%%
%%%%%%%%%%
%%%%%%%%%%
\smallskip

\section{Proof of Proposition~\ref{pp}}\label{proofpp}

\begin{lem}\label{lem1}Let $\mu$, $\nu\!  \in {\cal P}_1(\RD)$. We have $\mu \conright \nu$ if and only if the application $t\in[0, 1]\mapsto (1-t)\mu +t \nu$ is non-decreasing w.r.t. the convex order.

\end{lem}

\begin{proof}[Proof of Lemma~\ref{lem1}]
 Let $f:\RD\to \RR$ be convex function with linear growth and let $t\!\in [0,1]$. Then if $\mu \conright \nu$, 
\[
\int_{\RD}f(\xi)\mu(d\xi) \le (1-t)\int_{\RD}f(\xi)\mu(d\xi) +t\int_{\RD}f(\xi)\nu(d\xi)\le \int_{\RD}f(\xi)\nu(d\xi).
\]
Then one can conclude that $\mu \conright (1-t) \mu + t \nu \conright \nu$ for every $t\in[0, 1]$. 
Now let $0\leq s <t$. It follows from what precedes that 
%\begin{align*}
\[(1-s)\mu +s \nu  = (1-\tfrac st)\mu + \tfrac st \big( (1-t) \mu +t \nu\big) \conright (1-t)\mu +t \nu.\]
The proof for the converse direction is trivial.
\end{proof}

\begin{proof}[Proof of Proposition~\ref{pp}]
 $(a)$ The direct sense follows from Lemma~\ref{lem1} since $\mu \le_{cv}(1-\varepsilon)\mu +\varepsilon \nu$ for $\varepsilon \!\in [0,1]$.

\noindent For the converse, we proceed as follows. Let $\mu \le_{cv} \nu$. , note that $\varepsilon \mapsto \Phi\big(\mu+\varepsilon(\nu-\mu)\big)$ is continuous since $\varepsilon \mapsto \mu+\varepsilon(\nu-\mu)$ from  $[0,1]$ to $\big({\cal P}_2(\RD), {\cal W}_2\big)$ is continuous. Consequently $\widetilde \Phi(\varepsilon):= 
\Phi\big(\mu+\varepsilon(\nu-\mu)\big)-\Phi(\mu) -\varepsilon \big(\Phi(\nu)-\Phi(\mu)\big)$ is also continuous and satisfies $\widetilde \Phi(0)= \widetilde \Phi(1)= 0$. Hence,  $\widetilde \Phi$ attains its maximum at some $\varepsilon_0\!\in [0,1)$. This in turn implies that 
\begin{align*}
0\ge  \liminf_{\varepsilon \to \varepsilon_0+}& \frac{\widetilde \Phi(\varepsilon)-\widetilde \Phi(\varepsilon_0)}{\varepsilon-\varepsilon_0}\\
& =  \liminf_{\eta \to 0+}\frac{\Phi(\mu+\varepsilon_0(\nu-\mu)+\eta(\nu-\mu))-\Phi(\mu+\varepsilon_0(\nu-\mu))}{\eta}-  \big(\Phi(\nu)-\Phi(\mu)\big).
\end{align*}
Then set $\widetilde \mu = \mu +\varepsilon_0(\nu-\mu)$ and $\widetilde \nu = \nu$. One has $\widetilde \mu\le_{cv} \widetilde \nu$ owing to Lemma~\ref{lem1} and note that $\widetilde \nu - \widetilde \mu=(1-\varepsilon_0)(\nu-\mu)$. Consequently, $ \mu +\varepsilon_0( \nu- \mu)+\eta (\nu-\mu) =\widetilde{\mu} +\frac{\eta}{1-\varepsilon_0}(\widetilde \nu-\widetilde \mu)$. Applying the assumption made on $\Phi$ to $\widetilde \mu$ and $\widetilde \nu$ implies 
$$
 \liminf_{\eta \to 0+}\frac{\Phi(\mu+\varepsilon_0(\nu-\mu)+\eta(\nu-\mu))-\Phi(\mu+\varepsilon_0(\nu-\mu))}{\eta}\ge 0.
$$ 
Finally, this yields $\Phi(\nu)\ge \Phi(\mu)$.

%\smallskip
\noindent \textcolor{black}{$(b)$ }
Let  $\mu$, $\nu\!\in {\cal P}_2(\RD)$ such that $\mu\preceq_{cv} \nu$. Then, for every $\ve\!\in [0,1)$, $\mu \preceq_{cv} (1-\ve)\mu +\ve \nu$ so that, as $\Phi$ is linearly functionally differentiable, 
 \begin{align*}
\frac{\Phi( (1-\ve)\mu +\ve \nu)-\Phi(\mu)}{\ve} & = \int_0^1 \int_{\RD} \frac{\delta \Phi}{\delta m}\big((1-t)\mu+t((1-\ve)\mu +\ve\nu)\big)(x) d[\nu-\mu](x)dt\\
& =  \int_0^1 \int_{\RD} \frac{\delta \Phi}{\delta m}\big((1-\ve t)\mu+\ve t\nu\big)(x)  d[\nu-\mu](x)dt.
\end{align*}
The function  $u\mapsto (1-u)\mu+u\nu$, $u\!\in [0,1]$, being ${\cal W}_2$-continuous, it follows from the definition of the linear functional derivative, that
$$ 
\frac{\delta \Phi}{\delta m}\big((1-u)\mu+u\nu\big)(x)\to  \frac{\delta \Phi}{\delta m}(\mu )(x)
\quad \mbox{as}\quad u\to 0\quad\mbox{ for every }\quad x\!\in \RD.
$$
 Moreover, ${\cal K}= \{(1-u)\mu+u\nu, \, u\!\in [0,1]\}$ being  clearly ${\cal W}_2$-bounded, one has, still by this definition,
\[
\forall\, x\!\in \RD, \quad\sup_{u\in [0,1]} \left |\frac{\delta \Phi}{\delta m}((1-u)\mu+u\nu)(x)  \right|\le C_{{\cal K}}(1+|x|^2)
\]
for some real constant $C_{{\cal K}}$. Consequently, Lebesgue's dominated convergence applied with both $\mu\otimes dt$ and $\nu \otimes dt$ implies
\[
\int_0^1dt  \int_{\RD}  \frac{\delta \Phi}{\delta m}((1-t\ve)\mu+t\ve\nu)(x) d[\nu-\mu](x) \to  \int_{\RD} \frac{\delta \Phi}{\delta m}(\mu)(x) d[\nu-\mu](x) \quad\mbox{as}\quad \ve \to 0.
\]
This shows that  the function $\varepsilon \in[0, 1]\mapsto \Phi \big(\mu +\varepsilon (\nu-\mu)\big)$ is  right differentiable at $\ve =0$ with 
\[
\frac{d}{d\varepsilon}[\Phi \big(\mu +\varepsilon (\nu-\mu)\big)]_{|\ve =0}= \int_{\RD} \frac{\delta \Phi}{\delta m}(\mu)(x) d[\nu-\mu](x).
\]
%by applying \cite[Proposition 5.44 and Remark 5.47]{MR3752669}. 
Then the equivalence between $(i)$ and $(ii)$ is a direct consequence  of the characterization $(a)$.

\smallskip
\noindent $(c)$ This claim is a straightforward consequence of $(b)$.
\end{proof}

\section{Proof of Lemma~\ref{revisiedjensen} and Lemma~\ref{marglemma}}\label{appB}

The proof of Lemma~\ref{revisiedjensen} relies on the following lemma. 

\begin{lem}\label{orderZ} $($see~\cite[
Lemma 3.2]{jourdain2019convex} and~\cite{fadili2019convex} $)$
Let $Z\sim \mathcal{N}(0, \mathbf{I}_{q})$. If $u_{1}, u_{2}\in\mathbb{M}_{d\times q}$ with $u_{1}\preceq u_{2}$, then $u_{1}Z\preceq_{\,cv}u_{2}Z$.
\end{lem}
\begin{proof}[Proof of Lemma~\ref{orderZ}](\footnote{This proof is reproduced from \cite{jourdain2019convex} for convenience.})
We define $M_{1}\coloneqq u_{1}Z$ and $M_{2}\coloneqq M_{1}+\sqrt{u_{2}u_{2}^{*}-u_{1}u_{1}^{*}}\,\widetilde{Z}$, where $\sqrt{A}$ denotes the square root of a positive semi-definite matrix $A$ and $\widetilde{Z}\sim \mathcal{N}(0, \mathbf{I}_{d})$,  $\widetilde{Z}$ is independent to $Z$.   Hence the probability distribution of $M_{2}$ is $\mathcal{N}(0, u_{2}u_{2}^{*})$, which is the  distribution of $u_{2}Z$. 

For any convex function $\varphi$, we have, owing to the conditional Jensen inequality, \begin{align}
\EE \big[\varphi(M_{2})\big] &=\EE \big[\varphi\big(M_{1}+\sqrt{u_{2}u_{2}^{*}-u_{1}u_{1}^{*}}\cdot\widetilde{Z}\big)\big] \nonumber\\
&=\EE \Big[ \;\EE \big[\varphi\big(M_{1}+\sqrt{u_{2}u_{2}^{*}-u_{1}u_{1}^{*}}\cdot\widetilde{Z}\big)\mid Z\big] \;\Big]\nonumber\\
&\geq \EE\Big[ \varphi\big(\EE \big[M_{1}+\sqrt{u_{2}u_{2}^{*}-u_{1}u_{1}^{*}}\cdot\widetilde{Z}\mid Z\big]\big)\;\Big]\nonumber\\
&=\EE \Big[ \varphi\big(M_{1}+\mathbb{E}\big[\sqrt{u_{2}u_{2}^{*}-u_{1}u_{1}^{*}}\cdot\widetilde{Z}\big]\big)\Big]=\EE \varphi(M_{1}). 
\end{align}
Hence, $u_{1}Z\preceq_{\,cv}u_{2}Z$ owing to the equivalence of convex order of the random variable and its probability distribution.
\end{proof}
\begin{proof}[Proof of Lemma~\ref{revisiedjensen}]
%We firstly prove that $Q\varphi$ is convex. 
\noindent $(i)$ Let $(Q_{m}^{\,\mu}\varphi)(\cdot, \cdot)\coloneqq(Q_{m}\varphi )(\cdot, \mu, \cdot)$ to simplify the notation. For every $(x_{1}, u_{1}), (x_{2}, u_{2})\!\in \RD\times\MDQ$ and $\lambda\in[0, 1]$,
\begin{align}
(Q_{m}^{\,\mu}\varphi)&\big(\lambda (x_{1}, u_{1})+(1-\lambda) (x_{2}, u_{2})\big)\nonumber\\
&=\EE \Big[\varphi \Big( b_{m-1}\big(\lambda_1 x_1 + (1-\lambda)x_2, \mu\big) +\big(\lambda u_{1}+(1-\lambda) u_{2}\big)Z_{m}\Big)\Big] \nonumber\\
&=\EE \Big[\varphi\Big(\lambda b_{m-1}(x_1, \mu)+(1-\lambda)b_{m-1}(x_2, \mu)+\lambda u_{1}Z_{m}+(1-\lambda)u_{2}Z_{m}\Big)\Big] \nonumber\\
&\quad\;\text{(as $b_m$ is affine in $x$, see Assumption~\ref{IIprime}-(5))}\nonumber\\
&\le\lambda \EE\big[\varphi \big(b_{m-1}(x_1, \mu)+u_{1}Z_{m}\big)\big]+ (1-\lambda)\EE \big[\varphi\big(b_{m-1}(x_1, \mu)+u_{2}Z_{m}\big)\big]\nonumber\\
&\hspace{0.5cm} \text{(by the convexity of $\varphi$ and the linearity of the expectation)}\nonumber\\
&=\lambda (Q_{m}^{\,\mu}\varphi)(x_{1}, u_{1}) + (1-\lambda) (Q_{m}^{\,\mu}\varphi)(x_{2}, u_{2})\nonumber. 
\end{align}
Hence, $(Q_{m}\varphi )(\cdot, \mu, \cdot)$ is convex. %The proof for the $r$-polynomial growth of $Q\varphi$ is obvious.

\noindent$(ii)$ If we fix $(x, \mu)\in\RD\times\mathcal{P}_{1}(\RD)$, then for any $u\in\MDQ$,
\begin{align}
(Q_m\varphi) (x, \mu, u) &= \mathbb{E}\Big[\varphi\big(b_{m-1}(x, \mu)+uZ_{m}\big)\Big]\nonumber\\
&\geq \varphi\Big(\mathbb{E}\big[b_{m-1}(x, \mu)+uZ_{m}\big]\Big)\nonumber\\
&=\varphi\big(b_{m-1}(x, \mu)+ \mathbf{0}_{d\times1}\big)=(Q_m\varphi)(x, \mu, \mathbf{0}_{d\times q}). \nonumber
\end{align}

\noindent$(iii)$ For every fixed $(x, \mu)\in\RD\times\mathcal{P}_{1}(\RD)$, it is obvious that $\varphi\big(b_{m-1}(x, \mu)+\;\cdot\;\big)$ %\varphi(\bar{\alpha}x+\bar{\beta}+\cdot)$ 
is also a convex function. Thus, Lemma~\ref{orderZ} directly implies that if $u_{1}\preceq u_{2}$, then \[\EE\varphi\big(b_{m-1}(x, \mu)+u_{1}Z_m)\leq \EE \varphi\big(b_{m-1}(x, \mu)+u_{2}Z_m),\] which is equivalent to $Q_m\varphi(x, \mu, u_{1})\leq Q_m\varphi(x, \mu, u_{2})$. 
\end{proof}

\begin{proof}[Proof of Lemma~\ref{marglemma}]
Let $x, y \in\mathbb{R}^{d}$ and $\lambda\in[0, 1]$. For every $m=0, \ldots, M-1$, we have
\begin{align}
\EE &\Big[\varphi\Big(b_{m}\big(\lambda x+(1-\lambda)y, \mu\big)+\sigma_{m}\big(\lambda x+(1-\lambda)y, \mu\big) Z_{m+1}\Big)\Big]\nonumber\\
&\leq \EE \Big[\varphi\Big(\lambda\, b_m(x, \mu)+(1-\lambda)b_m(y, \mu)+ \lambda\,\sigma_{m}(x, \mu)Z_{m+1}+(1-\lambda)\,\sigma_{m}(y,\mu)Z_{m+1}\Big)\Big]\nonumber\\
&\hspace{0.5cm}\text{(by Assumption~(\ref{conm1}), (\ref{assmbm}) and Lemma~\ref{revisiedjensen})}  \nonumber\\
&\leq \lambda\,\EE \Big[\varphi\big(b_{m}(x, \mu)+\sigma_{m}(x, \mu)Z_{m+1}\big)\Big]+(1-\lambda)\,\EE \Big[\varphi\big(b_{m}(y, \mu)+\sigma_{m}(y, \mu)Z_{m+1}\big)\Big]\nonumber\\
& \hspace{0.5cm}\text{(by the convexity of $\varphi$ and the linearity of the expectation)}. \nonumber%\hfill\qedhere
\end{align}
The function $x\mapsto \mathbb{E}\Big[\varphi\big(b_{m}(x, \mu)+\sigma_{m}(x, \mu)  Z_{m+1}\big)\Big]$ %$x\mapsto \mathbb{E}\big[\varphi\big(\bar{\alpha}x+\bar{\beta}+\sigma_{m}(x, \mu)  Z_{m}\big)\big]$ 
obviously has a linear growth since Assumption~\ref{AssumptionI} implies that $b_m$ and $\sigma_{m}$ have a linear growth (see  (\ref{lineargrowth})).
\end{proof}

\section{Proof of Lemma~\ref{injectionmeasure}}\label{appc}
\begin{proof}[Proof of Lemma~\ref{injectionmeasure}]
%\begin{proof}[Proof of Lemma~\ref{injectionmeasure}]
%\textit{$\iota$ is well-defined.} 
%We prove firstly that $\iota$ is well-defined.
For any $\mu\in\PPC$, there exists $X: (\Omega, \mathcal{F}, \mathbb{P})\rightarrow\CRD$ such that $\PP_{X}=\mu$ and $\mathbb{E}\left\Vert X\right\Vert_{\sup}^{p}<+\infty$ so that $\sup_{t\in[0, T]}\mathbb{E}\left|X_{t}\right|^{p}<+\infty$. Hence, for any $t\in[0, T]$, we have $\mu_{t}\in\PPRD$. 
%We will firstly prove that for any $X\in \mathcal{H}_{p, C, T}$, then $P_{X}\in \CPP$. 
%For any $t\in[0, T]$, let $\mu_{t}$ denote the probability distribution of $X_{t}$. 

For a fixed $t\in[0,T]$, choose $(t_{n})_{n\in\mathbb{N}^{*}}\in[0,T]^{\mathbb{N}^{*}}$ such that $t_{n}\rightarrow t$. Then, for $\PP$-almost any $\omega\in\Omega$, $X_{t_{n}}(\omega)\rightarrow X_{t}(\omega)$ since  $X(\omega)$ has $\PP$-a.s. continuous paths. Moreover, %If $X\in \mathcal{H}_{p, C, T}$, then $\left\Vert X\right\Vert_{p, C, T}=\sup_{t\in[0,T]}e^{-Ct}\left\Vert\sup_{0\leq s\leq t}\left|X_{s}\right|\right\Vert_{p}<+\infty$
\[\sup_{n}\left\Vert X_{t_{n}}\right\Vert_{p}\vee\left\Vert X_{t}\right\Vert_{p}\leq\left\Vert\sup_{0\leq s\leq T}\left| X_{s}\right|\right\Vert_{p}<+\infty,\]
Hence, $\left\Vert X_{t_{n}}-X_{t}\right\Vert_{p}\rightarrow 0$ owing to the dominated convergence theorem, which implies that  $\mathcal{W}_{p}(\mu_{t_{n}}, \mu_{t})\rightarrow0$ as $n\rightarrow+\infty$. Hence, $t\mapsto\mu_{t}$ is a continuous application i.e. $\iota(\mu)=(\mu_{t})_{t\in[0, T]}\in\mathcal{C}\big([0,T],\mathcal{P}_{p}(\mathbb{R}^{d})\big).$
%Next, we will prove that $\iota$ is an injection. 
%For any $\mu, \nu\in\PPC$, it is obvious that $d_{\mathcal{C}}\big(\iota(\mu), \iota(\nu)\big)=0$ implies that $\mu=\nu$. Thus $\iota$ is an injection. 
\end{proof}

%\%\%\%\%\%\%\%\%\%\%

%\%\%\%\%\%\%\%\%\%\%

%\%\%\%\%\%\%\%\%\%\%

\end{appendices}

\end{document}